\documentclass[a4paper,onecolumn,nopdfoutputerror,superscriptaddress,10pt,accepted=2023-01-17,issue=4,volume=5, shorttitle=papers/compositionality-5-4]{compositionalityarticle}
\usepackage[utf8]{inputenc}
\usepackage[english]{babel}
\usepackage[T1]{fontenc}
\usepackage{amsmath}
\usepackage{hyperref}

\usepackage[all]{xy} 
\usepackage[usenames,dvipsnames]{pstricks}
\usepackage{amssymb,mathrsfs,euscript,enumitem,amsmath,bbold,amsthm}

\newtheorem{theorem}{Theorem}[section]
\newtheorem{proposition}[theorem]{Proposition}
\newtheorem{lemma}[theorem]{Lemma}

\newtheorem{corollary}[theorem]{Corollary}
\newtheorem{claim}[theorem]{Claim}

\newtheorem*{blow up}{Blow-up axiom - \BU}
\newtheorem*{Rigid}{Rigidity - \Rig}
\newtheorem*{SG}{Strict grading - \SGrad}
\newtheorem*{Facto}{Factorizability - \Fac}
\newtheorem*{QIB}{Invertibility of quasibijections - \QBI}
\newtheorem*{unique}{Unique fiber condition - \UFB}
\newtheorem*{weak_blow up}{Weak blow-up axiom - \WBU}

\theoremstyle{definition}
\newtheorem{definition}[theorem]{Definition}
\newtheorem{example}[theorem]{Example}

\newtheorem{remark}[theorem]{Remark}

\newtheorem{assu}[theorem]{Assumptions}

\def\AA{{\it Assumptions~\ref{A}\,}}
\def\fdl{{\mathfrak D}_{\mathfrak l}}
\def\fl{{\mathfrak l}}
\def\fSet{{\tt fSet}}
\def\Markl{{\EuScript M}}
\def\virt#1{\stackrel {#1}{\hbox{\large $\hskip -.05em \rightsquigarrow
      \hskip -.1em$}}\, }
\def\Mercin{\rotatebox{270}{$\blacktriangleright$}}
\def\kompozice{\relax{}}
\def\minggGrc{{\mathfrak M}_{\ggGrc}}
\DeclareMathOperator\can{{\it can}}
\def\D{{\mathbb D}}
\def\desusp{{\downarrow}}
\def\C{{\mathbb \Omega}}
\def\dusd#1{{\overline \delta}^\downarrow_{\!#1}}
\def\dus{{\desusp \uCO}}
\def\bftau{{\boldsymbol \tau}}
\def\redukce#1{\vbox to .3em{\vss\hbox{#1}}}

\def\bfpsi{{\boldsymbol \psi}}
\def\bfbeta{{\boldsymbol \beta}}
\def\bfalpha{{\boldsymbol \alpha}}
\def\bfphi{{\boldsymbol \phi}}
\def\pa{\partial}
\def\CO{\raisebox{.7em}{\rotatebox{180}{$\Markl$}}}
\def\rGddo{\oddGr^*}
\def\uCO{\overline{\raisebox{.7em}{\rotatebox{180}{$\Markl$}}}}

\def\colorop #1(#2;#3){{#1}
   \left(\rule{0pt}{15pt}\right.
         \hskip -3mm \begin{array}{c}
	              #3\\#2
                     \end{array}
         \hskip -3mm \left. 
   \rule{0pt}{15pt} \right)
}

\def\hGr{{\frac12\Gr}}
\def\Dio{{\tt Dio}}
\def\Whe{{\tt Whe}}

\def\frC{{\mathfrak C}}

\def\SFac{{\sf SFac}}
\def\UFB{{\sf UFib}}
\def\Rig{{\sf Rig}} 
\def\QBI{{\sf QBI}}
\def\Fac{{\sf Fac}} 
\def\WBU{{\sf WBU}} 
\def\BU{{\sf BU}}
\def\SGrad{{\sf SGrad}}

\def\Ord{{\mathbb {\tt \Delta}_{\rm semi}}}

\def\oP{{\EuScript P}}
\def\Lie{{\EuScript Lie}}
\def\Ass{{\EuScript Ass}}
\def\Com{{\EuScript Com}}
\def\out{{\rm out}}

\def\Wheterm{{\sf 1}_\Whe}
\def\End{{\EuScript E}nd}
\def\oEnd{{\mathring \End}}
\def\calE{{\mathcal E}}
\def\ttV{{\tt V}}
\def\vrt{|}

\def\bbA{{\mathbb N}}
\def\termGr{{\sf 1}_\ggGrc}
\def\oddGr{{\mathfrak {K}}_\ggGrc} 
\def\calI{{\EuScript I}}
\def\iso{{\hbox{\,$\stackrel\cong\longrightarrow$\,}}}

\def\oC{{\EuScript C}}
\def\ggGrc{{\tt ggGrc}}

\def\RTr{{\tt RTr}}\def\PRTr{{\tt PRTr}}\def\PTr{{\tt PTr}}

\def\susp{{\uparrow}}
\def\tlTw{{\widetilde {\lTw}}}
\def\im{isomorphism}
\def\Grp{{\tt Grp}}
\def\Free{{\mathbb F}}
\def\Coll{{\tt Coll}_1^\ttV}
\def\fib{\triangleright}
\def\Fib{\blacktriangleright}
\def\Term{{\tt trm}}
\def\term{{\sf 1}}
\def\Iso{{\ttO}_{\tt iso}}
\def\bfk{{\mathbb k}}
\def\ltw{{\it lTw}}

\def\lTw{{\tt lTw}}
\def\Surj{{\Fin_{\rm semi}}}
\def\Tw{{\it Tw}}

\def\Vect{{\tt Vect}}
\def\Per{{\tt Per}}
\def\ttP{{\tt P}}
\def\pperm{{\mathfrak{pp}}}

\def\LT{{\ttO}_{\tt ltrm}}
\def\VR{{\ttO}_{\tt vrt}}
\def\VI{\hbox{${\VR}\!\int\!{\Iso}$}}
\def\VIe{\hbox{${\VR(e)}\!\int\!{\Iso}$}}
\def\QV{{\tt Q}\VR}
\def\colim{\mathop{\rm colim}\displaylimits}

\def\rada#1#2{{#1,\ldots,#2}}
\def\stt#1{\{#1\}}
\def\Rada#1#2#3{#1_{#2},\dots,#1_{#3}}

\def\epi{\twoheadrightarrow}

\def\bfT{{\boldsymbol {\EuScript T}}}
\def\bfF{{\boldsymbol  {\EuScript F}}}
\def\bfS{{\boldsymbol  {\EuScript S}}}

\def\bfsigma{{\boldsymbol \sigma}}
\def\op{{\rm op}}
\def\ot{\otimes}
\def\id{1\!\!1}

\def\term{{\boldsymbol 1}}
\def\QO{{\ttO}_{\tt qb}}
\def\ttO{{\tt O}}
\def\ttP{{\tt P}}

\def\Tr{\tt Tr}

\def\bbN{{\mathbb N}}

\def\Vert{{\rm Ver}}

\def\Fin{{\tt Fin}}

\def\Bq{{\tt Bq}} 

\def\Gr{{\tt Gr}}
\def\OGr{{\tt Gr}_{\tt ord}}
\def\oM{{\mathring \Markl}}
\def\and{{\mbox { and }}}
\def\DO{{{\ttO_{\tt ord}}}}
\def\inv#1{{#1}^{-1}}

\title{Koszul duality for operadic categories}
\author{Michael Batanin}
\email{bataninmichael@gmail.com}
\author{Martin~Markl}
\email{markl@math.cas.cz}
\affiliation{Institute of Mathematics of the Czech Academy of Sciences, {\v Z}itn{\'a} 25,
         115 67 Prague 1, The Czech Republic\\ 
		 and
		 MFF UK, Sokolovsk\'a 83, 186 75 Prague 8, The Czech Republic}

\def\qb{quasi\-bijection}

\begin{document}
\bibliographystyle{plain}

\begin{abstract}
The aim of this paper (which is a sequel to {\em Operadic categories as a 
natural environment for Koszul duality}) 
is to set up the cornerstones of  
Koszul duality and Koszulity in the context of operads over a large class of
operadic categories. In particular, for these operadic categories
we will study concrete examples of binary quadratic operads, 
describe their Koszul duals and prove
that they are~Koszul.
This includes operads (for operadic categories)
whose algebras are the most important operad- and PROP-like structures such as 
the classical operads,  their variants such as cyclic or modular
operads, and also diverse versions of PROPs such as
wheeled properads, dioperads,
\hbox{$\frac12$PROPs}, and still more exotic objects such as
permutads and pre-permutads.
\end{abstract}

\maketitle

\setcounter{secnumdepth}{3}
\setcounter{tocdepth}{1}

\tableofcontents

\section*{Introduction}

Koszul duality is
an important ingredient in the theory of algebraic operads. The 
classical Koszul duality theory for associative algebras goes a long
way back to Priddy~\cite{Priddy}, but a~milestone was 
the 1994 paper {\em Koszul
duality for operads} by Ginzburg and
Kapranov~\cite{ginzburg-kapranov:DMJ94}, where they
generalize it to operads. The 
key examples are the operads $\Lie$
and $\Com$ for Lie, resp.~commutative associative algebras that are 
Koszul dual to each other, whereas the operad $\Ass$ for associative algebras
is self-dual. While many aspects of operad theory can be formulated in
general symmetric monoidal categories, such as the category of
sets or the category of spaces, Koszul duality theory is really
a feature specific to algebraic operads, meaning operads in the
category $\Vect$ of graded vector spaces.

By~\cite[Definition~4.1.3]{ginzburg-kapranov:DMJ94}, a symmetric 
quadratic operad
$\oP$ is  Koszul if the dual dg (abbreviating differential graded) operad~$\D(\oP^!)$ of 
its Koszul dual resolves $\oP$. The operad $\D(\oP^!)$
then provides, using a construction of~\cite{markl:JHRS10},
a canonical explicit $L_\infty$-algebra capturing deformations of 
$\oP$-algebras via its moduli space of Maurer{\textendash}Cartan elements.
Moreover, algebras for $\D(P^!)$ are strongly homotopy $\oP$-algebras,  
whose salient feature is the transfer property over
weak homotopy equivalences. The operad $\D(\oP^!)$ also leads to the canonical
(co)homology theory of $\oP$-algebras, 
cf.~\cite[Section~II.3.8]{markl-shnider-stasheff:book}. 
This explains the prominent position of 
quadratic Koszul operads in the traditional operad theory.

The aim of the present article is to implant the theory of
quadratic operads and their Koszulity into the context of
operadic categories. Operadic categories, introduced by the
authors in~\cite{duodel} is a general abstract framework to accommodate
operad-like structures. Just as algebraic structures of
different kinds can fruitfully be interpreted as algebras over
operads, there are also various kinds of operads, which in this
theory are interpreted as operads over operadic categories: each
operadic category has its attending notion of operad. By
carrying out the theory at this level of generality, we get
unified proofs of the main theorems of Koszul duality theory for
various flavors of operads. Not all operadic categories admit
this theory. In order for the theory to work it is necessary to
impose several further axioms on top of the general axioms of
operadic categories. This requires a considerable amount of
groundwork which we have carried out in the companion article
{\it Operadic categories as a natural environment for Koszul
duality}~\cite{part1}, whose results are really a prerequisite for the
present article. Although the additional axioms are briefly
recollected below, we have to refer the reader to~\cite{part1} for more
detailed background, also on the general theory of operadic
categories, which it is not practical to reproduce here.

As in the classical case, free operads over operadic categories will
play the central r\^ole in our generalized Koszul duality theory. 
Following the approach of~\cite{markl:zebrulka}, 
we use a version of operads whose
composition laws are binary. Operad-like
structures based on these ``partial compositions''
were later called Markl operads.   The theory of these operads in the
context of operadic categories was developed
in~\cite{part1}; we recall their definition including the necessary
auxiliary material in 
Subsection~\ref{Jsem nachlazen?}.  
In the rest of this introduction, by an operad we will mean a
Markl operad, typically denoted by $\Markl$ in contrast to the generic
notation $\oP$ for the standard operads~\cite[Definition~1.1]{duodel} 
with the compositions along all fibers performed~simultaneously.

After describing free operads, we introduce quadratic
operads as those isomorphic to free operads quotiented by
quadratic relations. As in the classical setup, each quadratic operad
$\Markl$ possesses a Koszul dual $\Markl^!$. Given an operad $\Markl$, we define
its dual dg operad $\D(\Markl)$ and, for
$\Markl$ quadratic, construct the canonical morphism \hbox{$\can_\Markl : \D(\Markl^!) \to
\Markl$}. A~quadratic operad $\Markl$ will be called
Koszul if $\can_\Markl$ induces a component-wise isomorphism of homology.
We finally prove that the operads whose algebras are the most relevant
operad- and PROP-like structures are Koszul. Our method is to show
that the dual dg operads of their Koszul duals 
are isomorphic to their minimal models described
in~\cite{BMO}. Here, as in the classical \hbox{case~\cite[Proposition~2.6]{markl:zebrulka}}, this
characterizes their Koszulity.

Throughout this article, we will be working  with operads in the
category of dg vector spaces, although the operadic
categories themselves are completely combinatorial and do not
depend on any linear structure. The reason is that even to write down the
algebraic presentations it is necessary to have a linear
structure, and that linear duality is a key ingredient, as are
the notions of homology and resolutions.

Koszulity of the semi-classical colored
operad whose algebras are modular operads was established, in the
setup of groupoid-colored operads,
in~\cite{Ward}. We believe that the advantage of our approach is that,
after some heavy preparatory work has been done, 
everything is stripped to bare bones, conveniently hiding details that only
complicate the picture,
such as the groupoid actions and explicit indices,
in the way explained in Remark~\ref{Grete vynechava zapalovani.}. 

\subsection*{Plan of the paper}

In Section~\ref{4,2} we recall from~\cite{part1} some additional
axioms of operadic categories required in the present paper, and  
Markl operads in the context of 
operadic categories. While the
underlying structure of a classical operad is a collection of spaces
equipped with actions of symmetric groups, for general operadic
categories  the situation is subtler. 
The r\^ole of underlying collections is
played by presheaves on the category $\QV(e)$ of virtual isomorphisms 
constructed  in the first part of Section~\ref{letam_205}.
The second part of that section describes  
the precise relation of Markl operads to the category $\QV(e)$.  
Free Markl operads are explicitly described in
Section~\ref{zitra_letim_do_Pragy}. 
With the notion of free operads available, we introduce
quadratic operads and their Koszul duals  in Section~\ref{pairing}.

The remaining sections are
devoted to explicit calculations.
In Section~\ref{zitra_budu_pit_na_zal} we study the constant operad
$\term_\ggGrc$  whose algebras are modular
operads. We show that this operad is binary quadratic and that
algebras over its Koszul dual are odd (aka twisted) modular operads.
In Section~\ref{Asi_pojedu_vecer.} we make a similar analysis for
operads whose algebras are ordinary and cyclic operads, and pre-permutads.  
In Section~\ref{Ta_moje_lenost_je_strasna.} we continue the analysis
for wheeled properads, dioperads, $\frac12$PROPs and permutads.
With the exception of the operads $\oddGr$, resp.~$\pperm$, 
whose algebras are odd modular
operads, resp.~pre-permutads,  all these operads are terminal operads over an
appropriate operadic category.

Section~\ref{Ceka mne nehezky vikend.} describes the cobar
construction and
Section~\ref{Prevezu Tereje?} the dual dg operad of a Markl operad,
needed for the
definition of Koszulity. Theorem~\ref{Privezu Tereje?}
then establishes Koszulity of the operads whose algebras are modular,
cyclic and ordinary operads, and wheeled PROPs. 

In Section~\ref{zitra_budu_pit_na_zal} we need to refer to concrete
axioms of modular and odd modular operads. Since the only source we
are aware of where these axioms are listed in a concise and itemized
form is the recent monograph~\cite{DJMS}, we decided to recall 
them in the Appendix. For the same reasons we also
included itemized axioms of the classical Markl operads so that 
the reader need not
consult the ancient paper~\cite{markl:zebrulka}.
To help the reader navigate through the paper, we included an index
of terminology and notation.

\bigskip

\noindent
{\bf Conventions.} 
Operadic categories and related notions were introduced
in~\cite{duodel}; some basic concepts of that paper are recalled in
\cite[Section~1]{part1}. We will freely use the terminology and notation
from there and, when necessary, refer to concrete definitions,
diagrams, results or formulas in those sources. If not stated otherwise,
$\phi^*$ will denote the image of a morphism~$\phi$ under some presheaf.

%

\subsection*{Acknowledgments}

The first author acknowledges 
support of  Praemium Academiae and RVO: 67985840 and GA\v CR EXPRO 19-28628X.
This material is based upon work supported by the 
National Science Foundation
    under Grant No.~DMS-1928930 while the second author participated in a program
  supported by the Mathematical Sciences Research Institute. The program was
   held in the Summer of 2022 in partnership with the Universidad Nacional
                              Aut\'onoma de M\'exico.
The second author was also supported by 
grant GA \v CR 18-07776S, Praemium Academiae and RVO: 67985840.
Both authors acknowledge the hospitality of the Max Planck
  Institute for Mathematics in Bonn where this work was initiated.
We express our gratitude to Joachim Kock and the referees for 
useful suggestions and comments that 
led to substantial improvement of our paper. 
The second author is also indebted
to Dominik Trnka for many useful remarks and corrections.

\section{Recollections}
\label{4,2}

In the first part of this section we recall  from~\cite{part1}
some additional requirements on the base operadic category~$\ttO$, 
formulated in the 
form of axioms, guaranteeing that free operads admit a nice explicit
description. In the second part we recall 
Markl operads.  The material is taken from~\cite{part1} almost verbatim.

\subsection{The Axioms}
\label{Potkali se u Kolina.}
Intuitively, the axioms below express 
that an operadic category $\ttO$ is of combinatorial nature,
close to various categories of graphs. 
The first axiom involves 
the subcategory $\QO \subset \ttO$ of {\em \qb{s}\/}, i.e.\  morphisms in 
$\ttO$ all of whose fibers are
the chosen local terminal objects; in~\cite{duodel} we called them
  ``trivial.''\label{Internet porad nejde.}

\begin{QIB} 
\label{Za chvili na kozni.}
All \qb{s} in $\ttO$ are invertible.
\end{QIB}

In the next axiom, $\DO$ \label{Zitra v 7 na koznim.} 
will denote  the subcategory of $\ttO$ with the objects of $\ttO$, and morphisms  $f: S
\to T$ of $\ttO$ such that  $|f| : |S| \to |T|$ is
order preserving.

\begin{weak_blow up}
\label{wbu}
For any $f' :S' \to T$ in $\DO$ and
morphisms $\pi_i : \inv{f'}(i) \to F''_i$ in $\ttO$, $i \in |T|$, 
there exists a unique factorization of $f'$
\[
\xymatrix@C=2em@R=1.2em{S' \ar[rr]^\omega \ar[dr]_{f'} && S'' \ar[ld]^{f''}
\\
&T&
}
\] 
such that $f'' \in \DO$ and 
the maps $\omega_i$ between the fibers induced by $\omega$ satisfy
$\omega_i = \pi_i$ for all $i \in |T|$.
\end{weak_blow up}

Before we recall the strong version of the above axiom, we need to
remind the reader of the notation introduced in
a lemma of~\cite{part1}:

\begin{lemma}[Lemma~2.4 of~\cite{part1}]
\label{l3}
Consider the commutative diagram in an operadic category
\begin{equation}
\label{s1}
\xymatrix@C=4em{S' \ar[d]_{f'}\ar[dr]^{f} \ar[r]^\pi  & S'' \ar[d]^{f''}
\\
T' \ar[r]^\sigma &\ T''\,.
}
\end{equation}
Let  $j\in |T''|$ and $|\sigma|^{-1}(j) = \{i\}$ for some 
$i\in |T'|$. Diagram~(\ref{s1}) determines:
\begin{itemize}
\item[(i)]
the map $f'_j:  \inv{f}(j) \to
\inv{\sigma}(j)$ whose unique fiber equals  $\inv{f'}(i)$, and
\item[(ii)]
{the induced map  $\pi_j :  
\inv{f}(j) \to
\inv{f''}(j)$.}
\end{itemize}
If $\inv{\sigma}(j)$ is
trivial, in particular if $\sigma$ is a \qb,~then
$\pi$ induces a map
\begin{equation}
\label{zitra_na_prohlidku_k_doktoru_Reichovi}
\pi_{(i,j)} : \inv{f'}(i) \to  \inv{f''}(j)
\end{equation}
which is a \qb\ if $\pi$ is.
\end{lemma}

In the situation of Lemma~\ref{l3} with $\sigma$ a \qb, the {\em
  derived sequence\/} is the sequence of morphisms 
\begin{equation}
\label{e1}
\left\{\pi_{(i,j)} : 
\inv{f'}(i) \to \inv{f''}(j), j = |\sigma|(i)\right\}_{i \in |T'|}
\end{equation}
consisting of \qb{s} if $\pi$ is a \qb. The derived sequence is featured
in the following axiom.

\begin{blow up}\label{bu}
Consider a corner in the operadic category $\ttO$
\begin{equation}
\label{c1}
\xymatrix@C=3.5em{S' \ar[d]_{f'}  &
\\
T' \ar[r]^\sigma_\sim &T''
}
\end{equation}
in which $\sigma$ is a \qb\ and $f' \in \DO$. Assume we are given
objects $F''_j$, $j \in |T''|$ together with a collection of maps
\begin{equation}
\label{e2}
\big\{\pi_{(i,j)} : \inv{f'}(i) \to F''_j,\ j = |\sigma|(i)\big\}_{i \in |T'|}.
\end{equation}
Then the corner~(\ref{c1}) can be completed uniquely  into the
commutative square
\begin{equation}
\label{eq:5}
\xymatrix@C=3.5em{S' \ar[d]_{f'} \ar[r]^\pi  & S'' \ar[d]^{f''}
\\
T' \ar[r]^\sigma_\sim &T''
}
\end{equation}
in which $f'' \in \DO$,  $\inv{f''}(j) = F''_j$ for $j \in |T''|$, and
such that
derived sequence~(\ref{e1}) induced by $f''$ coincides with~(\ref{e2}).
\end{blow up}

One of the axioms of operadic categories says that the fiber of the
unique morphism $T \to U$ to a chosen local terminal object $U$ is $T$. 
The following axiom requires that this property characterizes 
the chosen local terminal objects.

\begin{unique}
\label{Kveta_asi_spi.}
If the fiber of the unique morphism 
$!:T\to t$ to a local terminal object is $T$, then $t$ is a
chosen local terminal object.
\end{unique}

The following axiom refers to the categories $\DO$ and $\QO$ recalled
above.

\begin{Facto}
\label{dnes_prednaska_na_Macquarie}
An operadic category $\ttO$ is {\em factorizable\/} 
if each morphism $f \in \ttO$ decomposes, not necessarily uniquely, 
as $\phi \kompozice \sigma$ for some $\phi \in \DO$ and
$\sigma \in \QO$ or, symbolically, $\ttO = \DO \kompozice\, \QO$. 
\end{Facto} 

\begin{Rigid} 
\label{Koleno se pomalu lepsi.}
An operadic category $\tt O$  is
{\em rigid\/} if for each $\phi \in \DO$ the only \im\ $\sigma$ that makes
\begin{equation}
\label{Skrabe_mne_v_krku.}
\xymatrix@C=3.5em{S \ar[d]_{\phi} \ar@{=}[r]  & S \ar[d]^{\phi}
\\
T \ar[r]^\sigma_\cong &T
}
\end{equation}
commutative is the identity $\id_T : T \to T$.
\end{Rigid}

The last axiom recalled here involves a {\em grading\/}, 
defined as a map $e :
{\rm Objects}(\ttO) \to \bbN$ with the property that\label{Za
  14 dni prvni davka.}
\[
e(T) + e(F_1) + \cdots + e(F_k) = e(S)
\]
for each $f : S \to T$ with fibers $\Rada F1k$. In this situation,
the {\em grade\/} $e(f)$ of $f$ is the difference $e(f) := e(S)- e(T)$.

\begin{SG}
\label{Zadne viditelne kozni projevy nemam.}
A graded operadic category $\ttO$ is {\em strictly graded\/}
if a morphism $f \in \ttO$ is an isomorphism if and only if $e(f) =
0$.
\end{SG}

\label{Jak bude o vikendu jeste nevim.}
Finally, all operadic categories are assumed to be 
{\em constant-free\/}, which by~\cite[Definition~2.18]{part1} 
requires that  $|f| : |T| \to |S|$ is
surjective for each morphism $f : T \to S$.
With $\Fin$ the operadic category of finite
ordinals  $\bar{n} =\{1,\ldots,n\}$,
$n\in \bbN$, and their set-theoretic maps,
this means that
the cardinality functor $|\hbox{-}|:\ttO\to \Fin$  factorizes through
the operadic category $\Surj$  of nonempty finite sets and
their surjections.

The above axioms will be imposed at many places of the
paper including the rest of this section. 
The operadic category $\ttO$ will therefore
fulfill the following assumptions:

\begin{assu}
\label{A}
The operadic category $\ttO$ is a strictly graded and
factorizable, all
\qb{s} are invertible, the blow-up axiom and unique fiber condition are
fulfilled, and a
morphism $f$ is an isomorphism if and only if $e(f) =
0$.
In brief, we require
\[
\hbox{\rm \Fac\ \& \BU\ \& \QBI\ \& \UFB\ \& \SGrad.}
\]
\end{assu}

\subsection{Markl operads}
\label{Jsem nachlazen?}

Composition laws of Markl operads are labeled by morphisms which are
elementary in the sense of the following definition.

\begin{definition}
\label{plysacci_postacci}
A morphism $\phi 
: T \to S \in \DO$ in a graded operadic category $\ttO$ 
is {\em elementary\/} if all its fibers
are trivial (= chosen local terminal) 
except precisely~one whose grade is $\geq 1$.
\end{definition}

When $i \in |S|$ is the unique element with nontrivial fiber $\inv{\phi}(i)$, we will
sometimes write $\phi$ as the pair $(\phi,i)$ and say that $\phi$ is
{\em $i$-elementary\/}. If we want to name
the unique nontrivial fiber $F := \inv\phi( i)$\label{Jsme treti na ``Potkali se u Kolina.''} 
explicitly, we will write $F \fib_i T
\stackrel\phi\to S$, or  $F \fib T
\stackrel\phi\to S$ when the concrete $i \in |S|$ is not important.

\begin{definition}
\label{d3}
Let $T\stackrel{(\phi,j)}{\longrightarrow} S
\stackrel{(\psi,i)}{\longrightarrow} P$ be 
elementary morphisms. If  $|\psi|(j) = i$
we say that 
the fibers of $\phi$ and $\psi$ are {\em joint\/}.
If
 $|\psi|(j)\ne i$ we say that 
$\phi$ and $\psi$ have {\em disjoint fibers\/} or,  more
specifically, that the fibers of $\phi$ and $\psi$ are {\em
  $(i,j)$-disjoint}. Denoting $k:= |\psi|(i)$, we call
$(\psi,\phi)$ a {\em $(k,i)$-pair}.
\end{definition}

An example of a configuration with disjoint fibers is portrayed in the picture
following Definition~5.4 of~\cite{part1}. We need also to recall
from~\cite{part1} the following lemma and its corollary. 

\begin{lemma}[Lemma~5.5 of~\cite{part1}]
\label{l7} 
If the fibers of elementary morphisms $\phi$ and $\psi$ in
Definition~\ref{d3} are joint, then the
composite $\xi = \psi \circ \phi$ is elementary  as well,  with
nontrivial fiber over $i$, and the induced morphism
$\phi_i:\xi^{-1}(i)\to \psi^{-1}(i) $ is elementary with the
nontrivial fiber over $j$ equal to $\phi^{-1}(j)$.  For $l\ne i$ the
morphism $\phi_l$ equals the identity $U_c\to U_c$ of trivial objects.

If the fibers of $\phi$ and $\psi$ are { $(i,j)$-disjoint} then the morphism
$\xi = \psi \circ \phi$ has exactly two nontrivial fibers and these are
fibers over $i$ and $k: =|\psi|(j)$. Moreover, there is a canonical
induced quasibijection
\begin{subequations}
\begin{equation}
\label{harmonika}
\phi_i:\xi^{-1}(i)\to \psi^{-1}(i) \in \DO
\end{equation}
and the equality
\begin{equation}
\label{pisu_opet_v_Sydney}
\xi^{-1}(k) = \phi^{-1}(j).
\end{equation}
\end{subequations}
\end{lemma}

\begin{definition}
\label{har}
We will call the pair 
$T\stackrel{(\phi,j)}{\longrightarrow} S
\stackrel{(\psi,i)}{\longrightarrow} P$ of morphisms
in Definition~\ref{d3}  with disjoint fibers 
{\em harmonic\/}  if $\inv{\xi}(i) = \inv \psi
(i)$ and the map $\phi_i$ in~(\ref{harmonika}) is the identity.
\end{definition}

\begin{corollary}[Corollary~5.8 of~\cite{part1}]
\label{move}
Assume that
\begin{equation}
\label{eq:3}
\xymatrix@R = 1em@C=4em{& {P'} \ar[dr]^{(\psi',\, i)} &
\\
T\ar[dr]^{(\phi'',\, l)} \ar[ur]^{(\phi',\, j)} && S
\\ 
&{P''}\ar[ur]^{(\psi'',\, k)}&
}
\end{equation}
is a commutative diagram of elementary morphisms. 
Assume that $|\psi''|(l) = i$, $|\psi'|(j) = k$ and $i \not= k$.
Let $F', F'',G',G''$ be the
only nontrivial fibers of $\phi',\phi'',\psi',\psi''$, respectively.
Then one has canonical quasibijections
\begin{equation}
\label{za_tyden_poletim_do_Prahy}
\sigma': F' \longrightarrow G'' \ 
\mbox { and }\ \sigma'': F'' \longrightarrow G'.  
\end{equation}
If both pairs in~(\ref{eq:3}) are harmonic, then $F' = G''$, $F'' =
G'$ and $\sigma',\sigma''$ are the identities.
\end{corollary}

In the following definition of a Markl operad,  $\ttV$ is a 
strict
symmetric monoidal category  with 
a strict monoidal unit $\bfk$ and symmetry $\tau$, and $\Iso \subset
\ttO$ the subcategory of \label{Co  zitra zjisti?} isomorphisms. 
Axiom~(ii) of that definition is given in the simplified form assuming the
$\BU$ axiom that guarantees, by \cite[Corollary~5.7]{part1}, that all pairs of
elementary morphisms with disjoint fibers are harmonic. A general form can be
found in~\cite{part1}.

\begin{definition}
\label{markl} 
A {\em Markl $\ttO$-operad\/} in $\ttV$
is a presheaf $\Markl: \Iso^{\rm op} \to \ttV$ with values in $\ttV$ equipped, for each
elementary morphism $F\fib T \stackrel\phi\to S$,
with a ``circle product''
\begin{equation}
\label{ten_prelet_jsem_podelal}
\circ_{\phi}: \Markl(S)\otimes \Markl(F)\to \Markl(T).
\end{equation}
These operations must satisfy the following set of axioms.

\begin{itemize}
\item[(i)]
 Let
$T\stackrel{(\phi,j)}{\longrightarrow} S
  \stackrel{(\psi,i)}{\longrightarrow} P$
  be elementary morphisms such that $|\psi|(j) = i$ and let
  $\xi: T \to P$ be the composite $\psi \kompozice \phi$.  Then
the  diagram
\begin{equation}
\label{vymena}
\xymatrix@R = 1em{& {         \Markl(P)\otimes \Markl(\xi^{-1}(i))} 
\ar[dr]^(.65){\circ_\xi} &
\\
 \Markl(P)\otimes \Markl(\psi^{-1}(i))\otimes \Markl(\phi^{-1}(j)) 
\ar[dr]_{\circ_\psi\ot \id} \ar[ur]^{\id \ot \circ_{\phi_i}} &&  \Markl(T)
\\ 
&{ \Markl(S)\otimes \Markl(\phi^{-1}(j))}\ar[ur]_(.65){\circ_\phi}&
}
\end{equation}
commutes.

\item[(ii)]  
Consider the diagram 
\begin{equation}
\label{den_pred_Silvestrem_jsem_nachlazeny}
\xymatrix@R = 1em@C=4em{& {P'} \ar[dr]^{(\psi',\, i)} &
\\
T\ar[dr]^{(\phi'',\, l)} \ar[ur]^{(\phi',\, j)} && S
\\ 
&{P''}\ar[ur]^{(\psi'',\, k)}&
}
\end{equation}
of elementary morphisms with disjoint fibers as in
Corollary~\ref{move}. By the harmonicity implied by $\BU$,
the morphisms $\sigma'$, $\sigma''$
in~\eqref{za_tyden_poletim_do_Prahy} 
are the 
identities; denote $F:=F' = G''$ and $G := G' = F''$. With this
notation, the diagram
\begin{equation}
\label{Napadne jeste tento rok snih?}
\xymatrix@R = 1.5em@C=4em{
\Markl(S) \ot \Markl(G) \ot \Markl(F)\ar[r]^(.55){\circ_{\psi'} \ot \id} &
 \Markl(P') \ot \Markl(F)\ar[d]^(.5){\circ_{\phi'}}
\\
& \Markl(T)
\\
\ar[uu]^{\id \ot \tau}
\Markl(S) \ot \Markl(F) \ot \Markl(G)\ar[r]^(.55){\circ_{\psi''}  \ot \id} & 
\ \Markl(P'') \ot \Markl(G) \ar[u]_(.5){\circ_{\phi''}}
}
\end{equation}
commutes.

\item[(iii)]
For every commutative diagram
\[
\xymatrix@C=4em@R=1.4em{T'\ar[r]_\cong^\omega \ar[d]^{\phi'}
& T'' \ar[d]^{\phi''}
\\
S'\ar[r]^\sigma_\sim & S''
}
\] 
where $\omega$ is an isomorphism, $\sigma$ a \qb, 
and
$F' \fib_i T' \stackrel{\phi'}\to S'$\!, $F'' \fib_j 
T'' \stackrel{\phi''}\to S''$\!,
the~diagram
\begin{equation}
\label{posledni_nedele_v_Sydney}
\xymatrix@C=4em{\ar[d]_{\omega_{(i,j)}^* \ot \sigma^*}^\cong
\Markl(F'') \ot \Markl(S'') 
\ar[r]^(.65){\circ_{\phi''}}& \Markl(T'') \ar[d]^{\omega^*}_\cong
\\
\Markl(F') \ot
\Markl(S')\ar[r]_(.65){\circ_{\phi'}}&  \Markl(T')
}
\end{equation}
in which $\omega_{(i,j)} : F' \to F''$ is the
induced map~(\ref{zitra_na_prohlidku_k_doktoru_Reichovi}) of fibers, commutes.
\end{itemize}
Markl operad $\Markl$ is {\em  unital\/} if one is given, for
each\label{Habaneros} 
trivial (i.e.\ chosen local terminal) object~$U$ of $\ttO$, 
a~map $\eta_U
:\bfk \to \Markl(U)$  such that the diagram
\begin{equation}
\label{Holter_se_blizi.}
\xymatrix{\Markl(U) \ot \Markl(T) \ar[r]^(.6){\circ_!} & \Markl(T)
\\
\ar[u]^{\eta_U \ot \id}
\bfk \ot \Markl(T)  \ar@{=}[r]^(.58)\cong & \Markl(T) \ar@{=}[u]
}
\end{equation}
commutes whenever $T$ is such that $e(T) \geq 1$ and 
$T \fib T \stackrel!\to U$ is the unique map.
\end{definition}

\label{Jsem znepokojen.}
Let $\LT$ be the operadic subcategory of $\ttO$ consisting of its local
terminal objects. Denote by $\term_\Term : \LT \to \ttV$ the constant
functor. The collection $\{\eta_U :\bfk \to \Markl(U)\}$ of unit maps
extends  uniquely into a transformation
\begin{equation}
\label{2x2}
\eta :   \term_\Term \to \iota^* \Markl
\end{equation}
from  $\term_\Term$ to the
restriction of $\Markl$ along the inclusion
$\iota : \LT \hookrightarrow \ttO$.
Transformation~(\ref{2x2}) amounts to a family
of maps $\eta_u : \bfk \to \Markl(u)$
given for each local terminal $u \in
\ttO$, such that the diagram
\begin{equation}
\label{Je_vedro.}
\xymatrix@R=1.2em{\Markl(u)  \ar[r]^(.5){!^*} & \Markl(v)
\\
\bfk  \ar@{=}[r] \ar[u]^{\eta_u }   &
  \bfk
\ar[u]_{\eta_v}
}
\end{equation}
commutes for each (unique) map $! : v \to u$ of local terminal
objects. We will call the components   $\eta_u : \bfk \to
\Markl(u)$ of the transformation~\eqref{2x2} the {\em extended
  units\/}.\label{Mam zdravotni zpusobilost na dalsi dva roky.}

For each $T$ with $e(T) \geq 1$ and $F \fib T \stackrel!\to u$ with
$u$ a local terminal object, 
one has a map $\vartheta(T,u) :  \Markl(F) \to  \Markl(T)$ defined by the diagram
\begin{equation}
\label{proc_ty_lidi_musej_porad_hlucet}
\xymatrix@R=1em{\Markl(u) \ot \Markl(F) \ar[r]^(.6){\circ_!} & \Markl(T)
\\
  &
\\
\ar[uu]^{\eta_u \ot \id}
\bfk \ot \Markl(F) \ar@{=}[r]^\cong &\ \Markl(F)\,. \ar[uu]_{\vartheta(T,u)}
}
\end{equation}   
The unitality  offers
a generalization of Axiom~(iii) of Markl operads
which postulates for each commutative diagram
\begin{equation}
\label{moc_se_mi_na_prochazku_nechce}
\xymatrix@C=4em{T'\ar[r]^\omega_\cong \ar[d]^{\phi'}\ar[rd]^{\phi} 
& T'' \ar[d]^{\phi''}
\\
S'\ar[r]^\sigma_\cong & S''
}
\end{equation} 
where the horizontal maps are isomorphisms and 
the vertical
maps are elementary, with
$F' \fib_i T' \stackrel{\phi'}\to S'$, $F'' \fib_j 
T'' \stackrel{\phi''}\to S''$, the commutativity of 
the diagram
\begin{equation}
\label{X}
\xymatrix@C=4em
{\Markl(F) \ot \Markl(S'')&\Markl(F'') \ot \Markl(S'') \ar[l]_{\omega^*_j \ot \id}^\cong
\ar[r]^(.65){\circ_{\phi''}}& \Markl(T'') \ar[d]^{\omega^*}_\cong
\\
\ar[u]^{\vartheta(F,\inv{\sigma}(j)) \ot \id}
\Markl(F') \ot \Markl(S'')\ar[r]^{\id \ot \sigma^*}_\cong &\Markl(F') \ot
\Markl(S')\ar[r]_(.65){\circ_{\phi'}}&  \Markl(T') \,,
}
\end{equation}
in which $F := \inv{\phi}(j)$ and $\omega_j : F \to F''$ is the
induced map of fibers. Notice that if $\sigma$ is a \qb,
(\ref{X}) implies~(\ref{posledni_nedele_v_Sydney}).

\begin{definition}
\label{svedeni}
A Markl operad $\Markl$ is {\em strictly unital\/} if all the maps
$\vartheta(T,u)$ in~(\ref{proc_ty_lidi_musej_porad_hlucet})
are identities. It is {\em $1$-connected\/} 
if~the unit maps $\eta_U : \bfk \to \Markl(U)$ are isomorphisms for each
trivial~$U$.
\end{definition}

If $\Markl$ is strictly unital, one has $\Markl(F) = \Markl(F')$
in~(\ref{X}), so this diagram assumes a
particularly simple form, namely
\begin{equation}
\label{Ve_ctvrtek_letim_do_Prahy.}
\xymatrix@C=4em{
\Markl(F'') \ot \Markl(S'') \ar[d]_{\omega^*_j \ot \sigma^*}^\cong
\ar[r]^(.65){\circ_{\phi''}}& \Markl(T'') \ar[d]^{\omega^*}_\cong
\\
\Markl(F') \ot
\Markl(S')\ar[r]^(.65){\circ_{\phi'}}&\,  \Markl(T')\,.
}
\end{equation}

\section{Markl operads and virtual isomorphisms}
\label{letam_205}

We introduce the category of virtual isomorphisms $\VR$ 
related to an operadic category $\ttO$, its extension $\VI$, and the
quotient $\QV$ of  $\VI$  modulo virtual isomorphisms. In the presence of a
grading $e$ on $\ttO$  we will further consider the subgroupoid $\VR(e)$
of objects of grade $\geq 1$, the extension $\VIe$ 
and the related quotient $\QV(e)$.
Presheaves on $\QV(e)$ will
then serve as the underlying collections for Markl operads.

The operadic category $\ttO$ will be required to fulfill \AA\ although
the grading (which need not even be strict) will be used only in the
second half of this section.  All definitions and results of this
section hold also for operadic categories which are not constant-free. 
As before we denote by $\LT$ the
groupoid of local terminal objects in $\ttO$ and by
$\Iso \subset \ttO$ the subcategory with the same objects as $\ttO$,
and morphisms the isomorphisms of $\ttO$.

Let $T\in \ttO$ and let $t\in \ttO$ be a local terminal object in 
the connected component of $T$. We  therefore have a unique morphism $T\to t$
with a unique fiber $F$, which will be expressed by the shorthand
\hbox{$F \Fib T\to t$}.\label{Jsem treti na ``Potkali se u Kolina.''} 
In this situation we write
$F\virt{} T$ or $F \virt t T$ if $t$ needs to be specified, and speak about
a {\it virtual   morphism} from $F$ to $T$. Notice that $F\virt{} T$ does
not represent an ``actual'' morphism $F \to T$; in fact there may not
be any morphisms from $F$ to $T$ in $\ttO$, cf.~Example~\ref{Uz mam
  vozik v Koline.} below.\label{Za chvili sraz s Denisem.} Since we
assume $\BU$ and $\UFB$,
there exists at most one virtual morphism between two given objects of
$\ttO$ by~\cite[Lemma~2.15]{part1}. 
In fact, ``$\virt{}$'' interpreted as a relation on $\ttO$ is a preorder.

The notation $F \Fib T\to t$ shall not be confused with $F \fib T\to
t$ used before. The map $T \to t$  need not be an elementary morphism even when
$\ttO$ is graded, since we do not demand $e(T) \geq 1$ or any analog of this.
The following statement shows that one may define a~composition rule for virtual
morphisms so that they form a groupoid.

\begin{lemma} 
\label{Za_chvili_pojedu_na_schuzi_klubu.}
Virtual morphisms in the operadic category $\ttO$ form a groupoid
$\VR$. 
\end{lemma}

\begin{proof}
The lemma will follow from the following facts: each object of $\ttO$
possesses a~virtual endomorphism; virtual endomorphisms can be composed;
a virtual morphism $S \virt{} T$ exists if and only if there exist
a virtual morphism  $T \virt{} S$. 
Since there is at most one virtual morphism between two given objects of
$\ttO$ by~\cite[Lemma~2.15]{part1}, all properties of a groupoid 
follow automatically. 
 
The virtual  endomorphism of $T\in \ttO$ is 
$T \virt U T$, associated to the unique morphism
$T\to U$ to the chosen local terminal object $U$ in the connected
component of $T$.  
The composition of virtual morphisms is defined as follows.

Let $S\virt t T \virt r R$ be a chain of virtual morphisms. 
This means that $S$ is the fiber of the unique morphism 
$\phi:T\to t$, and $T$ is the fiber of $\psi:R\to r$,
i.e.\  $S \Fib T \stackrel\phi\to t$ and  $T \Fib R \stackrel\psi\to
r$, $t \in \LT$.      
By the weak blow-up axiom there exists a unique factorization of $\psi$ as in the
diagram
\begin{equation}
\label{Musim_Jarce_odvezt_kramy.}
\xymatrix@R=1em@C=1em{
R \ar[rr]^\xi \ar[rd]_\psi&&s \ar[ld]^\delta
\\
&r&
}
\end{equation}
such that $\xi_1$, the induced map between the unique fibers, equals $\phi$.
From Axiom~(iv) of an operadic category, 
one has $\xi^{-1}(1) = \xi_1^{-1}(1) = \phi^{-1}(1) =S,$ that is 
$S\Fib R \stackrel{\xi}{\to} s$.  We take the related 
virtual morphism $S\virt s R$ as the composite of $S \virt t T$ and $T
\virt r R$. 

Consider a virtual morphism  $S\virt t
T$ given by some  $S \Fib T\stackrel{\phi}{\to} t \in \ttO$. 
The morphism~$\phi$ has a unique
factorization $T\to U\stackrel{\delta}{\to} t$ through a chosen
local terminal object $U$. Let $s$ be the unique
fiber of $\delta$, i.e.\   $s\Fib
U\stackrel{\delta}{\to} t$. 
The diagram
\[
\xymatrix@R=1em@C=1em{
T \ar[rr] \ar[rd]_\phi&&U \ar[ld]^\delta
\\
&t&
}
\]
induces a morphism of fibers $S\to s$ whose 
fiber is $T$, giving rise to the required virtual morphism $T\virt {s}
S$ in the opposite direction.
\end{proof}

As we already noticed, 
if there exists a morphism $S\virt{} T$ in $\VR$ then it is
unique. Together with 
Lemma~\ref{Za_chvili_pojedu_na_schuzi_klubu.} this implies 
that $\VR$ is equivalent as a category to a discrete set.
The notation $T \virt{} S$ will  denote the unique isomorphism
from $T$ to $S$,
tacitly assuming its existence,  and $S \virt{} T$ its inverse.

As the next step towards our construction of $\QV$,  we extend $\VR$ to
a category $\VI$ which has the same objects as $\ttO$ but whose 
morphisms $T\to R$  are sequences
\[
S\stackrel{\phi}{\to} T\virt{} R
\]
where $\phi : S \to T \in \Iso$ is an isomorphism in $\ttO$. 
To define the composition, consider the~sequence
\[
S\stackrel{\phi}{\to} T\virt{} R\stackrel{\psi}{\to} Q\virt{} P.
\]
The virtual morphism  $T\virt{} R$ is related
to a morphism $T\Fib R\to r$ with a unique $r\in \LT$ and
the virtual morphism  $Q\virt{} P$ 
to $Q\Fib P\to p$ with $p\in \LT$,
The objects $R$ and $Q$ live in the same connected component, 
so  there is a unique $D\Fib Q\to r$.
We may therefore construct the diagram
\begin{equation}
\label{Motorove pily uz jedou.}
\xymatrix@R=0.1em@C=0em{
S \ar[rrrrrr]^\phi  &&&&&&T \ar[rrrrrr]^\xi&&&&&&D \ar@{=}[rr] &&D
\\
&&&&&&\raisebox{.7em}{\Mercin}&&&&&&\raisebox{.7em}{\Mercin}
&&\raisebox{.7em}{\Mercin}
\\
&&&&&&R  \ar@/_1.4em/[rrrrrrdddd]^{!} 
\ar[rrrrrr]^\psi   &&&&&&Q \ar[dddd]^{!}  &\Fib&P \ar[dddd]^\omega 
\ar@/^.9em/[rrrrrrdd]^{!} 
&&&&&&
\\
\\
&&&&&&&&&&&&&&&&\hbox{\scriptsize \WBU}&&&&p
\\
\\
&&&&&&&&&&&&r&\Fib&x \ar@/_.9em/[ruurrrrr]^{!} 
}
\end{equation}
in which $\omega$ is the unique map such that the
induced map between the fibers is $Q \stackrel!\to r$, and $\xi :=
\psi_1$ is the induced map between the fibers.
We then define the composite
\[
 (R\stackrel{\psi}{\to} Q\virt{}
 P) \kompozice (S\stackrel{\phi}{\to} T\virt{} R)
\] 
as the sequence
\begin{equation}
\label{v_patek_pak_Psenicka}
S\stackrel{\xi \kompozice\phi}{\longrightarrow}
D \virt{} P,
\end{equation}
with $D \virt{} P$ given by $D \Fib P \to x$.
The identity automorphism of $S$ is
$S\stackrel{\id}{\to} S\virt{} S$.

One can easily check that the above structure makes $\VI$ a category
which is, in fact, a  groupoid: 
for a morphism $\Phi: S\stackrel{\phi}{\to} T\virt{} R$
in $\VI$
take  the inverse
$R \virt{} T$ to $T \virt{} R$ 
and the inverse
$\psi:T \to S$ of~$\phi : S \to T \in \ttO$. Then the composite
\[
(T \stackrel \psi\to S \virt{} S)(R \stackrel {\id}\to R \virt{} T)
\]
in $\VI$ is the inverse to $\Phi$.
We leave the details to the reader.

\begin{remark}
The composite
$(T \stackrel\id \to T \virt{} R)(S \stackrel \phi\to T \virt{} T)$  
equals $S \stackrel \phi\to T \virt {}R$ as expected.
The diagram~(\ref{Motorove pily uz jedou.}) in this particular case becomes
\[
\xymatrix@R=0.1em@C=0em{
S \ar[rrrrrr]^\phi  &&&&&&T \ar[rrrrrr]^\xi&&&&&&T \ar@{=}[rr] &&T
\\
&&&&&&\raisebox{.7em}{\Mercin}&&&&&&\raisebox{.7em}{\Mercin}
&&\raisebox{.7em}{\Mercin}
\\
&&&&&&T  \ar@/_1.4em/[rrrrrrdddd]^{!}     \ar[rrrrrr]^\psi   &&&&&&T
\ar[dddd]^{!}  &\Fib&R  \ar[dddd]^{!} 
\ar@/^.9em/[rrrrrrdd]^{!}
&&&&&&
\\
\\
&&&&&&&&&&&&&&&&&&&&r
\\
\\
&&&&&&&&&&&&u&\Fib&x \ar@/_.9em/[ruurrrrr]^{!}
}
\]
where $u$ is the {\em chosen} local terminal object. Both $\xi$ and
$\phi$ are the identities by the axioms of operadic categories, while
$x=r$ because $u$ is chosen local terminal. The claim~follows.
One can similarly verify the identity
\[
(T \stackrel \psi\to R \virt{} R)(S \stackrel \id\to S \virt{} T) 
= (S \stackrel \xi\to D \virt{} T)
\]
with $S \stackrel\xi\to D$ given by the diagram
\[
\xymatrix@R=0.1em@C=0em{
S \ar[rrrrrr]^\xi&&&&&&D
\\
\raisebox{.7em}{\Mercin}&&&&&&\raisebox{.7em}{\Mercin}
\\
T  \ar@/_1em/[rrrdddd]     \ar[rrrrrr]^\psi   &&&&&&R
\ar@/^1em/[ddddlll]
&&&&&&
\\
\\
\\
\\
&&&r&&&
}
\]
We may express the result of the above calculation as the distributive law
\[
(T \stackrel \psi\to R)( S \virt{} T) \longmapsto (D \virt{} R)(S
\stackrel\xi\to D)
\]
between $\virt{}$ and $\stackrel \cong \to$~, cf.~\cite{RW}.
\end{remark}

We now consider the quotient $\QV$ of $\VI$ \label{Je streda.}
whose objects are classes of objects of $\VI$ with
respect to the  relation generated by virtual isomorphisms.  
That is, two objects are equivalent if there is a virtual isomorphism
between them. 
More precisely, $\QV$ is defined by the pushout
\begin{equation}
\label{Musim_napsat_doporuceni_pro_Jovanu.}
\xymatrix@C=3.5em@R=1.2em{\VR\ar[d]_{} \ar[r]  & \VI \ar[d]^{}
\\
\pi_0(\VR) \ar[r] &\QV
}
\end{equation}
in the category of groupoids. Since
the left vertical functor is an equivalence and the top horizontal
functor a cofibration of groupoids in the canonical model
structure on the category of groupoids~\cite{anderson}, 
the right vertical functor is an equivalence of~groupoids too. 

It is easy to see that 
morphisms between objects in $\QV$ are equivalence classes of 
non-virtual isomorphisms in the following sense. 
Let $\phi':T'\to S'$ and $\phi'':T''\to S''$ be two isomorphisms in $\ttO$.
They are 
equivalent if there exists a local terminal object $t$ such that
$\phi''$ is the induced fiber map 
in the diagram
\begin{equation}
\label{Dnes_jsem_se_koupal_v_Hradistku.}
\xymatrix@R=1em@C=1em{
T' \ar[rr]^{\phi'} \ar[rd]&&\ S' \ar[ld]
\\
&\phantom{.}t\,.&
}
\end{equation}

\begin{example}
\label{Budu_mit_nova_sluchatka.}
Assume that each connected component of $\ttO$ contains precisely one
terminal object, i.e.~all local terminal objects are the trivial (chosen)
ones. Then $\QV \cong \Iso$. This is the case of e.g.~the category $\Fin$ of
finite ordinals or of the operadic category $\Per$ in
Subsection~\ref{Michael_do_Prahy.}.
\end{example}

\begin{example} 
\label{Uz mam vozik v Koline.}
In the operadic category $\Bq(\frC)$ recalled in~\cite[Example~1.5]{part1}
two bouquets are
virtually equivalent if they differ only in the last color. Notice
that an ``actual'' morphism \hbox{$b' \to b''$} in $\Bq(\frC)$ between virtually
equivalent bouquets exists if and only if $b' = b''$.
The groupoid ${\tt Q}\Bq(\frC)_{\tt vrt}$ is
the groupoid of strings  $(i_1,\ldots,i_k)$, $k \geq 1$, 
with morphisms arbitrary~bijections.
\end{example}

\begin{example} 
\label{Snad_to_projde_i_formalne.}
Two graphs in the operadic category $\Gr$ of \cite[Definition~3.13]{part1}
are virtually equivalent if they  differ only in the global orders of their
leaves. Morphisms in ${\tt Q}\Gr_{\tt vrt}$ are 
isomorphisms of graphs which need not preserve the
global orders.  
\end{example}

Assume that $\ttO$ possesses a grading
$e : {\rm Objects}(\ttO) \to \bbN$. In
this case we denote by $\VR(e)\subset \VR$ the full subgroupoid  with
objects $T \in \ttO$ such that $e(T) \geq 1$. \label{Musim to dorvat.}
We construct $\VIe$ out
of $\VR(e)$ and $\Iso$  as before, and define 
its quotient $\QV(e)$ by replacing $\VR$
by $\VR(e)$ in~(\ref{Musim_napsat_doporuceni_pro_Jovanu.}).
In the following lemma, $\ttV$  denotes a 
strict symmetric 
monoidal category as in Definition~\ref{markl}.

\begin{lemma}
\label{zitra_Holter}
Each  unital Markl $\ttO$-operad $\Markl$ with values in $\ttV$ 
induces a covariant functor
$\VR(e) \to \ttV$, denoted~$\Markl$ again, which acts as $\Markl$ on objects, and
on virtual morphisms is defined by
\[
\Markl(F \virt{} T) := \vartheta(T,u),
\]
where $\vartheta(T,u)$ is as
in~(\ref{proc_ty_lidi_musej_porad_hlucet}). Since $\VR(e)$ is a groupoid, 
all maps
$\vartheta(T,u)$ are invertible. 
\end{lemma}

\begin{proof}
It follows from the unitality~(\ref{Holter_se_blizi.}) of $\Markl$ 
that $\Markl(T \virt{} T) = \id_{\Markl(T)}$. Let us verify the functoriality
\begin{equation}
\label{pozitri_Holter}
\Markl(S \virt{} R) = \Markl(T \virt{} R) \kompozice \Markl(S \virt{} T).
\end{equation}
To this end we consider the commutative diagram
\begin{equation}
\label{dopocitat_si_zapisnik}
\xymatrix{
&\Markl(T) \ot \Markl(r) \ar[r]^(.6){\circ_!} & \Markl(R)
\\
&\Markl(S) \ot \Markl(r) \ar[u]^{\vartheta(T,r) \ot \id} 
\ar[r]^{\id \ot \delta^*}   & \ \Markl(S) \ot \Markl(r')\,.\ar[u]^{\circ_!}
\\
\Markl(S)\ar[r]^(.4)\cong   & 
\ar[u]^{\id \ot \eta_r} \ar[ur]_{\id \ot \eta_{r'}}
\Markl(S) \ot \bfk
}
\end{equation}
Its upper square is~(\ref{X}) applied to the
diagram
\[
\xymatrix@C=3.5em@R=1.5em{
R  \ar@{=}[r] \ar[d]_\xi  \ar[dr]^\psi & R  \ar[d]^\psi
\\
r' \ar[r]^\delta & r
}
\]
in place of~(\ref{moc_se_mi_na_prochazku_nechce}), in which the
symbols have the same meaning as in~(\ref{Musim_Jarce_odvezt_kramy.}).
The commutativity of the bottom triangle follows from the
commutativity of~(\ref{Je_vedro.}).
It follows from the definition of the maps $\vartheta(T,u)$ that the
composite
\[
\xymatrix@1{\Markl(S) \cong \Markl(S) \ot \bfk \ \ar[r]^(.55){\id \ot \eta_{r'}}
&\ \Markl(S) \ot \Markl(r') \
\ar[r]^(.63){\circ_!} & \ \Markl(R)
}
\]
in~(\ref{dopocitat_si_zapisnik}) equals the left-hand side
of~(\ref{pozitri_Holter}), while the composite
\[
\xymatrix@1{
\Markl(S) \cong \Markl(S) \ot \bfk \
\ar[r]^(.55){\id \ot \eta_{r}} & \ \Markl(S) \ot \Markl(r) \
\ar[rr]^{\vartheta(T,r) \ot \id} &&
\ \Markl(T) \ot \Markl(r) \ \ar[r]^(.63){\circ_!} & \ \Markl(R)
}
\]
equals the right-hand side of~(\ref{pozitri_Holter}).
\end{proof}

\begin{proposition}
\label{dnes_schuze_klubu}
The $\Iso$-presheaf structure of a Markl $\ttO$-operad $\Markl$ in $\ttV$ combined with
the functor $\Markl: \VR(e) \to \ttV$ of Lemma~\ref{zitra_Holter} makes $\Markl$ an
$\VIe$-presheaf via the formula
\begin{equation}
\label{Pujdu_s_Jarkou_na_demonstraci?}
\Markl\big(S\stackrel{\phi}{\to} T\virt{} R\big) 
:=  \phi^* \kompozice \Markl(R \virt{} T).
\end{equation}
\end{proposition}

\begin{proof}
Since clearly  
$\Markl(T\stackrel\id\to T \virt{} T) = \id_{\Markl(T)}$, 
we need only to prove that
\begin{equation}
\label{uz_zase_stavebni_stroje_duni}
\Markl\big(S\stackrel{\phi}{\to} T\virt{} R\big) 
\kompozice \Markl\big(R\stackrel{\psi}{\to} Q\virt{}
 P\big) =
\Markl\big(S\stackrel{\xi \kompozice\phi}{\longrightarrow}
(Q \virt{} P)(D \virt{} Q)
\big)
\end{equation}
for the composite
\[
 (R\stackrel{\psi}{\to} Q\virt{}
 P) \kompozice (S\stackrel{\phi}{\to} T\virt{} R)
= S\stackrel{\xi \kompozice\phi}{\longrightarrow}
(Q \virt{} P)(D \virt{} Q)
\] 
defined in~(\ref{v_patek_pak_Psenicka}).   
Evaluating both sides of~(\ref{uz_zase_stavebni_stroje_duni})  
using~\eqref{Pujdu_s_Jarkou_na_demonstraci?} gives
\[
\phi^* \kompozice \Markl(R\virt{} T) \kompozice  \psi^* \kompozice \Markl(P
\virt{} Q) =
(\xi \kompozice\phi)^* \kompozice
\Markl(Q \virt{} D)\Markl(P \virt{} Q).
\]
Since all the maps involved are isomorphisms, we easily see
that~(\ref{uz_zase_stavebni_stroje_duni}) is equivalent~to
\begin{equation}
\label{V_pondeli_plicni.}
\Markl(D \virt{} Q) \kompozice (\inv\xi)^* = (\inv\psi)^* \kompozice \Markl(T
\virt{} R).
\end{equation}
To prove this equality, consider the diagram
\begin{equation}
\label{V_patek_pak_Psenicka.}
\xymatrix{
\Markl(T) \cong \Markl(T) \ot \bfk \ar[r]^(.54){\id \ot \eta_r} 
& \Markl(T) \ot \Markl(r)\ar[d]_{\circ_!} &\ar[l]_{\xi^* \ot \id} \Markl(D) \ot
 \Markl(r) \ar[d]^{\circ_!}
\\
& \Markl(R)&\Markl(Q)\ar[l]_{\psi^*}
}
\end{equation}
in which the square is~(\ref{X}) associated to
\[
\xymatrix@C=3.5em@R=1.5em{\ar[dr]^{!}
R\ar[d]_{!} \ar[r]^\psi &Q\ar[d]^{!}
\\
r \ar@{=}[r] &r
}
\]
in place of~(\ref{moc_se_mi_na_prochazku_nechce}). It follows from
the definitions that the composite of the maps
\[
\xymatrix@1{
\Markl(T) \cong \Markl(T) \ot \bfk \ \ar[r]^(.55){\id \ot \eta_r} &
\ \Markl(T) \ot \Markl(r) \    \ar[rr]^{\inv{(\xi^* \ot \id)}}&&
\ \Markl(D) \ot
 \Markl(r) \ \ar[r]^(.65){\circ_!}&\    \Markl(Q)
}
\]
in~(\ref{V_patek_pak_Psenicka.}) equals the left-hand side
of~(\ref{V_pondeli_plicni.}), while the composite
\[
\xymatrix{
\Markl(T) \cong \Markl(T) \ot \bfk \ \ar[r]^(.55){\id \ot \eta_r} &
\ \Markl(T) \ot \Markl(r) \     \ar[r]^(.6){\circ_!} & 
\ \Markl(R) \ \ar[r]^{\inv{\psi^*}} & \   \Markl(Q)
}
\]
equals its right-hand side. 
\end{proof}

\begin{proposition}
\label{Za_14_dni_LKDL}
If $\Markl$ is a unital Markl $\ttO$-operad in a cocomplete symmetric
monoidal category $\ttV$, then the $\VIe$-presheaf
of Proposition~\ref{dnes_schuze_klubu} 
associated to $\Markl$ functorially descends to a\/
$\QV(e)$-presheaf $\oM$ by means of the left Kan extension along the
equivalence $\big(\VIe\big)^\op \to \QV(e)^\op$.
\end{proposition}

\begin{proof}
We will construct the presheaf $\oM$ explicitly.
Objects of $\QV(e)$ are, by definition, the equivalence classes $[T]$
of the
objects of $\ttO$ modulo the relation $[T'] = [T'']$\/ if\/ $T'' \virt{}
T'$. We define $\oM([T])$ as the colimit
\begin{equation}
\label{V_sobotu_budu_letat_vycvik_vlekare.}
\oM([T]) : = \colim \Markl(S)
\end{equation} 
over the groupoid of all $S \in \ttO$ virtually isomorphic to $T$. It
is clear that the
canonical injection $\iota_T : \Markl(T) \hookrightarrow \oM([T])$ is an
isomorphism.

Consider a morphism $[\phi] : [T] \to [S]$ in $\QV(e)$ 
given by an isomorphism
$\phi : T \to S$. We define $ \oM([\phi]) :
\oM([S])  \to \oM([T])$ by the diagram
\[
\xymatrix@C=3em{
\oM([S])  \ar[r]^{\oM([\phi])}    & \oM([T])
\\
\Markl(S)\ar[r]^{\phi^*}
\ar@{_{(}->}[u]^{\iota_S}_\cong
& \Markl(T) \ar@{_{(}->}[u]_{\iota_T}^\cong
}
\]
in which $\phi^*$ refers to the $\Iso$-presheaf structure of $\Markl$.

We need to show that $\oM([\phi])$ does not depend on the choice of a
representative of the map $[\phi]$ under  the equivalence that identifies $\phi'$
as in~(\ref{Dnes_jsem_se_koupal_v_Hradistku.}) with the induced map $\phi''$
between the fibers over~$t$. To this end,
consider the commutative diagram
\begin{equation}
\label{Dnes_jde_Jarka_s_bouli_pod_uchem_k_doktorovi.}
\xymatrix{
\Markl(T') \ar@{=}[r] \ar[d]_{\vartheta(T',t)}
& \bfk \ot \Markl(T')  \ar[d]_{\eta_u \ot \id}  &  \ar[d]^{\eta_u \ot \id}
\bfk \ot \Markl(S') \ar@{=}[r]\ar[l]_{\id  \ot \phi'^*} &
\Markl(S')
\ar@/_2em/[lll]_{\phi'^*}  \ar[d]^{\vartheta(S',t)}
\\
\Markl(T'') &\ar[l]_(.6){\circ_!}  \Markl(u) \ot \Markl(T') &  \Markl(u) \ot \Markl(S') 
\ar[r]^(.63){\circ_!}\ar[l]_{\id  \ot \phi''^*} 
&  
\Markl(S'')\ar@/^2em/[lll]_{\phi''^*}
}
\end{equation}
in which the leftmost and rightmost squares are instances
of~(\ref{proc_ty_lidi_musej_porad_hlucet}). The commutativity of the
central square and  of the upper part
is clear.  Finally, the commutativity of the lower part follows from
axiom~(\ref{posledni_nedele_v_Sydney}) of Markl operads. An easy
diagram chase shows that the commutativity
of~(\ref{Dnes_jde_Jarka_s_bouli_pod_uchem_k_doktorovi.}) implies the
commutativity of the middle square in
\[
\xymatrix@C=3em{
\oM([T']) \ar@{^{(}->}[r]^{\iota_{T'}}_\cong \ar@{=}[d]
&\Markl(T')  \ar[d]_{\vartheta(T',t)}  &
\Markl(S')
\ar[l]_{\phi'^*}  \ar[d]^{\vartheta(S',t)}& \oM([S']) 
\ar@{_{(}->}[l]_{\iota_{S'}}^\cong \ar@{=}[d]
\\
\oM([T''])\ar@{^{(}->}[r]^{\iota_{T''}}_\cong
 &\Markl(T'')  
&  
\Markl(S'')\ar[l]_{\phi''^*}&\ \oM([S'']) \,.
\ar@{_{(}->}[l]\ar@{_{(}->}[l]_{\iota_{S''}}^\cong
}
\]
The independence of $\oM([\phi])$ on the choice of a
representative of $[\phi]$ is now clear. 
\end{proof}

\begin{remark}
\label{Zitra_mam_prednasku_na_Macquarie.}
If $\Markl$ is strictly 
unital, the definition in~(\ref{V_sobotu_budu_letat_vycvik_vlekare.}) via
a  colimit can be replaced by
$\oM([T]) := \Markl(T)$.
\end{remark}

\section{Free Markl operads}
\label{zitra_letim_do_Pragy}

This section is devoted to our construction of free strictly
unital $1$-connected Markl operads over $\ttO$ generated by
$\QV(e)$-presheaves. In the light of~\cite[Theorem~6.5]{part1} this will also
provide free (standard) $\ttO$-operads.  We require $\ttO$ to 
satisfy \AA. The base symmetric monoidal category is assumed to be
monoidally cocomplete.

\subsection{Chains of elementary morphisms} 
In this subsection
we introduce the cornerstones of free Markl operads. We will need
the following result.

\begin{proposition}
\label{bude_pekna_zima}
Consider a diagram 
\begin{equation}
\label{grand_mall}
\xymatrix@R=1.6em@C=3em{T'\ar[r]^{\sigma_T}_\sim\ar[d]_{(\phi',j')}  & 
T''\ar[d]^{(\phi'',j'')}
\\
P'\ar[r]^{\sigma_P}_\sim \ar[d]_{(\psi',i')} & P''\ar[d]^{(\psi'',i'')}
\\
S'\ar[r]^{\sigma_S}_\sim & S''
}
\end{equation}
whose vertical maps are elementary with disjoint fibers as indicated,
and where the horizontal maps are \qb{s}. Denoting $k' := |\psi'|(j')$,  
$k'' := |\psi''|(j'')$, one has   
\begin{equation}
\label{zase_podleham}
|\sigma_S|(i') = i'' \ \and \  |\sigma_S|(k') = k'' .
\end{equation}
If we are given a subdiagram
of~(\ref{grand_mall}) consisting only of the morphisms
$\phi',\phi'',\psi',\psi'',\sigma_T$ and $\sigma_S$, i.e.\ 
\begin{equation}
\label{Pojedu_za_Jarkou_na_chalupu?}
\xymatrix@R=1.6em@C=3em{T'\ar[r]^{\sigma_T}_\sim\ar[d]_{(\phi',j')}  & 
T''\ar[d]^{(\phi'',j'')}
\\
P' \ar[d]_{(\psi',i')} & P''\ar[d]^{(\psi'',i'')}
\\
S'\ar[r]^{\sigma_S}_\sim & S''\,,
}
\end{equation} 
then the
conditions~(\ref{zase_podleham}) are also sufficient for the existence
of a unique \qb\ $\sigma_P$ as in~(\ref{grand_mall}).
\end{proposition}

\begin{proof}
The only nontrivial fiber of $\psi'$ is $\inv{\psi'}(i')$ and 
the only nontrivial fiber of $\psi''$ is $\inv{\psi''}(i'')$ so, 
by~\cite[Remark~5.2]{part1},
we have $|\sigma_S|(i') = i''$. By the
same argument,  $|\sigma_P|(j') = j''$. Since $|-|$ is a functor, we
have 
\[
k'' = |\psi''||\sigma_P|(j') = |\sigma_S| |\psi'|(j') =  |\sigma_S| (k')
\]
proving the first part of the proposition.

To prove the second part, denote by $\xi'$ (resp.~by $\xi''$) the
composite of the maps in the left (resp.\ right) column
of~(\ref{Pojedu_za_Jarkou_na_chalupu?}). Since the left column
of~(\ref{Pojedu_za_Jarkou_na_chalupu?}) is 
harmonic by~\cite[Corollary~5.7]{part1},
we may define a map $(\sigma_P)_{(i',i'')}$ by the commutativity of the diagram
\begin{equation}
\label{streda_v_Srni}
\xymatrix@C=4em{\inv{\xi'}(i')  \ar[r]^{(\sigma_T)_{(i',i'')}} 
\ar@{=}[d]_{\phi'_{i'} = \id}
 &    \inv{\xi''}(i'')\ar[d]_{\phi''_{i''}}
\\
\inv{\psi'}(i')  \ar[r]^{(\sigma_P)_{(i',i'')}}  &    \inv{\psi''}(i'')\,.
}
\end{equation}
The blow-up axiom produces a commutative diagram 
\[
\xymatrix@R=1.6em@C=3em{
P'\ar[r]^{\sigma_P}_\sim \ar[d]_{\psi'} & \tilde P''\ar[d]^{\tilde \psi''}
\\
S'\ar[r]^{\sigma_S}_\sim & S''
}
\]
in which, by construction, 
$\tilde \psi''$ is elementary with the only nontrivial fiber
${\psi''}^{-1}(i'')$ over~$i''$, and the map between nontrivial fibers
induced by $\sigma_P$ is $(\sigma_P)_{(i',i'')}$. Consider now two
commutative diagrams
\begin{equation}
\label{Necham_auto_na_Rokyte.}
\xymatrix@R=1.6em@C=3em{
T'\ar[r]^{\sigma_P\kompozice \phi'}_\sim \ar[d]_{\xi'} & \tilde P''\ar[d]^{\tilde \psi''}
\\
S'\ar[r]^{\sigma_S}_\sim & S''
} 
\ \and \
\xymatrix@R=1.6em@C=3em{
T'\ar[r]^{\phi''\kompozice \sigma_T}_\sim \ar[d]_{\xi'} &  P''\ar[d]^{\psi''}
\\
S'\ar[r]^{\sigma_S}_\sim & S''\,.
}
\end{equation}
In both diagrams, the right vertical map is elementary, with the only
nontrivial fiber ${\psi''}^{-1}(i'')$. We will show that both
$\sigma_P\kompozice \phi'$ and $\phi'' \kompozice \sigma_T$ induce the same map
between nontrivial fibers. One has
\[
(\sigma_P\kompozice \phi')_{(i',i'')} = (\sigma_P)_{(i',i'')} \kompozice\phi'_{i'}
\]
while
\[
(\phi''\kompozice \sigma_T)_{(i',i'')} = \phi''_{i''}\kompozice (\sigma_T)_{(i',i'')}.
\]
By the defining diagram~(\ref{streda_v_Srni}), the right-hand sides of
both equations coincide. By~\BU,
the diagrams in~(\ref{Necham_auto_na_Rokyte.}) are the same, therefore
both squares in~(\ref{grand_mall}) with $\sigma_P$ constructed above
commute. This finishes the proof.  
\end{proof}

For the validity of the following lemma and
Lemma~\ref{spousta-prace-1} below, only the weak blow-up axiom $\WBU$ is 
required.

\begin{subequations}
\begin{lemma}
\label{spousta-prace}
Let $\rho: S \to T \in \DO$ be elementary with the unique fiber $F$
over $a \in |T|$. 
Suppose that we are given a chain of
elementary morphisms 
\begin{equation}
\label{eq:4}
F \stackrel {\varphi_1} \longrightarrow F_1  \stackrel  {\varphi_2} \longrightarrow
F_2  \stackrel  {\varphi_3} \longrightarrow
F_3  \stackrel  {\varphi_4} \longrightarrow \cdots  \stackrel
{\varphi_{l-1}} \longrightarrow F_{l-1}.
\end{equation}
Then there exists a unique factorization
\begin{equation}
\label{eq:2}
S \stackrel {\rho_1} \longrightarrow S_1  \stackrel  {\rho_2} \longrightarrow
S_2  \stackrel  {\rho_3} \longrightarrow
S_3  \stackrel  {\rho_4} \longrightarrow \cdots  \stackrel
{\rho_{l-1}} \longrightarrow S_{l-1}  \stackrel  {\rho_l} \longrightarrow  T
\end{equation}
of $\rho$ into elementary morphisms such that $(\rho_l \kompozice \cdots \kompozice
\rho_s)^{-1}(a) = F_{s-1}$ for each $2 \leq s \leq l$, 
and $(\rho_s)_a = \varphi_s$ for
each $1 \leq s < l$.
\end{lemma}
\end{subequations}

\begin{proof}
We will inductively construct maps in the commutative diagram
\begin{equation}
\label{eq:10}
\xymatrix{S \ar[r]^{\rho_1}\ar[d]_{\rho} & 
S_1\ar[r]^{\rho_2} \ar@/^.1pc/[dl]_{\eta_1}  & S_2 \ar@/^.3pc/[dll]_{\eta_2}
  \ar[r]^{\rho_3} & S_3  \ar[r]^{\rho_4}\ar@/^.6pc/[dlll]_{\eta_3} &
\cdots \ar[r]^{\rho_{l-1}} &  S_{l-1}\,. \ar@/^.9pc/[dlllll]_{\eta_l} 
\\
T
}
\end{equation}
The weak blow-up axiom implies that the maps
\[
\varphi_1 : F = \inv{\rho}(a) \to F_1, \ \id : \inv{\rho}(i) = U_i \to U_i
\mbox { for } i \ne a,
\]
uniquely determine a decomposition $\rho = \eta_1 \kompozice
\rho_1$. Clearly, $\eta_1$ is elementary with the unique fiber $F_1$
and we may apply the same reasoning to $\eta_1$ in place of $\rho$. The
result will be a unique decomposition $\eta_1 = \eta_2 \kompozice
\rho_2$. Repeating this process $(l-1)$ times and defining $\rho_l : =
\eta_l$ finishes the proof. 
\end{proof}

\begin{remark}
In the situation of~(\ref{eq:4}),
assume that  the pair $(\varphi_t,\varphi_{t+1})$ has
$(i,j)$-disjoint fibers for some $1 \leq t \leq l-2$. Then the corresponding pair
$(\rho_t,\rho_{t+1})$ in~(\ref{eq:2}) has  $(i+a-1,j+a-1)$-disjoint
fibers. This is an immediate consequence of Axiom~(iv) of an
operadic category.
\end{remark}

\begin{lemma}
\label{spousta-prace-1}
 With the notation of Lemma~\ref{spousta-prace}, suppose that we
are given two chains of elementary morphisms as in~(\ref{eq:4}) of the
form 
\begin{subequations}
\begin{equation}
\label{eq:8}
F \stackrel {\varphi_1} \longrightarrow F_1 \stackrel{\varphi_2} \longrightarrow
\cdots \stackrel{\varphi_{u-1}} \longrightarrow {F_{u-1}}
\stackrel{\varphi'_{u}}\longrightarrow F'_{u} \stackrel
         {\varphi'_{u+1}}\longrightarrow F_{u+1} \stackrel{\varphi_{u+2}}
         \longrightarrow \cdots \stackrel {\varphi_{l-1}} \longrightarrow
         F_{l-1}
\end{equation}
and
\begin{equation}
\label{eq:9}
F \stackrel {\varphi_1} \longrightarrow F_1  \stackrel{\varphi_2} \longrightarrow
\cdots  
\stackrel{\varphi_{u-1}} \longrightarrow {F_{u-1}} 
\stackrel{\varphi''_{u}}\longrightarrow F''_{u}  
\stackrel  {\varphi''_{u+1}}\longrightarrow F_{u+1} 
\stackrel{\varphi_{u+2}} \longrightarrow \cdots
\stackrel {\varphi_{l-1}} \longrightarrow F_{l-1}
\end{equation}
\end{subequations}
such that the diagram
\[
\xymatrix@R = 1em{& {F'_u} \ar[dr]^{\varphi'_{u+1}} &
\\
F_{u-1}\ar[dr]^{\varphi_u''} \ar[ur]^{\varphi_u'} && F_{u+1}
\\ 
&{F''_u}\ar[ur]^{\varphi''_{u+1}}&
}
\]
commutes.
Then the corresponding decompositions~(\ref{eq:2}) are of the form
\begin{subequations}
\begin{equation}
\label{eq:6}
S \stackrel {\rho_1} \longrightarrow S_1 \stackrel{\rho_2} \longrightarrow
\cdots \stackrel{\rho_{u-1}} \longrightarrow {S_{u-1}}
\stackrel{\rho'_{u}}\longrightarrow S'_{u} \stackrel
         {\rho'_{u+1}}\longrightarrow S_{u+1} \stackrel{\rho_{u+2}}
         \longrightarrow \cdots \stackrel {\rho_{l-1}} \longrightarrow
         S_{l-1}  \stackrel {\rho_{l}} \longrightarrow T ,
\end{equation}
respectively
\begin{equation}
\label{eq:7}
S \stackrel {\rho_1} \longrightarrow S_1  \stackrel{\rho_2} \longrightarrow
\cdots  
\stackrel{\rho_{u-1}} \longrightarrow {S_{u-1}} 
\stackrel{\rho''_{u}}\longrightarrow S''_{u}  
\stackrel  {\rho''_{u+1}}\longrightarrow S_{u+1} 
\stackrel{\rho_{u+2}} \longrightarrow \cdots
\stackrel {\rho_{l-1}} \longrightarrow S_{l-1}
\stackrel {\rho_{l}} \longrightarrow T,
\end{equation}
\end{subequations}
and the diagram
\begin{equation}
\label{eq:11}
\xymatrix@R = 1em{& {S'_u} \ar[dr]^{\rho'_{u+1}} &
\\
S_{u-1}\ar[dr]^{\rho_u''} \ar[ur]^{\rho_u'} && S_{u+1}
\\ 
&{S''_u}\ar[ur]^{\rho''_{u+1}}&
}
\end{equation}
commutes.
\end{lemma}

\begin{proof}
We will rely on the notation used in the proof of 
Lemma~\ref{spousta-prace}. It is clear from the inductive construction
described there that the initial parts of the canonical decompositions
corresponding to~(\ref{eq:8}) resp.~(\ref{eq:9}) coincide and are equal to
\[
S \stackrel {\rho_1} \longrightarrow S_1  \stackrel{\rho_2} \longrightarrow
\cdots  
\stackrel{\rho_{u-1}} \longrightarrow {S_{u-1}}. 
\]
Consider the following two stages of the inductive construction in the
proof of Lemma~\ref{spousta-prace}:
\[
\xymatrix{S \ar[r]^{\rho_1}\ar[d]_{\rho}   & \cdots
  \ar[r]^{\rho_{u-1}} & S_{u-1}  \ar[r]^{\rho'_u}\ar@/^.6pc/[dll]_{\eta_{u-1}} &
S'_u  \ar@/^1.1pc/[dlll]_(.3){\eta'_u}
\ar[r]^{\rho'_{u+1}} &  S'_{u+1} \ar@/^1.3pc/[dllll]_(.25){\eta'_{u+1}} 
\\
T
}
\]
and
\[
\xymatrix{S \ar[r]^{\rho_1}\ar[d]_{\rho}  & \cdots
  \ar[r]^{\rho_{u-1}} & S_{u-1}  \ar[r]^{\rho''_u}\ar@/^.6pc/[dll]_{\eta_{u-1}} &
S''_u  \ar@/^1.1pc/[dlll]_(.3){\eta''_u}
\ar[r]^{\rho''_{u+1}} &  S''_{u+1}\,. \ar@/^1.3pc/[dllll]_(.25){\eta''_{u+1}} 
\\
T
}
\]
The maps $\eta_{u-1}$,  $\eta'_{u+1}$ and $\eta''_{u+1}$ are 
elementary, with the 
nontrivial fibers $F_{u-1}$ resp.~$F_{u+1}$.  By construction, the 
horizontal maps in the factorizations
\[
\xymatrix@C=4em{S_{u-1} \ar[r]^{\rho'_{u+1} \kompozice \rho'_{u}} \ar[d]^{\eta_{u-1}} & S'_{u+1}
  \ar[ld]^{\eta_{u+1}} 
\\
T
}
\ \and \
\xymatrix@C=4em{S_{u-1} \ar[r]^{\rho''_{u+1} \kompozice \rho''_{u}} \ar[d]^{\eta_{u-1}} & S''_{u+1}
  \ar[ld]^{\eta_{u+1}} 
\\
T
}
\] 
induce the same map between these nontrivial fibers, namely 
$\varphi'_{u+1} \kompozice \varphi'_{u} = \varphi''_{u+1} \kompozice
\varphi''_{u}$. By the uniqueness of the weak blow-up, the diagrams in the
above display coincide, so diagram~(\ref{eq:11}) 
with $S_{u+1} = S'_{u+1} = S''_{u+1}$
commutes. It is obvious that the remaining parts of~(\ref{eq:6}) 
and~(\ref{eq:7}) are the same.
\end{proof}

\subsection{Free operads} Let us proceed  to our
description of free Markl operads. 
In this subsection, $\ttV$ will be a cocomplete strict symmetric monoidal
category and $0$ its initial object.

\begin{definition}
\label{Jarka_dnes_u_lekare}
$\QV(e)$-presheaves in $\ttV$ will be called  
{\em $1$-connected  $\ttO$-collections\/} in $\ttV$. We will denote by
$\Coll(\ttO)$, or simply $\Coll$ when $\ttO$ is understood,
the corresponding category.
\end{definition}

For $E \in \Coll$, we will often
write simply $E[T]$ instead of~$E([T])$.
Notice that a \hbox{$1$-connected}  $\ttO$-collection can equivalently be
defined as a $\QV$-presheaf $E$  such that $E(T) = 0$ if $e(T) = 0$. 
It follows that a  $1$-connected  $\ttO$-collection in $\ttV$
is the same as an $\Iso$-presheaf $\calE$ with values in $\ttV$ such that
\begin{itemize}
\item[(i)]
$\calE(T) = 0$ if $e(T) = 0$ ($1$-connectivity),
\item[(ii)]
 $\calE(T) = \calE(F)$ whenever  $F \Fib T 
\stackrel!\to u$, and
\item[(iii)]
$\phi'^* = \phi''^*$, where $\phi'$ is as 
in~(\ref{Dnes_jsem_se_koupal_v_Hradistku.}) and $\phi''$ is the
induced map between the fibers.
\end{itemize}

\begin{example}
It follows from Example~\ref{Budu_mit_nova_sluchatka.} that the
category $\Coll(\Fin)$ 
is isomorphic to the category of $1$-connected $\Sigma$-modules,
i.e.~sequences $\{E(n) \in \ttV\}_{n \geq 2}$, with actions of the
symmetric groups $\Sigma_n$. 
\end{example}

\begin{proposition}
\label{Pomuze_vitamin_C?}
One has a forgetful functor $\Box:   {\tt SUMOp}^\ttV_1(\ttO) \to \Coll(\ttO)$
from the category of\/ $1$-connected strictly
unital Markl $\ttO$-operads to the category  of $1$-connected
$\ttO$-collections in a cocomplete symmetric
monoidal category~$\ttV$ defined, on objects~by
\begin{equation}
\Box \Markl\big([T]\big) := 
\begin{cases}
\Markl(T) &\hbox {if $e(T) \geq 1$}
\\
0 & \hbox {otherwise.} 
\end{cases}
\end{equation}
\end{proposition}

\begin{proof}
The $\Iso$-presheaf structure induces on $\Markl$ an
$\VIe$-presheaf structure by Proposition~\ref{dnes_schuze_klubu}.
The functor $\Box$ is then the composite of the functor $\Markl \mapsto \oM$
of Proposition~\ref{Za_14_dni_LKDL}, cf.\ Remark~\ref{Zitra_mam_prednasku_na_Macquarie.}, with the functor that replaces
the values of the presheaf $\oM$ by $0$ on objects of grade zero. 
\end{proof}

In the rest of this section we construct a left adjoint $E \mapsto
\Free(E)$ to the forgetful functor of Proposition~\ref{Pomuze_vitamin_C?}.
Our strategy will be to construct a Markl $\ttO$-operad $\lTw$
with values  in
the category of groupoids $\Grp$, extend $E$ to a functor $E : \lTw
\to \ttV$ and define $\Free(E)$ as the colimit of this functor.
The building blocks of the operad $\lTw$ will be the towers
\begin{equation}
\label{t1-1}
\bfT: =
T \stackrel {\tau_1} \longrightarrow T_1  \stackrel  {\tau_2} \longrightarrow
T_2\stackrel  {\tau_3} \longrightarrow 
\cdots  \stackrel  {\tau_{k-1}}   \longrightarrow  T_{k-1}  
\stackrel  {\tau_{k} }  \longrightarrow  U_c 
\end{equation}
of elementary morphisms as in Definition~\ref{plysacci_postacci},
with $\tau_k$ the unique morphism to a
chosen local terminal object~$U_c$.
Since  $\tau_k$ bears no information, we will sometimes write the
tower~as
\begin{equation}
\label{t1}
\bfT: =
T \stackrel {\tau_1} \longrightarrow T_1  \stackrel  {\tau_2} \longrightarrow
T_2\stackrel  {\tau_3} \longrightarrow 
\cdots  \stackrel  {\tau_{k-1}}   \longrightarrow  T_{k-1}.
\end{equation}
Let $t_1,\ldots,t_{k}$ be the unique nontrivial fibers of
$\tau_1,\ldots,\tau_{k}$; notice that  $t_k = T_{k-1}$.
We will call $t_1,\ldots,t_{k}$ the {\em fiber sequence\/}
of the tower $\bfT$. The number $k$ is the {\em height\/} of the tower  $\bfT$.

We will denote by $\Tw(T)$ the set of all towers with the
initial term $T$. 
A~{\em~morphism\/} $\bfsigma : \bfT' \to \bfT''$   
of towers~(\ref{t1-1}) is an array $\bfsigma =
(\sigma_1,\sigma_2,\ldots,\sigma_{k})$ of \im{s} as in
\begin{equation}
\label{f1}
\xymatrix{T' \ar[d]_{\tau'_1} \ar[r]^{\sigma_1}_\cong  & T'' \ar[d]^{\tau''_1}
\\
T_1'  \ar[d]_{\tau'_2} \ar[r]^{\sigma_2}_\cong  & T_1'' \ar[d]^{\tau''_2}
\\
\vdots \ar[d]_{\tau'_{k-1}}  & \vdots \ar[d]^{\tau''_{k-1}}
\\
T_{k-1}' \ar[r]^{\sigma_{k}}_\cong & \,T_{k-1}''\,.
}
\end{equation}

\begin{definition}
\label{zitra_vedeni}
A {\em labeled tower\/} is a pair $(\ell,\bfT)$ consisting of a
tower $\bfT$ as in~(\ref{t1-1}) together with an \im\  (the {\em labeling})   
$\ell : X \to
T$. We denote by ${\it lTw}(X)$ the set of all labeled towers of this~form.
\end{definition}

We will equip  ${\ltw}(X)$ with the structure of a groupoid
generated by morphisms of two types, modulo the commutativity relations
specified below. 
Each morphism $\bfsigma : \bfT' \to \bfT''$   
of towers~(\ref{t1-1}) determines a 
morphism $(\ell,\bfsigma) 
:(\ell,\bfT') \to (\sigma_1\ell,\bfT'')$ of the {\em first
type\/}. These morphisms compose as follows. Suppose that
$(\ell,\bfsigma'): (\ell,\bfT')  \to (\sigma_1'\ell,\bfT'')$ and
$(\sigma_1'\ell,\bfsigma''): (\sigma_1'\ell,\bfT'')  \to
(\sigma_1''\sigma_1'\ell,\bfT''')$ are two such morphisms,  
then
\[
(\sigma_1'\ell,\bfsigma'') \kompozice  (\ell,\bfsigma') :=
(\sigma_1''\sigma_1'\ell, \bfsigma''\kompozice \bfsigma').
\]

To define morphisms of the second type, consider two 
towers of elementary morphisms,
\[
\bfT':= T \stackrel {\tau_1} \longrightarrow T_1 \stackrel{\tau_2} \longrightarrow
\cdots \stackrel{\tau_{u-1}} \longrightarrow {T_{u-1}}
\stackrel{\tau'_{u}}\longrightarrow T'_{u} \stackrel
         {\tau'_{u+1}}\longrightarrow T_{u+1} \stackrel{\tau_{u+2}}
         \longrightarrow \cdots \stackrel {\tau_{k-1}} \longrightarrow
         T_{k-1}
\]
and
\[
\bfT'':= 
T \stackrel {\tau_1} \longrightarrow T_1  \stackrel{\tau_2} \longrightarrow
\cdots  
\stackrel{\tau_{u-1}} \longrightarrow {T_{u-1}} 
\stackrel{\tau''_{u}}\longrightarrow T''_{u}  
\stackrel  {\tau''_{u+1}}\longrightarrow T_{u+1} 
\stackrel{\tau_{u+2}} \longrightarrow \cdots
\stackrel {\tau_{k-1}} \longrightarrow T_{k-1},  
\]
as in~(\ref{t1}). Their associated fiber sequences
are clearly of the form
\[
t_1,\ldots,t_{u-1},t'_u,t'_{u+1},t_{u+2},\ldots,t_k
\ \hbox { resp. } \ 
t_1,\ldots,t_{u-1},t''_u,t''_{u+1},t_{u+2},\ldots,t_k.
\]
Assume that the diagram 
\[
\xymatrix@R = 1em{& {T'_u} \ar[dr]^{\tau'_{u+1}} &
\\
T_{u-1}\ar[dr]^{\tau_u''} \ar[ur]^{\tau_u'} && T_{u+1}
\\ 
&{T''_u}\ar[ur]^{\tau''_{u+1}}&
}
\]
is as in~(\ref{eq:3}) from Corollary~\ref{move}, with  
$\phi' = \tau'_u$, $\phi'' = \tau''_u$,  
$\psi' = \tau'_{u+1}$ and $\psi'' = \tau''_{u+1}$, and thus $t'_u =
t''_{u+1}$ and $t'_{u+1} = t''_u$.   
The above situation, by definition, determines an invertible morphism
$\vartheta_u:(\ell,\bfT') \to (\ell,\bfT'')$ of the {\em second
  type\/}.  The resulting groupoid will be denoted $\lTw(X)$.

Morphisms of both types are
subject to relations. They are easy to figure out, so we address only
the most complicated case 
of morphisms $\vartheta_u, \vartheta_v$ of the
second type with $|v\!-\!u| = 1$. To this end, consider the diagram
\begin{equation}
\label{Stale jeste jedu podle letniho casu.}
\xymatrix@R = 1.3em@C = 3.7em{
&& {T'_u} \ar[r]^{\tau'_{u+1}}   \ar[rdd]_{\omega'_{u+1}} &  T'_{u+1} \ar[rdd]^{\tau'_{u+2}} 
&
\\
&&&\vartheta'_{u+1}
\\
\cdots\ar[r]&T_{u-1}  \ar[ddr]^{\tau_u''} \ar[uur]^{\tau_u'} &
\vartheta_u
& T_{u+1} \ar[r]^{\tau_{u+2}}  & T_{u+2}  \ar[r]&\cdots
\\
&&&\vartheta''_{u+1}
\\ 
&&{T''_u} \ar[ruu]^{\omega''_{u+1}}
\ar[r]^{\tau''_{u+1}} 
&T''_{u+1}  \ar[uur]_{\tau''_{u+2}}&
}
\end{equation}
representing the composition $\vartheta_{u+1}''\vartheta_u \vartheta'_{u+1}$, 
such that $\tau'_{u}$ is $r$-elementary, $\tau'_{u+1}$ is
$j$-elementary, $\tau'_{u+2}$ is $i$-elementary, and $i,
|\tau'_{u+2}|(j)$ $|\tau'_{u+2}\tau'_{u+1}|(r)$ are mutually
distinct elements of $|T_{u+2}|$ -- the case of three disjoint fibers.
Notice that the maps $\tau'_{u}, \tau'_{u+1}$ and $\tau'_{u+2}$
uniquely determine the remaining ones, by the following

\begin{lemma}
\label{Dnes naposledy jede Lednacek.}
There exists a natural reflection  $(\phi',\psi') \leftrightarrow
(\phi'',\psi'')$  between $(k,i)$-pairs, cf.~Definition~\ref{d3}, 
and $(i,k)$-pairs such that  $\phi'\psi' = \phi''\psi''$.
\end{lemma}

\begin{proof}
Assume that  $T \stackrel{\psi'}\to P'  \stackrel{\phi'}\to S$  is a
$(k,i)$-pair. Then $\alpha := \phi'\psi'$ has precisely two
nontrivial fibers, say $F'$ over $k$ and $F''$ over $i$. The weak
blow-up axiom produces the commutative diagram
\[
\xymatrix@R=1em@C=1em{
T \ar[ddr]_{\alpha}   \ar[rr]^{\psi''}  && P''\ar[ddl]^{\phi''}
\\
&\raisebox{.5em}{\WBU}&
\\
&\phantom{.}S\,.&
}
\]
in which $\phi''$ is elementary, with its only nontrivial fiber $F'$
over $k$, under the condition that the prescribed fiber maps
$\psi''_s$, $s \in |T|$, are the unique maps to the chosen local
terminal objects except $\psi''_k$ which is the identity
$\id : F' \to F'$. Then $(\phi'',\psi'')$ is an
$(i,k)$-pair. Applying the above construction to 
$(\phi'',\psi'')$ clearly produces the original $(\phi',\psi')$.
\end{proof}

Returning to~\eqref{Stale jeste jedu podle letniho casu.},
Lemma~\ref{Dnes naposledy jede Lednacek.} with $\phi' = \tau'_{u+2}$
and $\psi' = \tau'_{u+1}$ gives the maps $\tau_{u+2}$ and
$\omega'_{u+1}$, the
maps $\omega'_{u+1}$ and $\tau'_u$ determine $\omega''_{u+1}$ and
$\tau''_u$ and, finally, $\tau_{u+2}$ and  $\omega''_{u+1}$ determine
$\tau''_{u+2}$ and $\tau''_{u+1}$. In the same manner, we obtain the diagram
\[
\xymatrix@R = 1.3em@C = 3.7em{
&& {T'_u} \ar[r]^{\tau'_{u+1}}    &  T'_{u+1} \ar[rdd]^{\tau'_{u+2}} 
&
\\
&&\vartheta'_u
\\
\cdots\ar[r]&T_{u-1} \ar[r]^{\tau_u} \ar[ddr]^{\tau_u''} \ar[uur]^{\tau_u'} &T_u
\ar[rdd]^{\delta''_{u+1}}
\ar[ruu]_{\delta'_{u+1}}
&  \vartheta_{u+1} & T_{u+2}  \ar[r]&\cdots
\\
&&\vartheta''_u
\\ 
&&{T''_u}
\ar[r]^{\tau''_{u+1}} 
&T''_{u+1}  \ar[uur]^{\tau''_{u+2}}&
}
\]
One can prove, by generalizing Lemma~\ref{Dnes naposledy jede
  Lednacek.} to the case of three morphisms, that the 
maps $\tau''_{u}, \tau''_{u+1}$ and $\tau''_{u+2}$ are
the same in both diagrams, giving rise to the relation
\begin{equation}
\label{Vcera jsem si byl prodlouzit prukaz.}
\vartheta_{u+1}''\vartheta_u \vartheta'_{u+1} =
\vartheta_{u}''\vartheta_{u+1} 
\vartheta'_{u}.
\end{equation}
Since all morphisms of the second type are involutions by
Lemma~\ref{Dnes naposledy jede Lednacek.} again, we notice the
resemblance of~(\ref{Vcera jsem si byl prodlouzit prukaz.}) and 
the relation between the
generating transpositions of the symmetric group $\Sigma_3$.
This is because morphisms of the second type  generalize 
the interchange of adjacent levels in a leveled tree.

\begin{example}
\label{vcera_jsem_bezel_36_minut}
Since morphisms of both types preserve the height of towers, the
groupoid $\lTw(X)$ is  graded,
\[
\lTw(X) = \textstyle\coprod_{h \geq 1}\lTw^h(X),
\]
where $\lTw^h(X)$ is the subgroupoid of labeled towers of height $h$.
It is clear that $\lTw^1(X)$ is the category $X/\Iso$ of \im{s} in
$\ttO$ under $X$.

In $\lTw^2(X)$, only morphisms of the first type exist. Therefore,
labeled towers $(\ell',\bfT')$ and $(\ell'',\bfT'')$  are connected
by a morphism if and only if there is a commuting diagram
\begin{equation}
\label{vecer_ustrice}
\xymatrix{&X  \ar[ld]_{\ell'}^\cong  \ar[rd]^{\ell''}_\cong     &
\\
T'  \ar[d]_{\tau'} \ar[rr]^{\sigma_1}_\cong  &&T'' \ar[d]^{\tau''}
\\
T'_1 \ar[rr]^{\sigma_2}_\cong   &&T''_1
}
\end{equation}
with \im{s} $\sigma_1$ and $\sigma_2$.
\end{example}

For an \im\ $\omega : X' \iso X''$  one has the induced map
$\omega^*: \ltw(X'') \to \ltw(X')$ that sends the labeled tower
$(\ell'',\bfT'') \in \ltw(X'')$ to $(\ell''\omega,\bfT'')  \in
\ltw(X')$, which clearly extends to a functor (denoted by
the same symbol)  $\omega^*: \lTw(X'') \to \lTw(X')$.
This makes the collection of categories $\lTw(X)$ a $\Grp$-presheaf on
$\Iso$. 
Our next move will be to construct, for each $G \fib W \stackrel\phi\to X$, 
a  functor
\begin{equation}
\label{mohl_bych_si_dat_i_12_ustric}
\circ_\phi : \lTw(X) \times \lTw(G)   \to     \lTw(W).
\end{equation}

As the first step in this construction we will prove that each
labeled tower $(\ell, \bfT)$ can be functorially replaced by one in
which $\ell$ is a quasibijection. To this end we prove a couple of
auxiliary lemmas.

\begin{lemma}
The factorization  $\xi =\phi \kompozice \sigma$, 
$\phi \in \DO$, $\sigma \in \QO$, of an isomorphism $\xi :A \to B$ 
guaranteed by the factorization axiom
is unique, and
both $\phi$ and $\sigma$ are isomorphisms too. 
\end{lemma}

\begin{proof}
Consider two factorizations,  $\phi' \kompozice \sigma'$  and  $\phi''
\kompozice \sigma''$, of  $\xi$. Since $\sigma'$ and $\sigma''$ are \qb{s},
they are invertible, so one may define $u$ by the
commutativity of the diagram
\[
\xymatrix@R=1em{&X'  \ar[rd]^{\phi'}_\cong &
\\
A\ar[ur]^{\sigma'}_\sim \ar[rd]^{\sigma''}_\sim&&\ B\,.
\\
&X'' \ar[ru]^{\phi''}_\cong \ar[uu]_u  &
}
\]
By the left triangle, $u$ is a \qb\ while it belongs
to $\DO$ by the right triangle. The uniqueness follows
from~\cite[Corollary~2.6]{part1}. 
The invertibility of $\phi'$ and
$\phi''$ is clear.
\end{proof}

\begin{subequations}
\begin{lemma}
\label{zitra_na_ustrice}
Each corner
\begin{equation}
\label{Bude mi ta mycka fungovat?}
\xymatrix@R=1.2em@C=2.5em{\tilde T \ar[r]^\omega_\cong & T \ar[d]^\phi
\\
&S
}
\end{equation}
in which $\omega$ is an \im\ from $\DO$ and $\phi$ is elementary, can
be canonically and functorially completed to the square
\begin{equation}
\label{uz_se_to_krati}
\xymatrix@R=1.2em@C=2.5em{\tilde T \ar[d]_{\tilde \phi} \ar[r]^\omega_\cong & T \ar[d]^\phi
\\
\tilde S\ar[r]^{\tilde \omega}_\cong&S
}
\end{equation}
with $\tilde \omega$  an \im{s} from $\DO$ and $\tilde \phi$ elementary.
\end{lemma}
\end{subequations}
 
\begin{proof}
Since $\phi$ is elementary, it has  one fiber $F = \inv \phi
(i)$ for some $i \in |S|$ such that $e(F) \geq 1$, and the remaining
fibers are the 
chosen local terminal objects $u_1,\ldots,u_{i-1},u_{i+1},\ldots,u_{|S|}$.
Consider the diagram
\[
\xymatrix@C=1em@R=1.4em{\tilde T \ar[rr]^\omega \ar[rd]^\pi   && T \ar[ld]_\phi
\\
&S&
}
\]
in which $\pi := \phi\omega$. Since $\omega$ is an isomorphism, by the
functoriality of fiber functors, each induced map $\omega_s : \inv \pi
(s) \to \inv \phi (s)$, $s \in |S|$, is an isomorphism too. Using this and
the obvious fact that $\pi \in \DO$, we see that the ordered list of
fibers of $\pi$ equals
\[
v_1,\ldots,v_{i-1},G,v_{i+1},\ldots,v_{|S|}
\]
where $e(G) \geq 1$, while all the remaining fibers are of grade $0$
and cardinality $1$. The $\WBU$ produces the diagram
\[
\xymatrix@R=.6em@C=.7em{
\tilde T \ar[ddr]_(.4){\pi}   
\ar[rr]^{\tilde \phi}  && \tilde S\ar[ddl]^(.4){\tilde\omega}
\\
&\raisebox{.5em}{\hbox{\scriptsize \WBU}}&
\\
&S&
}
\]
such that the  ordered list of
fibers of $\tilde \omega$ equals
\[
v_1,\ldots,v_{i-1},U,v_{i+1},\ldots,v_{|S|}
\]
and all the induced maps between the fibers are the identities except
$\tilde \phi_i:  G \to U$, which is the unique map to the chosen local
terminal object $U$ in the connected component of $G$.

Axiom (iv) of an operadic category, cf.~\cite[Section~1]{part1} 
or~\cite[page~1634]{duodel}, identifies the fibers of $\tilde \phi$
with the corresponding fibers of the maps induced by $\tilde \phi$ 
between the fibers. 
Since, again
by the axioms of an operadic category, the fibers of the identity map
are the chosen local terminal objects, and the fiber of the unique map $G
\to U$ is $G$, we conclude that $\tilde \phi$ is elementary. Finally,
$\tilde \omega$ is an isomorphism
by~\cite[Lemma~2.12]{part1}.
This concludes the constructions.  
 
The functoriality means that any isomorphism of corners
in~(\ref{Bude mi ta mycka fungovat?}), in the usual sense of
isomorphisms of diagrams, uniquely extends to an isomorphism of the
corresponding squares~(\ref{uz_se_to_krati}). This can be established
by a standard diagram chase.
\end{proof}

\begin{proposition}
\label{22_hodin_cesty}
Each $(\ell, \bfT)\in \lTw(X)$ can be functorially replaced within 
its isomorphism
class by some  $(\tilde\ell, \tilde\bfT)$ in which $\tilde\ell$ is a \qb.
\end{proposition}

\begin{proof}
Let $\bfT$ be as in~(\ref{t1-1}) and let $\ell : X \to T$ be an \im. We
decompose $\ell$ as $\sigma_1 \kompozice \tilde\ell$, with  $\tilde\ell$ a \qb\
and $\sigma_1$ an \im\ in $\DO$. Lemma~\ref{zitra_na_ustrice} gives a
canonical square
\[
\xymatrix@R=1.2em@C=2.5em{\tilde T\ar[d]_{\tilde \tau_1} 
\ar[r]^{\sigma_1}_\cong & T \ar[d]^{\tau_1}
\\
\tilde T_1\ar[r]^{\sigma_2}_\cong&T_1
}
\]
in which $\tilde \tau_1$ is elementary and $\sigma_2$ an \im\ in
$\DO$. Repeated  application of Lemma~\ref{zitra_na_ustrice} produces a tower $\tilde \bfT$ labeled
by the \qb\ $\tilde\ell : X \to \tilde T$, as well as a morphism of
the first type $(\tilde\ell,\tilde \bfT) \to (\ell,\bfT)$.
\end{proof}

Let $\tlTw(X)$ be the graded category whose objects are
towers $(\ell,\bfT)$ labeled by a \qb. 
 Morphisms of the first type in  $\tlTw(X)$  are
those  $(\ell,\bfsigma) 
:(\ell,\bfT') \to (\sigma_1\ell,\bfT'')$ in which $\sigma_1$ is a \qb.
Morphisms of the second
type are the same as those in  $\lTw(X)$. Notice that  $\tlTw(X)$ is a
full subcategory of $\lTw(X)$. Indeed, if both $\ell$ and
$\sigma_1\ell$ are \qb{s}, then $\inv \ell$ is a \qb\ by $\QBI$, so
$\sigma_1 = (\sigma_1\ell)\inv\ell$ is a \qb\ too.
This, along with Proposition~\ref{22_hodin_cesty}, 
implies that  $\tlTw(X)$ is a full reflective graded subcategory of~$\lTw(X)$.

\begin{example}
\label{pisu_druhy_den_v_Srni}
Labeled towers $(\ell',\bfT'), (\ell'',\bfT'') \in\tlTw^2(X)$ are 
isomorphic if and only if there is the commuting
diagram~(\ref{vecer_ustrice}) in which the maps in the upper triangle
are \qb{s}.
\end{example}

Let $G \fib W \stackrel\phi\to X$ be elementary. 
Assume we are given a labeled tower $(\ell',\bfF) \in
\ltw(G)$, where
\begin{equation}
\label{17b}
\bfF: =
F \stackrel {\varphi_1} \longrightarrow F_1  \stackrel  {\varphi_2} \longrightarrow
F_2\stackrel  {\varphi_2} \longrightarrow 
\cdots  \stackrel  {\varphi_{l-1}}   \longrightarrow  F_{l-1}
\end{equation}
is a tower with the associated fibers $f_1,\ldots,f_l$,
with the labeling $\ell' : G \to F$. Assume we are also given a labeled
tower $(\tilde \ell,\bfT) \in \tlTw(X)$, with $\tilde\ell$ a \qb. 
The blow-up axiom gives a unique diagram
\[
\xymatrix@C=3.5em@R=1.2em{W \ar[d]_{\phi} \ar[r]^{\ell''}_\cong  & S \ar[d]^{\rho}
\\
X \ar[r]^{\tilde \ell}_\sim &T
}
\]
in which $F \fib S \stackrel\rho\to T$ is elementary  and
$\ell''$ an \im\ inducing the map $\ell' : G \to F$ between the unique
nontrivial fibers of $\phi$ and $\rho$, respectively.  
Lemma~\ref{spousta-prace} gives the composite tower
\begin{equation}
\label{Taronaga_Zoo}
\bfT \circ_\rho \bfF:=
S \stackrel {\rho_1} \longrightarrow S_1  \stackrel  {\rho_2} \longrightarrow
S_2  \stackrel  {\rho_3} \longrightarrow \cdots  \stackrel
{\rho_{l-1}} \longrightarrow S_{l-1}  \stackrel  {\rho_l} \longrightarrow  T
  \stackrel  {\tau_1}\longrightarrow T_1  \stackrel  {\tau_2} \longrightarrow
T_2\stackrel  {\tau_2} \longrightarrow 
\cdots  \stackrel  {\tau_{k-1}}   \longrightarrow  T_{k-1}  
\end{equation}
whose initial part is~(\ref{eq:2}), so we have the composite labeled tower
\begin{equation}
\label{Safari_ani_ne_za_tyden.}
(\tilde\ell,\bfT) \circ_\phi (\ell',\bfF) :=  (\ell'',\bfT \circ_\rho \bfF) 
\in \ltw(W).
\end{equation}
The above construction clearly extends to a functor
\[
\circ_\phi : \tlTw(X) \times \lTw(G)   \to     \lTw(W)
\]
which, precomposed with the equivalence $\lTw(X) \to \tlTw(X)$ in the
first variable, gives~(\ref{mohl_bych_si_dat_i_12_ustric}).

Let $E \in \Coll$ be a $1$-connected collection.
For a tower~(\ref{t1}) we define
\[
E(\bfT) := E[t_1] \otimes \cdots  \otimes E[t_k]
\in \ttV.
\]
We will show how the rule $E(\ell,\bfT) := E(\bfT)$ extends to a functor $E
: \lTw(X) \to \ttV$. Consider a morphism $(\ell,\bfsigma) 
:(\ell,\bfT') \to (\sigma_1\ell,\bfT'')$ of the first
type, with  $\bfsigma : \bfT' \to \bfT''$ a map   
of towers as in~\eqref{f1}. For each $0 \leq s \leq k$ one has
the commutative diagram
\[
\xymatrix@C=4em{T_s' \ar[d]_{(\tau_s',i)}\ar[dr]^{\tau_s}
  \ar[r]^{\sigma_{s+1}}_\cong  
& T_s'' \ar[d]^{(\tau_s'',j)}
\\
T_{s-1}' \ar[r]^{\sigma_s}_\cong &T_{s-1}''
}
\]
in which $\tau_s := \sigma_s \kompozice \tau'_s =  \tau''_s \kompozice
\sigma_{s+1}$ and where $T'_s := T'$,   $T''_s := T''$ if $s=0$.
Lemma~\ref{l3} provides us~with
\[
t'_s \fib  \tau^{-1}_s(j)  
\stackrel{(\tau'_s)_j}\longrightarrow \sigma_s^{-1}(j)
\ \hbox { and } \
(\sigma_{s+1})_j: \tau_s^{-1}(j) \to t''_s
\]
where $ \sigma_s^{-1}(j)$ is local terminal
by~\cite[Lemma~2.12]{part1}, so we can define $\sigma_s^\star : E[t''_s] \to   E[t'_s]$
as the composite
\[
\sigma_s^\star : E[t''_s] \stackrel{(\sigma_{s+1})_j^*}\longrightarrow E[
\tau_s^{-1}(j)] = E[t'_s]
\]
where the equality uses the
fact that $E$ is constant along virtual isomorphisms.
This in turn induces a map
\begin{equation}
\label{Dnes snad uvidim Tereje.}
\bfsigma^\star : E(\bfT'') 
=  E[t'_1] \otimes \cdots  \otimes E[t'_k] \longrightarrow
E(\bfT')
=  E[t''_1] \otimes \cdots  \otimes E[t''_k]
\end{equation}
by $\bfsigma^\star := {\sigma}_1^\star \ot \cdots \ot 
{\sigma}_k^\star$.
Define finally $E(\ell,\bfsigma) :  E(\ell,\bfT'') \to E(\sigma_1\ell,\bfT')$
as $E(\ell,\bfsigma):= \bfsigma^\star$.

Let us discuss morphisms of the second type.
Corollary~\ref{move} gives identities
\[
 t'_u = t''_{u+1} \ \mbox { and } \
t''_u = t'_{u+1}.
\]
We define the $E$-image of this map as the identification of
\[
e_1 \ot \cdots\ot e''_{u+1} \ot e'_{u+1} \ot \cdots \ot e_k
\in E[t_1] \ot \cdots \ot E[t'_u] \ot E[t'_{u+1}] \ot 
\cdots \ot E[t_k]
\]
in $E(\bfT')$ with
\[
e_1 \ot \cdots\ot e_{u+1}'\ot e''_{u+1} \ot \cdots \ot e_k
\in E[t_1] \ot \cdots \ot E[t''_u] \ot E[t''_{u+1}] \ot 
\cdots \ot E[t_k]
\]
in $E(\bfT'')$ given by the symmetry constraint in $\ttV$. 
It is simple to show that the above definition of the functor $E$ 
is compatible with the
relations between the generating morphisms of $\lTw(X)$. 
For instance, the compatibility with~\eqref{Vcera jsem si
  byl prodlouzit prukaz.} is guaranteed by the hexagon axiom for the
symmetric monoidal category $\ttV$.

Here and at
several places below we use the notation that assumes that the objects
of $\ttV$ have elements.  The interested reader can easily rewrite
formulas of this type to more general but less intuitive language of diagrams.

\begin{lemma}
\label{popozitri_Janacek}
The diagram of functors
\[
\xymatrix{& \ttV &
\\
\tlTw(X) \times \lTw(G) \ar[ur]^{E \ot E}  \ar[rr]^(.6){\circ_\phi} && 
\lTw(W) \ar[ul]_E
}
\]
commutes for an arbitrary elementary morphism $G \fib W \stackrel\phi\to X$.
\end{lemma}

\begin{proof}
Assume that $(\tilde \ell,\bfT) \in \tlTw(X)$ and  
 $(\ell',\bfF) \in \lTw(G)$, with $\bfT$ as in~(\ref{t1}) and $\bfF$
as in~(\ref{17b}). Recall that then $(\tilde\ell,\bfT) \circ_\phi
(\ell',\bfF) \in  \lTw(W)$ is
given by formula~(\ref{Safari_ani_ne_za_tyden.}). The crucial fact is
that the fiber sequence of $\bfT \circ_\rho \bfF$ is 
\[
f_1,\ldots,f_l,t_1,\ldots,t_{k},
\] 
where  $f_1,\ldots,f_l$ 
resp.~$t_1,\ldots,t_k$ is the fiber sequence of $\bfF$ resp.~of $\bfT$. 
The canonical isomorphism 
\[
E(\ell,\bfT) \ot E(\ell',\bfF) 
\cong E\big((\ell,\bfT) \circ_\phi (\ell',\bfF)\big)
\]
then follows immediately from the definition of the functor $E$ as
given above.
\end{proof}

\begin{theorem}
\label{dnes_bylo_40}
Let $\ttV$ be a cocomplete symmetric monoidal category and  let
$E \in \Coll$ be a $1$-connected collection in $\ttV$. Then the formula
\begin{equation}
\label{Konci_Neuron?}
\Free(E)(X) := 
\begin{cases}\displaystyle
\colim_{(\ell,\bfT) \in\lTw(X)} E(\ell,\bfT)&\hbox{if $e(X) \geq 1$}
\\
\bfk &\hbox{if $e(X) = 0$}
\end{cases}
\end{equation}
defines a left adjoint $E \mapsto \Free(E)$ to the forgetful functor\/
$\Box$ of  Proposition~\ref{Pomuze_vitamin_C?}. Therefore $\Free(E)$
is  the free $1$-connected strictly unital 
Markl operad generated by $E$.
\end{theorem}

Adjoining the monoidal unit
in~(\ref{Konci_Neuron?}) should be
compared to adjoining the unit to the free nonunital operad in
formula~(II.1.58) of~\cite{markl-shnider-stasheff:book}.

\begin{proof}[Proof of Theorem~\ref{dnes_bylo_40}]
Assume that $X\in \ttO$ is such that   $e(X) \geq 1$.  
It is clear that $\Free(E)(X)$ is 
graded by the height $k$  of the
underlying towers so that it decomposes as
\begin{equation}
\label{co_udela_ta_noha_na_Safari}
\Free(E)(X) \cong \coprod_{k \geq 1}  \Free^k(E)(X).
\end{equation} 
Elements of $\Free^k(E)(X)$ are equivalence classes
$[\ell,e]$ consisting of a labeling $\ell :X \iso T$ and of an
element 
$e \in E(\bfT)$ associated with a labeled tower $(\ell,\bfT)$  of height 
$k$ as 
in Definition~\ref{zitra_vedeni}.
For an isomorphism $\omega : Y \iso X$ one puts $\omega^*[\ell,e] :=
[\ell\omega,e] \in  \Free^k(E)(Y)$. This turns $\Free^k(E)$ into 
an $\Iso$-presheaf in $\ttV$.
Defining formally $\Free^0(E)$ to be the trivial presheaf $\bfk$,
one thus has a decomposition
\[
\Free(E) \cong \coprod_{k \geq 0}  \Free^k(E)
\]
of $\Iso$-presheaves in $\ttV$.

In particular, 
this shows that $\Free(E)$ is an $\Iso$-presheaf as required in
the definition of Markl operad.

To define the composition, recall that $\tlTw(X)$ is a full
reflective subcategory of $\lTw(X)$, therefore one has the canonical isomorphism 
\[
\Free^k(E)(X) \cong
\colim_{(\tilde \ell,\bfT) \in\tlTw^k(X)}  E(\tilde\ell,\bfT)
\]
with the colimit taken over towers of height $k$ labeled by \qb{s}.
The composition~law
\begin{equation}
\label{Zavolam_ty_Sarce?}
\circ_\phi: \Free^k(E)(X) \otimes \Free^l(E)(G) \longrightarrow 
\Free^{k+l}(E)(W),\ 
G \fib W \stackrel\phi\to X,\ k,l \geq 1,
\end{equation} 
is then defined as the colimit of the natural isomorphism between the
functors
\[
E \ot E: \tlTw^k(X) \times \lTw^l(G) \to \ttV \hbox { and } 
E : \lTw^{k+l}(W)  \to \ttV
\] 
established in Lemma~\ref{popozitri_Janacek}. One must also define the
composition law in~(\ref{Zavolam_ty_Sarce?}) for $k=0$, i.e.,
specify a map
\begin{equation}
\label{zavolal_jsem_ji}
\circ_\phi: \Free^l(E)(G) \cong \Free^0(E)(X) \otimes \Free^l(E)(G) 
\longrightarrow \Free^{l}(E)(W).
\end{equation}

Notice first that the grade of $X$ must be zero, so by our assumptions on $\ttO$, $X$ is 
a~local terminal object. Consider an element
$[\ell',e] \in \Free^l(E)(G)$ with  $\ell' : G \iso F$ and $e \in
E(\bfF)$ with $\bfF$ as in~(\ref{17b}). 
Using the blow-up axiom we include $\phi : W \to X$ to the diagram
\[
\xymatrix@C=1em{
W \ar[rr]^{\ell''}_\cong\ar[rd]_\phi  &&S\ar[ld]^\rho 
\\
&X&
}
\]
in which $\ell''$ induces the map  $\ell' : G \iso F$ between the
fibers. Let ${\bfS}$ be the tower as in~(\ref{eq:2})
with $X$ in place of $T$. Then $(\ell'',\bfS) \in \lTw(W)$. Since by
construction the associated fiber sequence of $\bfS$ is the same as
the associated fiber sequence of $\bfF$, one has $E(\bfF) = E(\bfS)$,
thus it makes sense to define $\circ_\phi$ in~(\ref{zavolal_jsem_ji}) by
$\circ_\phi ([\ell',e]) := [\ell'',e]$.

Notice that one cannot have $l = 0$ in~(\ref{Zavolam_ty_Sarce?}),
since the fiber of an elementary map has always positive grade. We
leave to the reader to verify that the above constructions make  
$\Free(E)$ a~Markl operad.

Let us describe  $\Free^1(E)$ explicitly. As noticed in
Example~\ref{vcera_jsem_bezel_36_minut},  
$\lTw^1(X)$ is the category $X/\Iso$ of \im{s} in
$\ttO$ under $X$. Elements of  $\Free^1(E)(X)$ are equivalence classes
$[\omega,e]$ of pairs $\omega : X \iso T$, $e \in E[T]$, modulo the
identification
$[\sigma\omega' ,e''] = [\omega' ,\sigma^*e'']$ 
for each diagram
\[
\xymatrix@C=1em{
&X \ar[ld]_{\omega'}^\cong\ar[rd]^{\omega''}_\cong   &
\\
T' \ar[rr]^\sigma_\cong && T''
}
\]
of isomorphisms in $\ttO$.
Since $\lTw^1(X)$ is connected, with a
distinguished object $\id : X \to X$, the map 
$i : E[X] \to \Free^1(E)(X)$ given by $i(e) :=  [\id,e]$ for
$e \in E[X]$, is an isomorphism. These isomorphisms
assemble 
into an isomorphism $E \cong \Free^1(E)$ of collections. 
Let us finally denote by $\iota : E \hookrightarrow \Box \Free(E)$ the
composite
\begin{equation}
\label{Kdy bude hotovy Terej?}
\iota:
E \iso \Free^1(E)  \hookrightarrow \Box \Free(E).
\end{equation}

To establish the freeness of Theorem~\ref{dnes_bylo_40} means
to prove that, for each  $1$-connected strictly  unital 
Markl operad $\Markl$ and a map of collections  $y : E \to \Box \Markl$, there
exists precisely one map $\hat y : \Free(E) \to \Markl$ of  strictly unital 
Markl operads making the diagram
\[
\xymatrix@C=3em{E\ \ar@{^{(}->}[r]^(.4){\iota}\ar[rd]^y & 
\Box \Free(E) \ar[d]^{\Box
    \hat y}
\\
&\Box \Markl
 }
\]
commutative.

Let us assume that such a map  $\hat y : \Free(E) \to \Markl$ exists and prove that
it is unique. To this end consider an arbitrary element $[\ell, e] \in
\Free(E)(X)$ given by a pair $\ell : X \iso T$, $e \in E(\bfT)$ for 
a~labeled tower  $(\ell, \bfT) \in \lTw(X)$ as in
Definition~\ref{zitra_vedeni}. For
\[
e = e_1 \ot \cdots \ot e_k  \in  E[t_1] \ot \cdots \ot E[t_k]
\]
it immediately follows from the definition of the operad
structure of $\Free(E)(X)$ that 
\[
[\ell, e] = \ell^*\big(e_1 \circ_{\tau_1}(e_2  \circ_{\tau_2}
\cdots(e_{k-1} 
\circ_{\tau_{k-1}} e_{k}) \cdots )\big) ,
\]
where we used the notation
\[
x \circ_{\tau_i} y := (-1)^{|x||y|} \circ_{\tau_i}(y \otimes x)
\]
for $x \in E[T_i],\  y \in E[t_i]$ and $1 \leq i \leq k-1$. We moreover
considered $\Rada e1k$ as elements of $\Free^1(E)$ via the
isomorphism $i : E \iso  \Free^1(E)$. Since $\hat y$ is a
morphism of operads, we have
\begin{equation}
\label{dneska_musim_zvladnout_jeste_toho_Slovaka}
\hat y([\ell, e]) = \ell^*\big(y(e_1) \circ_{\tau_1}(y(e_2)  \circ_{\tau_2}
\cdots(y(e_{k-1}) 
\circ_{\tau_{k-1}} y(e_{k})) \cdots )\big).
\end{equation}

On the other hand, one may verify
that~(\ref{dneska_musim_zvladnout_jeste_toho_Slovaka}) indeed defines a
morphism of operads with the required property, to finish the proof.
\end{proof}

\section{Quadratic Markl operads and duality}
\label{pairing}

The goal of this section is to introduce quadratic Markl operads over operadic
categories and define their Koszul duals. As customary in this
context, the base monoidal category~$\ttV$ here and in the rest of the paper
will be the category $\Vect$ of graded vector spaces
over a~field $\bfk$ of characteristic $0$. All operads will be tacitly
assumed to be strictly unital. The operadic category
$\ttO$ shall fulfill  \AA\ plus the rigidity axiom~\Rig. 

\begin{definition}
\label{Prestane Jarka pit?}
An {\em ideal\/} $\calI$ in a Markl operad $\Markl$  is a
sub-$\Iso$-presheaf of $\Markl$ which is simultaneously an ideal with respect to the
circle products~(\ref{ten_prelet_jsem_podelal}), i.e.\
\[
\circ_\phi(a \ot b) \in \calI(T) 
\hbox { if $a\in \calI(S)$ or   $b \in \calI(F)$ for $F \fib T
  \stackrel \phi\to S$.}
\]
For a  sub-$\Iso$-presheaf  $R$ 
we denote by $(R)$ the component-wise intersection of all
ideals containing $R$. We call $(R)$ the ideal {\em generated\/} by $R$.
\end{definition}

\begin{definition}
\label{zitra_seminar}
A {\em quadratic data\/} consists of a
$1$-connected collection  $E \in \Coll$ and 
an~sub-$\Iso$-presheaf  $R$ of $\Free^2(E)$. A $1$-connected
Markl operad $\Markl$ is {\em quadratic\/} if it is of the form
\[
\Markl = \Free(E)/(R).
\]
It is {\em binary\/} if the
generating collection $E$ is such that $E[T] \not= 0$ implies that $e(T)= 1$.
\end{definition}

Many examples of binary operads will be given in the
following sections. Let us proceed to our generalization of the
operadic Koszul duality of \cite{ginzburg-kapranov:DMJ94} to operads over general
operadic~categories.

We start by noticing that the 
piece  $\lTw^k(X)$ of height $k$ of the 
groupoid $\lTw(X)$ constructed in Section~\ref{zitra_letim_do_Pragy} 
decomposes into  the coproduct 
\[
\lTw^k(X) = \coprod_{c \in \pi_0(\lTw^k(X))} \lTw^k_c(X) 
\]
over the
set $ \pi_0(\lTw^k(X))$ of connected components of $\lTw^k(X)$,
which is thus also true for the $k$th piece of the $X$-component of
the free operad
\begin{equation}
\label{Do_konce_Safari_zbyvaji_tri_dny.}
\Free^k(E)(X) = \bigoplus_{c \in \pi_0(\lTw^k(X))} \Free^k_c(E)(X). 
\end{equation}
Choose a labeled tower $(\ell^c,\bfT^c)$ in each connected component
$c$ of $\lTw^k(X)$ and assume the notation
\[
\bfT^c: =
T^c \stackrel {\tau^c_1} \longrightarrow T^c_1  \stackrel  
{\tau^c_2} \longrightarrow
T^c_2\stackrel  {\tau^c_2} \longrightarrow 
\cdots  \stackrel  {\tau^c_{k-1}}   \longrightarrow  T^c_{k-1},
\]
with the associated fiber sequence $t^c_1,\ldots,t^c_{k}$. 
Since there are no automorphisms of the first type of
$(\ell^c,\bfT^c)$ in $\lTw^k(X)$ by the
rigidity of $\ttO$, we have
\[
\Free^k_c(E)(X) \cong  E[t^c_1] \otimes \cdots  \otimes E[t^c_k],
\]
so we have an isomorphism of graded vector spaces
\begin{subequations}
\begin{equation}
\label{dnes_bude_vedro}
\Free^k(E)(X) \cong
\bigoplus_{c \in \pi_0(\lTw^k(X))}    E[t^c_1] \otimes \cdots  \otimes E[t^c_k],
\end{equation}
cf.~the similar 
presentation~\cite[formula (II.1.51)]{markl-shnider-stasheff:book} 
for ``ordinary'' free operads. In the light of
Proposition~\ref{22_hodin_cesty}, one may assume that the tower 
$(\ell^c,\bfT^c)$ in~(\ref{Do_konce_Safari_zbyvaji_tri_dny.}) belongs
to $\tlTw^k(X)$, therefore~(\ref{dnes_bude_vedro}) can be written as the
direct sum
\begin{equation}
\label{dnes_bude_vedro_b}
\Free^k(E)(X) \cong
\bigoplus_{c \in \pi_0(\tlTw^k(X))}    E[t^c_1] \otimes \cdots  \otimes E[t^c_k]
\end{equation}
\end{subequations}
over isomorphism classes of objects of  $\tlTw^k(X)$.

Let $\susp E^*$ be the suspension of the component-wise linear dual of
the collection $E$.
With the above preliminaries, it is easy to define a pairing 
\begin{equation}
\label{jdu_si_lepit_142}
\langle - | - \rangle :
\Free^2(\susp E^*)(X) \ot \Free^2(E)(X) \longrightarrow 
\bfk,\ (\alpha,x) \longmapsto \alpha(x)
\end{equation}
as follows.
If $c' \not= c''$ we declare the subspaces $\Free_{c'}^2(\susp E^*)(X)$ and
$\Free_{c''}^2(E)(X)$ of $\Free^2(\susp E^*)(X)$ resp.~$\Free^2(E)(X)$  
to be orthogonal. If $c := c' = c''$, then
\[
\Free_c^2(\susp E^*)(X) \cong \susp E^*[t^c_1] \ot \susp E^*[t^c_2]\
\hbox { and }\
\Free_c^2(E)(X) \cong E[t^c_1] \ot E[t^c_2]
\]
and the pairing between $\Free_c^2(\susp E^*)(X)$ and
$\Free_c^2(E)(X)$ is defined as the canonical evaluation
\[
\susp E^*[t^c_1] \ot \susp E^*[t^c_2] \ot   E[t^c_1] \ot E[t^c_2]
\longrightarrow \bfk.
\]
We leave as an exercise to show that this definition does not depend
on the choices of the representatives $(\ell^c,\bfT^c)$.

\begin{definition}
\label{vice nez 10 000 pripadu denne}
Let $\Markl$ be a quadratic Markl operad as in
Definition~\ref{zitra_seminar}.
Its {\em Koszul dual\/} $\Markl^!$ is the quadratic Markl operad defined as
\[
\Markl^! = \Free(\susp E^*)/(R^\perp),
\]
where $R^\perp$ denotes the component-wise annihilator of $R$ in
$\Free^2(\susp E^*)$ under the pairing~(\ref{jdu_si_lepit_142}). 
\end{definition}

\begin{definition}
\label{Uz_je_o_te_medaili_zapis.}
A quadratic Markl operad $\Markl$ is {\em self-dual\/} if the associated
categories of $\Markl$- and $\Markl^!$-algebras in $\Vect$ are isomorphic. 
\end{definition}

\begin{example}
All assumptions of this section are met by the operadic category
$\Surj$ of finite non-empty sets and their surjections. The operads
for this category are the classical constant-free operads for which
Koszul duality is classical 
heritage~\cite{ginzburg-kapranov:DMJ94}. A similar example is the
operadic category $\Ord$ of non-empty ordered finite sets and
their order-preserving surjections. Our theory in this case recovers Koszul
duality for nonsymmetric operads. The terminal category ${\tt 1}$ also
satisfies the assumptions of this section. The only $1$-connected
${\tt 1}$-operad, i.e.~an associative algebra, is however the ground
field $\bfk$.
\end{example}

\section{Modular and odd modular operads}
\label{zitra_budu_pit_na_zal}

In this section we analyze binary quadratic operads over the operadic category
$\ggGrc$ of connected genus-graded ordered graphs 
introduced in \cite[Example~4.19]{part1}.
Recall from \hbox{\cite[Example~4.26]{part1}} that  $\ggGrc$ 
satisfies all of the properties
required for Koszul duality, namely, $\ggGrc$
is rigid and fulfills \AA.

We will prove that the terminal operad $\term_\ggGrc$ in the category
\ggGrc\ is binary
quadratic and describe its Koszul dual  $\oddGr :=\term_\ggGrc^!$. We then
show that algebras for $\term_\ggGrc$ are modular operads of
\cite{getzler-kapranov:CompM98} while algebras for
$\term_\ggGrc^!$ are their suitably
twisted versions. We start by analyzing
graphs in $\ggGrc$ with a small number of internal edges.

\begin{example}
\label{Snad_se_ta_Achylovka_trochu_lepsi.}
The local terminal objects of  $\ggGrc$ are genus-graded corollas
$c(\sigma)^g$ for a permutation $\sigma =(\Rada\sigma1n) \in \Sigma_n$
and a genus $g \in \bbN$:
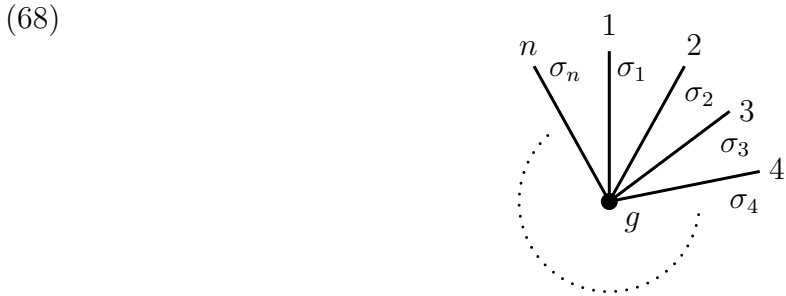
\begin{equation}
\label{Dasa_Vokata}
\raisebox{-10em}{
\psscalebox{1.0 1.0} 
{
\begin{pspicture}(0,-2.6672062)(11,1.672062)
\psdots[linecolor=black, dotsize=0.225](5.275,-0.8477939)
\rput[bl](5.475,-1.2477939){$g$}
\rput(5.275,1.5){$1$}
\psline[linecolor=black, linewidth=0.04](5.275,1.1522061)(5.275,-0.8477939)
\psline[linecolor=black, linewidth=0.04](6.275,0.95220613)(5.275,-0.8477939)
\psline[linecolor=black, linewidth=0.04](4.275,0.95220613)(5.275,-0.8477939)
\rput{-227.95706}(8.177993,-5.332993){\psarc[linecolor=black, linewidth=0.04, linestyle=dotted, dotsep=0.10583334cm, dimen=outer](5.275,-0.8477939){1.2}{0.0}{220}}
\psline[linecolor=black, linewidth=0.04](6.875,0.3522061)(5.275,-0.8477939)
\psline[linecolor=black, linewidth=0.04](7.275,-0.44779387)(5.275,-0.8477939)
\rput[bl](6.3,1.1){$2$}
\rput[bl](7,0.22061){$3$}
\rput[bl](7.4,-0.54779387){$4$}
\rput[bl](4.075,1.1){$n$}
\rput[bl](5.375,0.7522061){$\sigma_1$}
\rput[bl](6.275,0.422061){$\sigma_2$}
\rput[bl](6.875,-1){$\sigma_4$}
\rput[bl](6.75,-0.247793884){$\sigma_3$}
\rput[bl](4.475,0.7522061){$\sigma_n$}
\end{pspicture}
}
}
\end{equation}
The chosen local terminal objects are the genus-graded corollas   
$c_n^g :=c(\id_n)^g$ with $\id_n \in \Sigma_n$ 
the identity permutation.
\end{example}

\begin{example}
Any ordered connected genus-graded graph with one internal 
edge and one vertex looks like
$\xi(\Rada \lambda1k | \lambda_{k+1},\lambda_{k+2})^g$ in
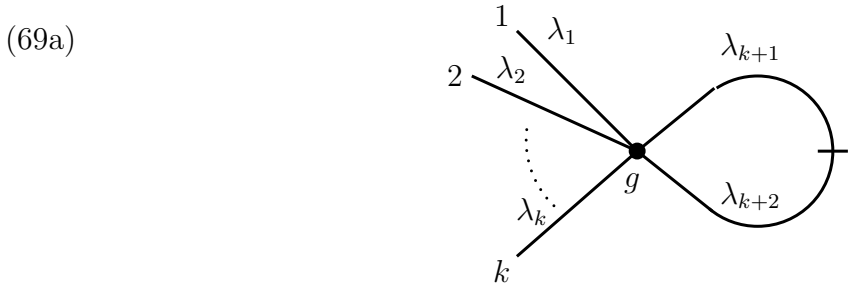
\begin{subequations}
\begin{equation}
\label{dnes_na_Rusalku left}
\raisebox{-7.5em}{
\psscalebox{1.0 1.0} 
{
\begin{pspicture}(0,-1.7275)(5.055,1.9275)
\rput{-227.95706}(5.3,-4.332993){\psarc[linecolor=black, linewidth=0.04, linestyle=dotted, dotsep=0.10583334cm, dimen=outer](5.275,-0.8477939){1.2}{40}{95}}
\rput[bl](2.5,-.55){$g$}
\psarc[linecolor=black, linewidth=0.04, dimen=outer](4.275,0.0125){1.0}{230}{123}
\psline[linecolor=black, linewidth=0.04](2.675,0.0125)(1.075,1.6125)(1.075,1.6125)
\psline[linecolor=black, linewidth=0.04](2.675,0.0125)(1.075,-1.3875)(1.075,-1.3875)
\psline[linecolor=black, linewidth=0.04](5.075,0.0125)(5.475,0.0125)
\psdots[linecolor=black, dotsize=0.225](2.675,0.0125)
\psline[linecolor=black, linewidth=0.04](2.675,0.0125)(0.475,1.0125)
\rput[bl](3.075,-0.1875){$\relax$}
\psline[linecolor=black, linewidth=0.04](3.7,0.85)(2.675,0.0125)
\psline[linecolor=black, linewidth=0.04](2.675,0.0125)(3.675,-0.7875)
\rput(0.875,1.8125){$1$}
\rput(0.25,1.0125){$2$}
\rput(0.875,-1.5875){$k$}
\rput(1.675,1.6125){$\lambda_1$}
\rput(1,1.1){$\lambda_2$}
\rput(1.275,-0.7875){$\lambda_k$}
\rput(4.175,1.39){$\lambda_{k+1}$}
\rput(4.175,-0.5875){$\lambda_{k+2}$}
\end{pspicture}
}
}
\end{equation}
with  half-edges  labeled by a
permutation $\{\Rada \lambda1{k+2}\}$ of $\{\rada1{k+2}\}$. Its automorphism group
equals $\Sigma_2$ that interchanges the half-edges
forming the loop. 
Any two graphs of this kind are isomorphic. In
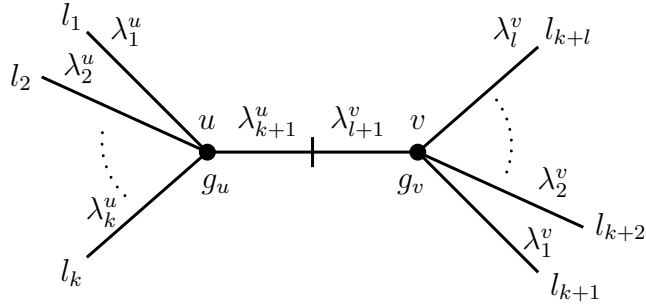
\begin{equation}
\label{dnes_na_Rusalku right}
\raisebox{-4.7em}{\rule{0em}{0em}}
\raisebox{-4em}{
\psscalebox{1.0 1.0} 
{
\begin{pspicture}(2.5,-1.7275)(40.055,1.9275)
\rput[bl](9,-.55){$g_u$}
\rput[bl](11.6,-.55){$g_v$}

\rput{-227.95706}(11.5779,-4.332993){\psarc[linecolor=black, linewidth=0.04, linestyle=dotted, dotsep=0.10583334cm, dimen=outer](5.275,-0.8477939){1}{40}{95}}
\rput{-227}(15.1,-4.332993){\psarc[linecolor=black, linewidth=0.04,
  linestyle=dotted, dotsep=0.10583334cm,
  dimen=outer](5.275,-0.8477939){1}{205}{-95}}
\psline[linecolor=black, linewidth=0.04](9.075,0.0125)(7.475,1.6125)(7.475,1.6125)
\psline[linecolor=black, linewidth=0.04](9.075,0.0125)(7.475,-1.3875)(7.475,-1.3875)
\psline[linecolor=black, linewidth=0.04](9.075,0.0125)(6.875,1.0125)
\rput(7.275,1.8125){$l_1$}
\rput(6.6,1.0125){$l_2$}
\rput(7.275,-1.5875){$l_k$}
\rput(8,1.6125){$\lambda^u_1$}
\rput(7.375,1.125){$\lambda^u_2$}
\rput(7.675,-0.7875){$\lambda^u_k$}
\psdots[linecolor=black, dotsize=0.225](9.075,0.0125)
\psline[linecolor=black, linewidth=0.04](13.475,-1.5875)(11.875,0.0125)(11.875,0.0125)
\psline[linecolor=black, linewidth=0.04](13.475,1.4125)(11.875,0.0125)(11.875,0.0125)
\psline[linecolor=black, linewidth=0.04](14.075,-0.9875)(11.875,0.0125)
\rput(13.075,1.6125){$\lambda^v_l$}
\rput(13.675,-0.3875){$\lambda^v_2$}
\rput(13.475,-1.1875){$\lambda^v_1$}
\psdots[linecolor=black, dotsize=0.225](11.875,0.0125)
\rput(13.975,-1.7875){$l_{k+1}$}
\rput(14.575,-0.9875){$l_{k+2}$}
\rput(13.875,1.6125){$l_{k+l}$}
\psline[linecolor=black, linewidth=0.04](10.475,0.2125)(10.475,-0.1875)
\psline[linecolor=black, linewidth=0.04](9.075,0.0125)(11.875,0.0125)(11.875,0.0125)
\rput(9.875,0.4125){$\lambda^u_{k+1}$}
\rput(11.075,0.4125){$\lambda^v_{l+1}$}
\rput(9.075,0.4125){$u$}
\rput(11.875,0.4125){$v$}
\end{pspicture}
}
}
\end{equation}
\end{subequations}
we depict a general graph
$\nu(\Rada {\lambda^u}1k | \lambda^u_{k+1},\lambda^v_{l+1}|\Rada{\lambda^v}1l)^{g_u|g_v}$ with one
internal edge and two vertices labeled by $u,v \in \{1,2\}$ with 
genera $g_u,g_v \in \bbN$. 
Its global order is determined by a $(k,l)$-shuffle 
\[
\{l_1 < \cdots < l_k, \ l_{k + 1} < \cdots < l_{k + l}\} = \{\rada 1{k+l}\}.
\]
The half-edges adjacent to $u$ are labeled by a permutation $\lambda^u$ of
$\{\rada 1{k+1}\}$; the half-edges adjacent to 
$v$  by a permutation $\lambda^v$ of $\{\rada {1}{l+1}\}$.
Two such graphs with the same global orders and the same genera
are always isomorphic. There are no nontrivial
automorphisms except for the case $k=l=0$ and $g_u = g_v$ when
the graph is an interval with no legs. Then one has the automorphism
flipping it around its middle.
\end{example}

\def\semidirect{\makebox{\hskip .3mm$\times \hskip
    -.8mm\raisebox{.2mm}
 {\rule{.17mm}{1.9mm}}\hskip-.75mm$\hskip 2mm}}

\def\+{\!+\!}
\begin{example}
\label{chteji_po_mne_abych_velel_grantu}
A general graph   $\xi(\Rada \lambda1k |
\lambda_{k+1},\lambda_{k+2}|\lambda_{k+3},\lambda_{k+4})^g$ with two
internal edges and one vertex is depicted in
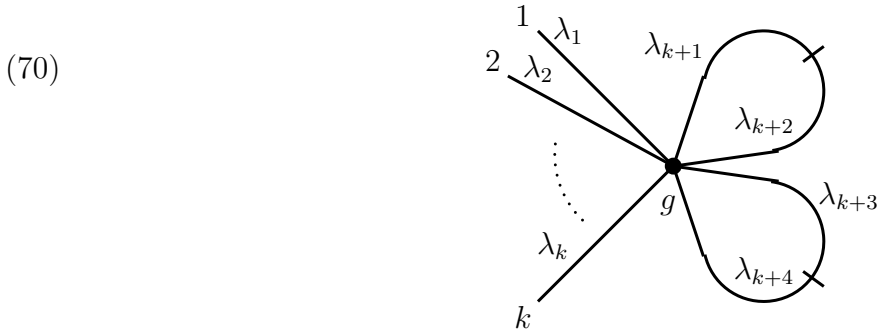
\begin{equation}
\label{pozitri_do_Srni}
\raisebox{-8em}{
\psscalebox{1.0 1.0} 
{
\begin{pspicture}(2.5,-2.1219277)(8.795,2.1219277)
\psline[linecolor=black,
linewidth=0.04](6.275,0.0030721934)(4.675,1.6030722)(4.475,1.8030722)
\rput(3.5,0){
\rput{-227.95706}(5.3,-4.332993){\psarc[linecolor=black, linewidth=0.04, linestyle=dotted, dotsep=0.10583334cm, dimen=outer](5.275,-0.8477939){1.2}{40}{95}}
\rput[bl](2.6,-.65){$g$}
}
\psdots[linecolor=black, dotsize=0.225](6.275,0.0030721934)
\psline[linecolor=black, linewidth=0.04](6.275,0.0030721934)(4.075,1.2030722)
\rput[bl](6.075,-0.5969278){$\relax$}
\rput(4.275,-1.9969279){$k$}
\rput(4.275,1.9969279){$1$}
\rput(4.875,1.8030722){$\lambda_1$}
\rput(4.475,1.30722){$\lambda_2$}
\rput(4.675,-1.05){$\lambda_k$}
\rput(6.275,1.6030722){$\lambda_{k+1}$}
\rput(7.475,0.60307217){$\lambda_{k+2}$}
\rput{-1.6502923}(-0.025787057,0.21568896){\psarc[linecolor=black, linewidth=0.04, dimen=outer](7.475,1.0030721){0.8}{280}{170}}
\psline[linecolor=black, linewidth=0.04](6.275,0.0030721934)(7.675,0.20307219)
\rput(8.6,-0.3969278){$\lambda_{k+3}$}
\rput(7.475,-1.3969278){$\lambda_{k+4}$}
\rput(3.875,1.4030722){$2$}
\psline[linecolor=black, linewidth=0.04](6.675,1.2030722)(6.275,0.0030721934)
\rput{-87.20956}(8.106839,6.5177426){\psarc[linecolor=black, linewidth=0.04, dimen=outer](7.475,-0.9969278){0.8}{280}{170}}
\psline[linecolor=black, linewidth=0.04](7.675,-0.1969278)(6.275,0.0030721934)
\psline[linecolor=black, linewidth=0.04](6.675,-1.1969278)(6.275,0.0030721934)
\psline[linecolor=black, linewidth=0.04](8, 1.3930722)(8.275,1.5969278)
\psline[linecolor=black, linewidth=0.04](8,-1.3969278)(8.275,-1.5969278)
\psline[linecolor=black, linewidth=0.04](6.275,0.0030721934)(4.475,-1.7969278)
\end{pspicture}
}
}
\end{equation}
Its local order at its single vertex is determined by a permutation
$\lambda$ of $\{\rada 1 {k\+4}\}$.
Its automorphism group equals the semidirect
product $\Sigma_2 \semidirect ( \Sigma_2  \times
\Sigma_2) $. We leave the similar detailed analysis of the remaining
graphs with two internal edges as an exercise. 

Our next task will be to describe free operads over $\ggGrc$ using
formula~(\ref{dnes_bude_vedro_b}). 
As the first step towards this goal we 
describe isomorphism classes of labeled towers 
$(\ell,\bfT) \in  \tlTw^2(X)$ for the ordered graph
\[
X := \xi(\rada 1k
 |{k\+1},{k\+2}|{k\+3},{k\+4})^g
\]
i.e.\ for the graph in~\eqref{pozitri_do_Srni} with $\lambda$ the identity.
As observed in Example~\ref{pisu_druhy_den_v_Srni}, it suffices to consider 
diagrams~(\ref{vecer_ustrice}) in which all maps in the upper triangle
are \qb{s}. Since the graphs $X, T'$ and $T''$ in that triangle have one
vertex only, all $\ell', \ell''$ and $\sigma_1$ belong to~$\OGr$,
therefore they are the identities by
\cite[Corollary~2.6]{part1}. Isomorphism classes in
$\tlTw^2(X)$ are thus represented by maps 
\begin{equation}
\label{dnes_na_Breznik}
\tau:  \xi(\rada 1k
 |{k\+1},{k\+2}|{k\+3},{k\+4} )^g
\longrightarrow  \xi(\Rada \nu1k | \nu_{k+1},\nu_{k+2})^{g+1}
\end{equation}
modulo the equivalence that identifies  
$\tau'$ with $\tau''$  if and only if there
exists an isomorphism $\sigma$ such that $\tau'' = \sigma \kompozice
\tau'$. Notice that a map as in~(\ref{dnes_na_Breznik}) is automatically
elementary, and that all elementary maps from $X$ decreasing the grade
by $1$ are of this form.
Now define the ``canonical''~maps
\[
p_i: 
 \xi(\rada 1k |{k\+1},{k\+2}|{k\+3},{k\+4})^{g} \longrightarrow
\xi(\rada 1k |{k\+1},{k\+2})^{g+1}, \ i = 1,2,
\]
by postulating that $p_1$ (resp.~$p_2$) contracts the loop
$\{{k\+1},{k\+2}\}$ (resp.~$\{{k\+3},{k\+4}\})$ leaving the other loop
unaffected. In other words, the injection $\psi_1$ (resp.~$\psi_2$) of half-edges
defining $p_1$ (resp.~$p_2$) is the order-preserving injection
\[
(\rada 1{k\+2}) \hookrightarrow (\rada 1{k\+4}) 
\] 
that misses the subset $\{k\+1,k\+2\}$ (resp.~$\{k\+3,k\+4\})$.

We claim that for each $\tau$
in~(\ref{dnes_na_Breznik}) there exist a unique $i \in \{1,2\}$ and a
unique isomorphism $\sigma$ making the diagram  
\begin{equation}
\label{lyze_mi_nejedou}
\xymatrix@R=3em@C=-6em{
&\xi(\rada 1k |{k\+1},{k\+2}|{k\+3},{k\+4})^g\ar[ld]_{p_i}\ar[dr]^\tau& 
\\
\xi(\rada 1k |{k\+1},{k\+2})^{g+1}\ar[rr]^\sigma_\cong
&& \xi(\Rada \nu1k | \nu_{k\+1},\nu_{k\+2})^{g+1}
}
\end{equation}
commutative. Since, by definition, morphisms in $\ggGrc$ preserve global
orders, one has for the injections $\psi_\tau$ resp.~$\psi_\sigma$
of half-edges defining $\tau$ resp.~$\sigma$,
\[
\psi_\tau(\nu_j) = \psi_\sigma(\nu_j) = j \ \hbox { for } \  1\leq j \leq k.
\]
Since  $\psi_\tau$ must further preserve the involutions on the sets
of half-edges, there are only two
possibilities:

\noindent 
{\it Case 1:  $\psi_\tau\{\nu_{k+1},\nu_{k+2}\} = \{k\+3,k\+4\}$.\/}  
In this case we take $i=1$ in~(\ref{lyze_mi_nejedou}) and define 
\[
 \psi_\sigma(\nu_{k+1}) := \psi_\tau(\nu_{k+1}) -2,\
 \psi_\sigma(\nu_{k+2}) := \psi_\tau(\nu_{k+2}) -2.
\]
It is clear that with this choice the diagram
in~(\ref{lyze_mi_nejedou}) is commutative and that it is the only such
choice.

\noindent 
{\it Case 2:  $\psi_\tau\{\nu_{k+1},\nu_{k+2}\} = \{k\+1,k\+2\}$.\/}  
In this case we take $i=2$ and define 
\[
\psi_\sigma(\nu_{k+1}) := \psi_\tau(\nu_{k+1}),\
\psi_\sigma(\nu_{k+2}) := \psi_\tau(\nu_{k+2}).
\]
Intuitively, in Case 1 the map $\tau$ contracts the loop
$\{k\+1,k\+2\}$, in  Case 2 the loop $\{k\+3,k\+4\}$. In both cases
the isomorphism
$\sigma$ is uniquely determined by the behavior of
$\tau$ on the non-contracted edge.

The above calculation shows that there are precisely two isomorphism
classes of objects of $\tlTw^2(X)$, namely those of $p_1$ and
$p_2$.
Notice that
\[
p^{-1}_1(1) = \xi(\rada 1k,k\+3,k\+4 |{k\+1},{k\+2})^g
\ \hbox { and } \
p^{-1}_2(1) = \xi(\rada 1k,{k\+1},{k\+2} |{k\+3},{k\+4})^g.
\]
Let $E\in \Coll$ be a  $1$-connected  $\ggGrc$-collection as in 
Definition~\ref{Jarka_dnes_u_lekare}. Formula~(\ref{dnes_bude_vedro})
gives
\begin{align*}
\Free^2(E)(X) \cong\  & 
E[\xi(\rada 1k,{k\+3},{k\+4} |{k\+1},{k\+2})^g] 
\ot E[\xi(\rada 1k |{k\+1},{k\+2})^{g+1}]
\\
\nonumber 
\oplus\ & E[\xi(\rada 1k,{k\+1},{k\+2} |{k\+3},{k\+4})^g]
\ot E[\xi(\rada 1k |{k\+1},{k\+2})^{g+1}]\,.
\end{align*}
Analogous expressions for $X =  \xi(\Rada \lambda1k |
\lambda_{k+1},\lambda_{k+2}|\lambda_{k+3},\lambda_{k+4} )^g$ can be obtained from
the above ones by substituting $j \mapsto \lambda_j$ for $1 \leq j
\leq k+4$. The result is
\begin{align}
\label{po_navratu_ze_Srni_b}
\Free^2(E)(X) \cong\  & 
E[\xi(\Rada \lambda1k,\lambda_{k+3},\lambda_{k+4} |\lambda_{k+1},\lambda_{k+2})^g] 
\ot E[\xi(\Rada \lambda1k |\lambda_{k+1},\lambda_{k+2})^{g+1}]
\\
\nonumber 
\oplus\ & E[\xi(\Rada \lambda1k,\lambda_{k+1},\lambda_{k+2} |\lambda_{k+3},\lambda_{k+4})^g]
\ot E[\xi(\Rada \lambda1k |\lambda_{k+1},\lambda_{k+2})^{g+1}]\,.
\end{align}
\end{example}

\begin{example}
\label{toho_snehu_je_moc}
The right-hand side of~(\ref{po_navratu_ze_Srni_b}) depends only on the
virtual isomorphism classes in $\QV(e)$ of the graphs involved. By the
observations made in
Example~\ref{Snad_to_projde_i_formalne.}, these classes do
not depend on the global orders. 
In this particular case, this means that they do not depend on 
the indices $\Rada \lambda1k$; we can therefore simplify 
the exposition by removing 
them from notation and drawings. We also replace  $\Rada
\lambda{k+1}{k+4}$  by less clumsy
symbols $a,b,c$ and $d$.  
With this convention, we write the two representatives of isomorphism classes in
$\tlTw^2(X)$~as
\begin{align*}
\xi(*,c,d |a,b)^g& \fib \xi(*|a,b|c,d)^g
\stackrel {p_1} \longrightarrow
\xi(*|c,d)^{g+1} \hbox { and}
\\
\xi(*,a,b |c,d)^{g}& \fib \xi(*|a,b|c,d)^{g}
\stackrel {p_2} \longrightarrow
\xi(*|a,b)^{g+1},
\end{align*}
where $*$ stands for unspecified labels.
The right-hand side of~(\ref{po_navratu_ze_Srni_b}) now takes the form
\begin{align}
\label{Bude_Ovar_dalsich_pet_let prezidentem?}
\big\{
E[\xi(*,c,d |a,b)^g] \ot E[\xi(* |c,d)^{g+1}]\big\}
\oplus\big\{ E[\xi(*,a,b |c,d)^g]\ot E[\xi(* |a,b)^{g+1}] \big\}
\end{align}
with the first summand corresponding to the class of $p_1$ and the second
to the class of $p_2$.

We also noticed that the maps $p_1$ and $p_2$ are determined by specifying
which of the two loops of $\xi(*|a,b|c,d)^g$ they contract. The map
$p_1$ and its unique nontrivial fiber is thus depicted as
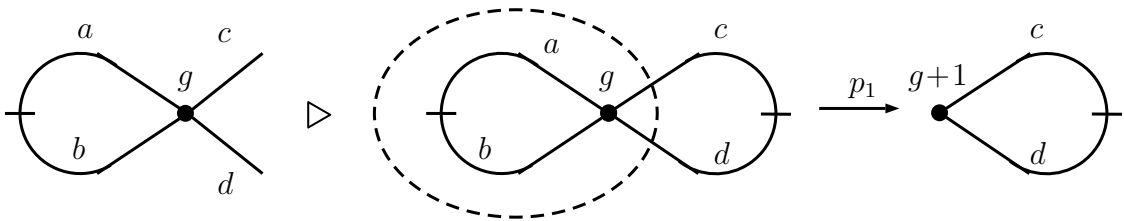
\[
\psscalebox{1.0 1.0} 
{
\begin{pspicture}(0,-1.4)(14.850119,1.4)
\rput[b](2.4,.3){$g$}
\rput[b](8,.3){$g$}
\rput[b](12.4,.3){$g\+1$}
\rput{-55.289433}(0.4305688,0.822039){\psarc[linecolor=black, linewidth=0.04, dimen=outer](1.0,0.0){0.8}{108.04867}{355.90524}}
\psline[linecolor=black, linewidth=0.04](0.4,0.0)(0.0,0.0)
\psdots[linecolor=black, dotsize=0.225](2.4,0.0)
\psellipse[linecolor=black, linewidth=0.04, linestyle=dashed, dash=0.17638889cm 0.10583334cm, dimen=outer](6.7916665,0.0)(1.9,1.4)
\rput[bl](0.9583333,1.0){$a$}
\rput[bl](0.89166665,-0.6){$b$}
\rput[bl](7.1583333,0.8){$a$}
\rput[bl](6.2916665,-0.6){$b$}
\rput[bl](2.825,.95){$c$}
\rput[bl](2.825,-1.0666667){$d$}
\rput[bl](13.625,1.0){$c$}
\rput[bl](9.425,1.0){$c$}
\rput[bl](9.425,-0.6666667){$d$}
\rput[bl](13.625,-0.6666667){$d$}
\psline[linecolor=black, linewidth=0.04](3.425,0.8)(2.425,0.0)
\psline[linecolor=black, linewidth=0.04](2.425,0.0)(3.425,-0.8)
\rput[bl](11.225,0.2){$p_1$}
\rput[bl](4,-0.2){\LARGE $\fib$}
\psline[linecolor=black, linewidth=0.04, arrowsize=0.05291666666666668cm 2.0,arrowlength=1.4,arrowinset=0.0]{<-}(11.891666,0.06666667)(10.825,0.06666667)
\psline[linecolor=black, linewidth=0.04](2.425,0.0)(1.225,0.8)
\psline[linecolor=black, linewidth=0.04](2.425,0.0)(1.225,-0.8)
\rput{-180.46338}(24.849737,-0.082693085){\psdots[linecolor=black, dotsize=0.225](12.425036,0.008896015)}
\rput{-235.7528}(21.6421,-11.452748){\psarc[linecolor=black, linewidth=0.04, dimen=outer](13.849989,-0.002628368){0.8}{108.04867}{355.90524}}
\psline[linecolor=black, linewidth=0.04](14.449969,-0.00748074)(14.8499565,-0.010715654)
\psline[linecolor=black, linewidth=0.04](12.425036,0.008896015)(13.618526,-0.80078256)
\psline[linecolor=black, linewidth=0.04](12.425036,0.008896015)(13.631467,0.7991651)
\psdots[linecolor=black, dotsize=0.225](8.025,0.0)
\rput{-55.289433}(2.8417542,5.4254575){\psarc[linecolor=black, linewidth=0.04, dimen=outer](6.6,0.0){0.8}{108.04867}{355.90524}}
\psline[linecolor=black, linewidth=0.04](6.0,0.0)(5.6,0.0)
\psline[linecolor=black, linewidth=0.04](8.025,0.0)(6.825,0.8)
\psline[linecolor=black, linewidth=0.04](8.025,0.0)(6.825,-0.8)
\rput{-180.46338}(16.049881,-0.047109026){\psdots[linecolor=black, dotsize=0.225](8.025036,0.008896015)}
\rput{-235.7528}(14.765935,-7.815632){\psarc[linecolor=black, linewidth=0.04, dimen=outer](9.449989,-0.002628368){0.8}{108.04867}{355.90524}}
\psline[linecolor=black, linewidth=0.04](10.04997,-0.00748074)(10.449957,-0.010715654)
\psline[linecolor=black, linewidth=0.04](8.025036,0.008896015)(9.218527,-0.80078256)
\psline[linecolor=black, linewidth=0.04](8.025036,0.008896015)(9.231466,0.7991651)
\end{pspicture}
}
\]
where he dashed oval indicates which part of the graph is contracted by $p_1$.
The pictorial expression of $p_2$ is similar.

We will use similar pictures as a language for free operads over
$\ggGrc$. As an illustration, 
here is a pictorial version
of~(\ref{Bude_Ovar_dalsich_pet_let prezidentem?}):
\begin{equation}
\label{dnes_jsem_objel_kanal}
\raisebox{-6em}{
\psscalebox{1.0 1.0} 
{
\begin{pspicture}(0,-1.4019446)(12.320392,1.4019446)
\rput[b](3.2,.3){$g$}
\rput[b](9.2,.3){$g$}
\psline[linecolor=black, linewidth=0.04](5.133333,-0.0019446583)(5.5333333,-0.0019446583)
\psdots[linecolor=black, dotsize=0.225](3.1333332,-0.0019446583)
\rput{-2.439889}(0.0016315024,0.072724074){\psarc[linecolor=black, linewidth=0.04, dimen=outer](1.7083331,-0.0019446583){0.8}{55.279263}{317.44205}}
\psline[linecolor=black, linewidth=0.04](1.1083331,-0.0019446583)(0.70833313,-0.0019446583)
\psellipse[linecolor=black, linewidth=0.04, linestyle=dashed, dash=0.17638889cm 0.10583334cm, dimen=outer](1.8999997,-0.0019446583)(1.9,1.4)
\rput[bl](6.0,-0.13527799){$\oplus$}
\psline[linecolor=black, linewidth=0.04](3.1333332,-0.0019446583)(1.9333332,0.79805535)
\psline[linecolor=black, linewidth=0.04](3.1333332,-0.0019446583)(1.9333332,-0.8019447)
\rput{-184.56499}(9.025159,-0.39291084){\psarc[linecolor=black, linewidth=0.04, dimen=outer](4.5204096,-0.01659256){0.8}{55.279263}{317.44205}}
\psline[linecolor=black, linewidth=0.04](3.0963898,0.03624849)(4.265899,-0.807699)
\psline[linecolor=black, linewidth=0.04](3.1333332,-0.0019446583)(4.333333,0.79805535)
\rput{-179.75348}(18.374098,0.030811006){\psdots[linecolor=black, dotsize=0.225](9.187082,-0.0043584616)}
\rput{-182.19337}(21.216465,-0.40032578){\psarc[linecolor=black, linewidth=0.04, dimen=outer](10.612064,0.002911593){0.8}{55.279263}{317.44205}}
\psline[linecolor=black, linewidth=0.04](11.212056,0.0059726685)(11.612051,0.008013386)
\rput{-179.75348}(20.840694,0.048701778){\psellipse[linecolor=black, linewidth=0.04, linestyle=dashed, dash=0.17638889cm 0.10583334cm, dimen=outer](10.4204,0.0019337495)(1.9,1.4)}
\psline[linecolor=black, linewidth=0.04](9.187082,-0.0043584616)(10.391149,-0.7982259)
\psline[linecolor=black, linewidth=0.04](9.187082,-0.0043584616)(10.382985,0.8017533)
\rput{-4.318467}(0.021902805,0.5873463){\psarc[linecolor=black, linewidth=0.04, dimen=outer](7.799949,0.003212673){0.8}{55.279263}{317.44205}}
\psline[linecolor=black, linewidth=0.04](9.22422,-0.042362634)(8.050421,0.7956073)
\psline[linecolor=black, linewidth=0.04](9.187082,-0.0043584616)(7.9911795,-0.8104702)
\psline[linecolor=black, linewidth=0.04](6.833333,-0.0019446583)(7.133333,-0.0019446583)
\rput[bl](0.33333313,0.19805534){$E$}
\rput[bl](0,0.99805534){$E$}
\rput[bl](11.6,0.19805534){$E$}
\rput[bl](12,0.99805534){$E$}
\rput[bl](12.5,-0.03){$.$}
\rput[bl](2.1333332,0.79805535){$a$}
\rput[bl](7.5333333,0.99805534){$a$}
\rput[bl](1.5333332,-0.6019447){$b$}
\rput[bl](7.733333,-0.6019447){$b$}
\rput[bl](9.533333,0.59805536){$c$}
\rput[bl](4.733333,0.99805534){$c$}
\rput[bl](4.333333,-0.6019447){$d$}
\rput[bl](10.533333,-0.6019447){$d$}
\end{pspicture}
}
}
\end{equation}
It features the souls of the relevant graphs in the sense of the following 

\begin{definition}
The {\em soul\/} of a graph $\Gamma$ is the graph obtained from $\Gamma$ by
amputating its legs.
\end{definition}

The $E$'s inside the dashed circles indicate the
decoration of the fiber represented by the subgraph inside
the circle, while the $E$'s outside the circles indicate the decoration of the
images. Thus the object on the left of~\eqref{dnes_jsem_objel_kanal}
represents the first summand of~(\ref{Bude_Ovar_dalsich_pet_let
  prezidentem?}) and the object on the right the second one. 
This description should be compared to the description of free
``classical'' operads in terms of trees with decorated vertices,
cf.~\cite[Section~II.1.9]{markl-shnider-stasheff:book}. 
Here we have graphs instead of trees and ``nests'' of subgraphs ordered by
inclusion in place of vertices.  
\end{example}

\begin{example}
\label{Spregner}
Using the same reasoning as in Examples
\ref{chteji_po_mne_abych_velel_grantu} and \ref{toho_snehu_je_moc}, we
can draw similar pictures describing $\Free^2(E)(X)$ for $X$ a graph
with two internal edges and two vertices with genera $g_1$ and $g_2$. 
Their souls are shown
in~\eqref{Zeman_nebo_Drahos? nahore}--\eqref{Zeman_nebo_Drahos? dole} below
\begin{subequations}
\begin{equation}
\label{Zeman_nebo_Drahos? nahore}
\raisebox{-7em}{
\psscalebox{1.0 1.0} 
{
\begin{pspicture}(0,-1.9412143)(13.26,1.9412143)
\rput{-90.73138}(4.696944,1.3197962){\psarc[linecolor=black, linewidth=0.04, dimen=outer](3.0,-1.6587857){2.6}{130.21674}{230.1}}
\psdots[linecolor=black, dotsize=0.225](1.0,-0.058785707)
\psdots[linecolor=black, dotsize=0.225](5.0,-0.058785707)
\psarc[linecolor=black, linewidth=0.04, linestyle=dashed, dash=0.17638889cm 0.10583334cm, dimen=outer](0.2,-2.1254523){0.26666668}{131.51372}{131.67642}
\pscustom[linecolor=black, linewidth=0.04, linestyle=dashed, dash=0.17638889cm 0.10583334cm]
{
\newpath
\moveto(16.133333,-1.1254523)
}
\pscustom[linecolor=black, linewidth=0.04, linestyle=dashed, dash=0.17638889cm 0.10583334cm]
{
\newpath
\moveto(12.533334,-2.4587858)
}
\pscustom[linecolor=black, linewidth=0.04, linestyle=dashed, dash=0.17638889cm 0.10583334cm]
{
\newpath
\moveto(16.133333,0.607881)
}
\rput{18.595245}(-0.299674,-1.0204285){\psarc[linecolor=black, linewidth=0.04, linestyle=dashed, dash=0.17638889cm 0.10583334cm, dimen=outer](2.9666667,-1.4254524){1.7666667}{15.0}{121.2717}}
\rput{-229.69476}(8.168101,-4.099964){\psarc[linecolor=black, linewidth=0.04, linestyle=dashed, dash=0.17638889cm 0.10583334cm, dimen=outer](5.0333333,-0.1587857){0.6333333}{89.0}{260.0}}
\rput{-229.69476}(1.5806634,-1.0495273){\psarc[linecolor=black, linewidth=0.04, linestyle=dashed, dash=0.17638889cm 0.10583334cm, dimen=outer](1.0333333,-0.1587857){0.6333333}{21.63039}{204.26073}}

\psarc[linecolor=black, linewidth=0.04, linestyle=dashed, dash=0.17638889cm 0.10583334cm, dimen=outer](3.0333333,-1.692119){3.1}{39.0}{143.0}
\rput{-90.0}(5.7254524,8.741215){\psarc[linecolor=black, linewidth=0.04, linestyle=dashed, dash=0.17638889cm 0.10583334cm, dimen=outer](7.233333,1.5078809){0.26666668}{131.51372}{131.67642}}
\rput{-270.98245}(4.489549,-1.4847707){\psarc[linecolor=black, linewidth=0.04, dimen=outer](3.0,1.5412143){2.6}{132.3}{230.1}}
\rput{-179.83253}(25.716059,4.203225){\psarc[linecolor=black, linewidth=0.04, linestyle=dashed, dash=0.17638889cm 0.10583334cm, dimen=outer](12.861101,2.0828211){0.26666668}{131.51372}{131.67642}}
\rput{-271.28363}(11.644145,-8.692223){\psarc[linecolor=black, linewidth=0.04, dimen=outer](10.266661,1.6078811){2.6}{130.21674}{230.1}}
\rput{-180.55226}(24.532843,-0.102453604){\psdots[linecolor=black, dotsize=0.225](12.266668,0.007890621)}
\rput{-180.55226}(16.533028,-0.06393689){\psdots[linecolor=black, dotsize=0.225](8.266668,0.007871618)}
\rput{-161.95702}(19.667742,5.8717227){\psarc[linecolor=black, linewidth=0.04, linestyle=dashed, dash=0.17638889cm 0.10583334cm, dimen=outer](10.299994,1.374548){1.7666667}{15.0}{121.2717}}
\rput{-50.247025}(2.8853579,6.368748){\psarc[linecolor=black, linewidth=0.04, linestyle=dashed, dash=0.17638889cm 0.10583334cm, dimen=outer](8.233335,0.10787146){0.6333333}{89.0}{260.0}}
\rput{-50.247025}(4.327428,9.44399){\psarc[linecolor=black, linewidth=0.04, linestyle=dashed, dash=0.17638889cm 0.10583334cm, dimen=outer](12.233335,0.107890464){0.6333333}{21.63039}{204.26073}}
\rput{-180.55226}(20.481997,3.1837165){\psarc[linecolor=black, linewidth=0.04, linestyle=dashed, dash=0.17638889cm 0.10583334cm, dimen=outer](10.233327,1.6412143){3.1}{39.0}{143.0}}
\rput{-91.53471}(12.133192,8.628233){\psarc[linecolor=black, linewidth=0.04, dimen=outer](10.266676,-1.5921189){2.6}{132.3}{230.1}}
\rput[bl](4.4,-1.2587857){$E$}
\rput[bl](5.2,-0.4587857){$E$}
\rput[bl](12.4,-0){$E$}
\rput[bl](13.0,0.5412143){$E$}
\psline[linecolor=black, linewidth=0.04](3.0,1.1412143)(3.0,0.7412143)
\psline[linecolor=black, linewidth=0.04](10.2,-0.8587857)(10.2,-1.2587857)
\psline[linecolor=black, linewidth=0.04](10.2,1.1412143)(10.2,0.7412143)
\psline[linecolor=black, linewidth=0.04](3.0,-0.8587857)(3.0,-1.2587857)
\rput[bl](2.0,0.35){$a$}
\rput[bl](2.0,-0.6587857){$c$}
\rput[bl](4.2,0){$b$}
\rput[bl](3.8,-0.6587857){$d$}
\rput[bl](9.0,0.9412143){$a$}
\rput[bl](11.2,0.9412143){$b$}
\rput[bl](9.2,-0.6587857){$c$}
\rput[bl](11.0,-0.6587857){$d$}
\rput[bl](6.4,-0.2587857){$\oplus$}
\end{pspicture}
}
}
\end{equation}
\begin{equation}
\label{Zeman_nebo_Drahos? dole}
\raisebox{-6em}{
\psscalebox{1.0 1.0} 
{
\begin{pspicture}(0,-1.4000019)(12.4,1.4000019)
\rput(-.2,0){
\psline[linecolor=black, linewidth=0.04](4.8,-0.0038784079)(5.2,-0.0038784079)
\psdots[linecolor=black, dotsize=0.225](2.8,-0.0038784079)
\rput[bl](6.266667,-0.13721174){$\oplus$}
\rput{-184.56499}(8.359396,-0.3702423){\psarc[linecolor=black, linewidth=0.04, dimen=outer](4.1870766,-0.01852631){0.8}{55.279263}{317.44205}}
\psline[linecolor=black, linewidth=0.04](2.7630565,0.03431474)(3.9325662,-0.8096328)
\psline[linecolor=black, linewidth=0.04](2.8,-0.0038784079)(4.0,0.7961216)
\rput{-179.75348}(18.17405,0.039097566){\psellipse[linecolor=black, linewidth=0.04, linestyle=dashed, dash=0.17638889cm 0.10583334cm, dimen=outer](9.087067,0.0)(1.7,1.0)}
\rput[bl](0.0,0.19612159){$E$}
\rput[bl](3.0,-1.0038784){$E$}
\psdots[linecolor=black, dotsize=0.225](0.6,-0.0038784079)
\psline[linecolor=black, linewidth=0.04](0.6,-0.0038784079)(2.8,-0.0038784079)
\psline[linecolor=black, linewidth=0.04](1.8,0.19612159)(1.8,-0.2038784)
\psline[linecolor=black, linewidth=0.04](12.0,-0.0038784079)(12.4,-0.0038784079)
\psdots[linecolor=black, dotsize=0.225](10.0,-0.0038784079)
\rput{-184.56499}(22.736555,-0.9432886){\psarc[linecolor=black, linewidth=0.04, dimen=outer](11.387076,-0.01852631){0.8}{55.279263}{317.44205}}
\psline[linecolor=black, linewidth=0.04](9.963057,0.03431474)(11.132566,-0.8096328)
\psline[linecolor=black, linewidth=0.04](10.0,-0.0038784079)(11.2,0.7961216)
\rput[bl](8.2,-0.6038784){$E$}
\rput[bl](10.4,0.9961216){$E$}
\psdots[linecolor=black, dotsize=0.225](7.8,-0.0038784079)
\psline[linecolor=black, linewidth=0.04](7.8,-0.0038784079)(10.0,-0.0038784079)
\psline[linecolor=black, linewidth=0.04](9.0,0.19612159)(9.0,-0.2038784)
\rput{-179.75348}(7.7740974,0.016724302){\psellipse[linecolor=black, linewidth=0.04, linestyle=dashed, dash=0.17638889cm 0.10583334cm, dimen=outer](3.8870666,0.0)(1.5,1.4)}
\rput[bl](1.2,0.19612159){$a$}
\rput[bl](8.4,0.19612159){$a$}
\rput[bl](2.2,0.19612159){$b$}
\rput[bl](9.4,0.19612159){$b$}
\rput[bl](4.0,0.9961216){$u$}
\rput[bl](4.2,-0.6038784){$v$}
\rput[bl](11.4,0.9961216){$u$}
\rput[bl](11.2,-0.6038784){$v$}
}
\end{pspicture}
}
}
\end{equation}
\end{subequations}
where, in order to ease the interpretation, we did not show the genera of the vertices.
The picture in~\eqref{Zeman_nebo_Drahos? nahore} represents 
an analog of~(\ref{Bude_Ovar_dalsich_pet_let prezidentem?}):
\begin{align*}
\Free^2(E)(X) \cong& \
\big\{
E[\nu(*,c|a,b|d,*)^{g_1|g_2}] \ot E[\xi(*|c,d)^{g_1+g_2}]\big\}
\\
\oplus & \ \big\{ E[\nu(*,a|c,d|b,*)^{g_1|g_2}] \ot 
E[\xi(* |a,b)^{g_1+g_2}]\big\}
\end{align*}
in which the notation $\nu(*,c|a,b|d,*)$ resp.\ $\nu(*,a|c,d|b,*)$ refers to the
graph in~\eqref{dnes_na_Rusalku right}. 
The picture in~\eqref{Zeman_nebo_Drahos? dole} symbolizes
\begin{align*}
\Free^2(E)(X) \cong& \ 
\big\{
E[\xi(*,b|u,v)^{g_1}] \ot E[\nu(*|a,b|*)^{g_1|g_2+1}]\big\}
\\
\oplus& \ \big\{ E[\nu(*|a,b|u,v,*)^{g_1|g_2}] 
\ot E[\xi(* |u,v)^{g_1+g_2}]\big\}\,.
\end{align*}

The last relevant case is when $X$ is an ordered  graph with two
internal edges and three vertices with genera $g_1,g_2$ and $g_3$. 
The situation is portrayed in
\begin{equation}
\label{Mozna_bude_Ovara_volit_i_Jana.}
\raisebox{-4.5em}{\rule{0em}{1em}}
\raisebox{-4em}{
\psscalebox{1.0 1.0} 
{
\begin{pspicture}(0,-.9412143)(13.052372,1)
\psdots[linecolor=black, dotsize=0.225](0.8,-0.058785707)
\rput[bl](1.4,-0.6587857){$E$}
\rput[bl](0.3,0.74){$E$}
\psdots[linecolor=black, dotsize=0.225](3.2,-0.058785707)
\psarc[linecolor=black, linewidth=0.04, linestyle=dashed, dash=0.17638889cm 0.10583334cm, dimen=outer](0.2,-2.1254523){0.26666668}{131.51372}{131.67642}
\pscustom[linecolor=black, linewidth=0.04, linestyle=dashed, dash=0.17638889cm 0.10583334cm]
{
\newpath
\moveto(16.133333,-1.1254523)
}
\pscustom[linecolor=black, linewidth=0.04, linestyle=dashed, dash=0.17638889cm 0.10583334cm]
{
\newpath
\moveto(12.533334,-2.4587858)
}
\pscustom[linecolor=black, linewidth=0.04, linestyle=dashed, dash=0.17638889cm 0.10583334cm]
{
\newpath
\moveto(16.133333,0.607881)
}

\rput{-90.0}(5.7254524,8.741215){\psarc[linecolor=black, linewidth=0.04, linestyle=dashed, dash=0.17638889cm 0.10583334cm, dimen=outer](7.233333,1.5078809){0.26666668}{131.51372}{131.67642}}
\rput{-179.83253}(25.716059,4.203225){\psarc[linecolor=black, linewidth=0.04, linestyle=dashed, dash=0.17638889cm 0.10583334cm, dimen=outer](12.861101,2.0828211){0.26666668}{131.51372}{131.67642}}
\psdots[linecolor=black, dotsize=0.225](5.6,-0.058785707)
\psline[linecolor=black, linewidth=0.04](0.8,-0.058785707)(5.6,-0.058785707)
\psdots[linecolor=black, dotsize=0.225](7.2,-0.058785707)
\rput[bl](12.4,0.7412143){$E$}
\rput[bl](11.0,-0.6587857){$E$}
\psdots[linecolor=black, dotsize=0.225](9.6,-0.058785707)
\psdots[linecolor=black, dotsize=0.225](12.0,-0.058785707)
\psline[linecolor=black, linewidth=0.04](7.2,-0.058785707)(12.0,-0.058785707)
\rput[bl](6.266667,-0.2587857){$\oplus$}
\psellipse[linecolor=black, linewidth=0.04, linestyle=dashed, dash=0.17638889cm 0.10583334cm, dimen=outer](2.0,-0.058785707)(1.48,1)
\psellipse[linecolor=black, linewidth=0.04, linestyle=dashed, dash=0.17638889cm 0.10583334cm, dimen=outer](10.8,-0.058785707)(1.48,1)
\psline[linecolor=black, linewidth=0.04](4.4,0.1412143)(4.4,-0.2587857)
\psline[linecolor=black, linewidth=0.04](8.4,0.1412143)(8.4,-0.2587857)
\psline[linecolor=black, linewidth=0.04](10.8,0.1412143)(10.8,-0.2587857)
\psline[linecolor=black, linewidth=0.04](2.0,0.1412143)(2.0,-0.2587857)
\rput[bl](7.8,0.1412143){$a$}
\rput[bl](1.4,0.1412143){$a$}
\rput[bl](2.4,0.1412143){$b$}
\rput[bl](8.8,0.1412143){$b$}
\rput[bl](3.8,0.1412143){$c$}
\rput[bl](10.2,0.1412143){$c$}
\rput[bl](5.0,0.1412143){$d$}
\rput[bl](11.4,0.1412143){$d$}
\end{pspicture}
}
}
\end{equation}
where again the genera of the vertices are not shown.
The resulting formula~is
\begin{align*}
\Free^2(E)(X) \cong& \
\big\{
E[\nu(*|a,b|c,*)^{g_1|g_2}] \ot E[\nu(* |c,d|*)^{g_1+g_2|g_3}]\big\}
\\
\oplus &\   
\big\{ E[\nu(*,b|c,d|*)^{g_2|g_3}] \ot E[\nu(* 
|a,b|*)^{g_1|g_2+g_3}]\big\}\,.
\end{align*}
\end{example}

\begin{remark}
\label{Grete vynechava zapalovani.}
In order to appreciate the
advantages of our approach,  we suggest to compare  
the simple and self-explaining pictures
in~\eqref{dnes_jsem_objel_kanal}--\eqref{Mozna_bude_Ovara_volit_i_Jana.}
with Figures~2 and~3 of~\cite{Ward} expressing the axioms of
modular operads as algebras over a colored operad. 
\end{remark}

\def\edg{{\rm edg}}
The observations in Examples \ref{toho_snehu_je_moc}
and~\ref{Spregner} easily generalize to descriptions of
isomorphism classes of labeled towers in $\lTw(\Gamma)$ for an
arbitrary graph $\Gamma \in \ggGrc$. Since we will be primarily
interested in free operads generated by {\em binary\/}
collections,~i.e.\   $1$-connected collections
that are trivial on graphs with more than one internal edge,
we will consider only towers whose
associated fiber sequence consists of graphs with one internal
edge. Let 
\[
\Gamma \stackrel {\tau_1} \longrightarrow \Gamma_1  \stackrel  {\tau_2} \longrightarrow
\Gamma_2\stackrel  {\tau_2} \longrightarrow 
\cdots  \stackrel  {\tau_{k-1}}   \longrightarrow  \Gamma_{k-1}  
\]
be such a tower. By the definition of graph morphisms, one has the
associated sequence 
\begin{equation}
\label{ta_zruda_bude_prezident}
\edg(\Gamma) \supset \edg(\Gamma_1) \supset \cdots \supset \edg(\Gamma_{k-1}) 
\end{equation}
of inclusions of the sets of internal edges. Since the cardinalities
of the sets in~(\ref{ta_zruda_bude_prezident}) decrease by one,  there 
is an obvious
one-to-one correspondence between
sequences~(\ref{ta_zruda_bude_prezident}) and linear orders on
$\edg(\Gamma)$ such that
$x > y$  if $y\in \edg(\Gamma_i)$, $x \not\in  \edg(\Gamma_i)$ for some
$i$, $1 \leq i \leq k-1$. We~formulate:

\begin{proposition}
\label{uz_melo_prvni_cislo_davno_vyjit}
The isomorphism classes of labeled towers in $\ltw(\Gamma)$ whose
associated fiber sequence consists of graphs with one internal edge
are in one-to-one correspondence with linear orders on
$\edg(\Gamma)$ modulo the relation $\bowtie$ that
interchanges two adjacent edges without a common vertex in $\Gamma$.  
\end{proposition}

\begin{example}
One has two isomorphism classes of towers for the graph 
in~\eqref{pozitri_do_Srni}. In the notation
of~\eqref{dnes_jsem_objel_kanal}, let $x$ be the edge $\{a,b\}$ and
$y$ the edge  $\{c,d\}$.  Then the picture on the left in that display
corresponds to the order $x > y$ ($x$ is contracted first), the one on
the right to~$y > x$.
\end{example}

\begin{proof}[Proof of Proposition~\ref{uz_melo_prvni_cislo_davno_vyjit}]
Using the same arguments as in Examples \ref{toho_snehu_je_moc}
and~\ref{Spregner} we show that each tower can be replaced
by an isomorphic one all of whose morphisms are pure
contractions of internal edges, in the sense of~\cite[Definition~3.3]{part1}. 
Such towers are determined by the order in which the edges are
contracted. The relation \ $\bowtie$ \ reflects morphisms
of towers of the second type introduced in Section~\ref{zitra_letim_do_Pragy}.
\end{proof}

\begin{theorem}
\label{v_nedeli_si_snad_zaletam_s_NOE} 
The terminal $\ggGrc$-operad $\termGr$ having $\termGr(\Gamma) := \bfk$
for each $\Gamma \in \ggGrc$
and constant composition laws is  binary quadratic.
\end{theorem} 

\begin{proof}
Let us define a collection  $E \in \Coll$ by
\begin{equation}
\label{vcera}
E[\Gamma] :=
\begin{cases}
\bfk & \hbox {if $\Gamma$ has exactly one internal edge}
\\
0 & \hbox {otherwise}  
\end{cases}
\end{equation}
with the constant $\QV(e)$-presheaf structure.
As we already noticed,
the quadratic part $\Free^2(E)$ of the free operad may be nontrivial only
for graphs with precisely two internal edges, i.e.~those analyzed
in Examples \ref{toho_snehu_je_moc} and~\ref{Spregner}. 
For $X$ as in~\eqref{pozitri_do_Srni} and $E$ in~(\ref{vcera}), 
formula~(\ref{Bude_Ovar_dalsich_pet_let prezidentem?})~gives
\[
\Free^2(E)(X) \cong \{\bfk\ot \bfk\} \oplus
\{\bfk\ot \bfk\}
\]
which is a two-dimensional vector space with the basis 
\begin{equation}
\label{zni_mi_hlavou_Joe_Karafiat}
b_1^1 : =\{ 1\ot 1\} \oplus \{0 \ot 0\} \ \hbox { and } \ 
b_2^1 : = \{0\ot 0\} \oplus \{1 \ot 1\}.
\end{equation}
For the situations portrayed
in~\eqref{Zeman_nebo_Drahos? nahore}--\eqref{Zeman_nebo_Drahos? dole}
and~\eqref{Mozna_bude_Ovara_volit_i_Jana.} we get  similar 
spaces, with bases $(b^t_1,b^t_2)$, $2 \leq t \leq 4$. 
Let us denote
\begin{equation}
\label{Krammer}
r_1:= b_1^1 - b_2^1, \ r_2:= b_1^2 - b_2^2, \  r_3:= b_1^3 - b_2^3
\hbox { and } \ 
r_4:= b_1^4 - b_2^4
\end{equation}
We define
$R(X) $ to be the subspace of  $\Free^2(E)(X)$ spanned by
\hfill\break 
-- $r_1$ for $X$ with soul as in~\eqref{dnes_jsem_objel_kanal}, \hfill\break
-- $r_2$ for $X$ with soul as  in~\eqref{Zeman_nebo_Drahos? nahore}, \hfill\break
--   $r_3$ for $X$ with soul as  in~\eqref{Zeman_nebo_Drahos? dole}, and \hfill\break
-- $r_4$ for $X$ with soul as
in~\eqref{Mozna_bude_Ovara_volit_i_Jana.}, \hfill\break 
while $R(X) :=0$ if $X$ does not have exactly two internal edges. 
We are going to prove that
\begin{equation}
\label{dnes_jsem_predsedal_vedecke_rade}
\termGr \cong \Free(E)/(R)
\end{equation}
for the sub-presheaf $R = \{R(X)\}_{X \in \ggGrc}$ of $\Free(E)$ defined above.

By Proposition~\ref{uz_melo_prvni_cislo_davno_vyjit}
combined with formula~(\ref{dnes_bude_vedro}), the vector space
$\Free(E)(\Gamma)$ is spanned 
by the set of total orders on $\edg(\Gamma)$ modulo the relation $\bowtie$ that
interchanges any two edges $x,y \in \edg(\Gamma)$ that {\em do
  not\/} share a common vertex in $\Gamma$.

All possible relative configurations  of edges $x,y$ that {\em do share\/} a
common vertex are in~\eqref{dnes_jsem_objel_kanal}--\eqref{Mozna_bude_Ovara_volit_i_Jana.}.
Relations in~(\ref{Krammer}) guarantee that two orders that differ by
the interchange $x \leftrightarrow y$ agree in the
quotient~(\ref{dnes_jsem_predsedal_vedecke_rade}). We conclude that
all orders on $\edg(\Gamma)$ are mutually equivalent modulo $(R)$, so
$\Free(E)/(R)(\Gamma) \cong \bfk$ as required.  
\end{proof}

\begin{proposition}
\label{pozitri_turnaj_v_Patku}
Algebras over the terminal $\ggGrc$-operad $\termGr$ are modular operads.
\end{proposition}

\def\oM{{\mathcal M}}
\begin{proof}
The key ingredients of the proof are
presentation~(\ref{dnes_jsem_predsedal_vedecke_rade}) together 
with Proposition~6.13 of~\cite{part1}
which describes
$\termGr$-algebras as morphisms  to the endomorphism operad.
We start by determining what the underlying
collection
\[
V = \{V_c \ | \ c \in \pi_0(\ttO)\}
\]
of the endomorphism operad is in this case.

We noticed in
Example~\ref{Snad_se_ta_Achylovka_trochu_lepsi.} that the 
local terminal objects of
$\ggGrc$ are the $n$-corollas  $c(\sigma)^g$ with the vertex of
genus $g$ and the local order given by a permutation~$\sigma \in \Sigma_n$. The chosen local terminal
objects are the $n$-corollas  $c_n^g:=c(\id_n)^g$, $n,g \in \bbN$.
Therefore the set $\pi_0(\ggGrc)$ is indexed by pairs $(n;g)$ of natural
numbers consisting of an ``arity'' $n$ and a ``genus''~$g$, i.e.
\[
\pi_0(\ggGrc) = \big\{(n;g) \  \vrt \  (n;g)   \in  \bbA \times \bbA \big\}.
\]
The underlying collection of the endomorphism operad is thus a family
\[
\oM = \big\{\oM(n;g) \in \Vect\  \vrt \  (n;g)   \in  \bbA \times \bbA \big\}.
\]
The actions
$
u: V_{\pi_0(s_1(u))} \to  V_{\pi_0(u)}
$
of the groupoid of local terminal objects in this particular case
give rise to actions of the symmetric group $\Sigma_n$ on each
$\oM(n;g)$. We recognize $\oM$ as the skeletal version of a~modular
module recalled in Appendix~\ref{a1}.
Proposition~6.13 of~\cite{part1} now  identifies 
$\termGr$-algebras with~morphisms
\begin{equation}
\label{nevydrzim_dele_nez_tri_dny}
a: \Free(E)/(R)  \to \End_\oM,
\end{equation}
where $E$ is as in~(\ref{vcera}) and $R$ is spanned by
relations~(\ref{Krammer}). 

\def\xxi{\circ}
\newcommand{\ooo}[2]{\sideset{_{#1}}{_{#2}}{\mathop{\circ}}}

By Proposition~\ref{Za_14_dni_LKDL}, the unital
operad $\End_\oM$
determines a $\QV(e)$-presheaf $\oEnd_\oM$.
Although $\End_\oM$ is not strictly  unital, 
morphism~(\ref{nevydrzim_dele_nez_tri_dny}) is still uniquely determined by a
map $\tilde a : E \to  \oEnd_\oM$ of  $\QV(e)$-presheaves given by
a family
\begin{equation}
\label{slepit_vrtulku_nebude_snadne}
\tilde a_{[\Gamma]} : E[\Gamma] \to \oEnd_\oM([\Gamma]),\ [\Gamma] \in \QV(e).
\end{equation}

By definition, the generating collection $E$ is
supported on graphs with one internal
edge portrayed in~\eqref{dnes_na_Rusalku left} and~\eqref{dnes_na_Rusalku right},
whose souls are:
\begin{equation} 
\label{hluboka_deprese_z_voleb}
\raisebox{-4em}{
\psscalebox{1.0 1.0} 
{
\begin{pspicture}(0,-1.0434394)(7.1108775,1.0434394)
\rput[b](.15,0.2){$g_1$}
\rput[b](2.3,0.2){$g_2$}
\rput[b](4.7,0.2){$g$}
\rput[b](7.3,-.18){.}
\psdots[linecolor=black, dotsize=0.225](2.3108778,-0.11656067)
\psdots[linecolor=black, dotsize=0.225](0.110877685,-0.11656067)
\psline[linecolor=black, linewidth=0.04](0.110877685,-0.11656067)(2.3108778,-0.11656067)
\psline[linecolor=black, linewidth=0.04](1.3108777,0.08343933)(1.3108777,-0.31656066)
\rput[bl](0.71087766,0.08343933){$a$}
\rput[bl](1.7108777,0.08343933){$b$}
\psline[linecolor=black, linewidth=0.04](6.710878,-0.11656067)(7.1108775,-0.11656067)
\psdots[linecolor=black, dotsize=0.225](4.710878,-0.11656067)
\rput{-184.56499}(12.166121,-0.7473357){\psarc[linecolor=black, linewidth=0.04, dimen=outer](6.0979543,-0.13120857){0.8}{55.279263}{317.44205}}
\psline[linecolor=black, linewidth=0.04](4.6739345,-0.07836752)(5.843444,-0.92231506)
\psline[linecolor=black, linewidth=0.04](4.710878,-0.11656067)(5.9108777,0.6834393)
\rput[bl](5.9108777,0.8834393){$u$}
\rput[bl](5.9108777,-0.71656066){$v$}
\end{pspicture}
}
}
\end{equation}
The operations $\tilde a_{[\Gamma]}$ may therefore
be nontrivial only for graphs of this form.

Let us analyze the
operation~\eqref{slepit_vrtulku_nebude_snadne} 
induced by the virtual isomorphism class of the graph 
$\Gamma := \xi(\Rada \lambda1k | \lambda_{k+1},\lambda_{k+2})$
in~\eqref{dnes_na_Rusalku left}. One clearly has
$\pi_0(s_1(\Gamma)) = (k\+2;g)$  and $\pi_0(\Gamma) = (k;g\+1)$, therefore
$\tilde a_{[\Gamma]}$ is by~(\ref{V_sobotu_budu_letat_vycvik_vlekare.}) a~map
\[
\tilde a_{[\Gamma]} : E[\Gamma]  
= \bfk \longrightarrow    \colim_{\tilde\Gamma}    \End_\oM(\tilde\Gamma) \cong
 \colim_{\sigma \in\Sigma_k}
 \Vect\big(\oM(k\+2;g),\oM(k;g\+1)_\sigma\big),
\]
where the first colimit is taken over all $\tilde\Gamma$'s virtually isomorphic to
$\Gamma$, where 
 $\sigma = (\Rada \sigma1k)$ and  where $\oM(k;g\+1)_\sigma$ is the copy of $\oM(k;g\+1)$ corresponding
to the graph 
\[
\xi(\Rada \sigma 1k | \lambda_{k+1},\lambda_{k+2}),
\]
which is virtually isomorphic to $\xi(\Rada \lambda1k | \lambda_{k+1},\lambda_{k+2})$.
The map $\tilde a_{[\Gamma]}$ is clearly determined by 
\[
\tilde a_{[\Gamma]}(1) : \oM(k\+2;g)\to \colim_{\sigma\in \Sigma_k} \oM(k;g\+1)_\sigma
\]
which is the same as a collection of morphisms
\[
\xxi^\sigma_{uv} :\oM(k\+2;g) 
\longrightarrow \oM(k;g\+1),\ u : = \lambda_{k+1},\ v :=
\lambda_{k+2},\
\sigma \in \Sigma_k,
\]
satisfying
\[
\xxi^{\sigma \delta}_{uv}(x) = \sigma \xxi^{\delta}_{uv}(x),\
x \in   \oM(k\+2;g), \ \sigma, \delta \in \Sigma_k.
\]

The operation $\xxi_{uv} := \xxi^{\id_k}_{uv}$ is the
skeletal version 
of the contraction~(\ref{Galway}).
The identity $\xxi_{uv} = \xxi_{vu}$ follows from the
$\Sigma_2$-symmetry of the graph~$\Gamma$. 
In exactly the same manner, the graph in~\eqref{dnes_na_Rusalku right} gives rise to the operations 
in~(\ref{v_Galway1}). 

The map  $\tilde a$ determines a
morphism~(\ref{nevydrzim_dele_nez_tri_dny}) if and only if it sends
the generators~(\ref{Krammer}) of $R$ to $0$. The vanishing  $\tilde
a(r_i) = 0$ for $1 \leq i \leq 4$  corresponds to the 
remaining axiom of modular~operads:

\begin{tabular}{ll}
\hbox{Axiom~(\ref{eq:38}) corresponds to relation $r_2$}, &
\hbox{Axiom~(\ref{eq:39}) corresponds to relation $r_3$}, 
\\
\hbox{Axiom~(\ref{Dnes_s_Jaruskou_k_Pakouskum}) corresponds to relation
$r_4$,}&
\hbox{Axiom~(\ref{eq:33}) corresponds to relation $r_1$.}\rule{0em}{1em}
\end{tabular}

\noindent 
This finishes the proof.
\end{proof}

\begin{theorem}
\label{vcera_jsem_podlehl}
The Koszul dual of the operad $\termGr$, denoted 
$\oddGr$, is the operad whose algebras are odd modular operads.
\end{theorem}

\begin{proof}
The Koszul dual  $\oddGr := \termGr^!$ is, by definition, generated by
the collection
\[
\susp E^* :=
\begin{cases}
\susp\, \bfk & \hbox {if $\Gamma$ has exactly one internal edge}
\\
0 & \hbox {otherwise} . 
\end{cases}
\]
We get the similar type of generators $d^i_1,d^i_2$, $1 \leq i \leq
4$, for $\Free^2(\susp E^*)$ as in the proof of
Theorem~\ref{v_nedeli_si_snad_zaletam_s_NOE} except that now they will be in
degree $2$. The pairing~(\ref{jdu_si_lepit_142}) in this particular
case is given by
\[
\langle \ b^i_k\ |\  d^j_l\ \rangle =
\begin{cases}
1 & \hbox {if $i=j$, $k=l$}  
\\
0 & \hbox {otherwise}. 
\end{cases}
\]
Therefore the annihilator $R^\perp$ of the relations~(\ref{Krammer})
is spanned by 
\[
o_1:= d_1^1 + d_2^1, \ o_2:= d_1^2 + d_2^2, \  o_3:= d_1^3 + d_2^3
\hbox { and } \ 
o_4:= d_1^4 + d_2^4.
\]
Repeating the arguments in the proof of
Theorem~\ref{v_nedeli_si_snad_zaletam_s_NOE} 
we identify algebras  
over $\Free(\susp E^*)/(R^\perp)$ with odd modular operads whose
definition is recalled in Appendix~\ref{a1}.
\end{proof}

\begin{remark}
\label{zitra_vycvik_vlekare}
As observed in
\cite[Example~4.19]{part1}, the category
$\ggGrc$ is similar to the category of graphs of 
\cite[\S 2.15]{getzler-kapranov:CompM98}. The difference is the
presence of the local orders of graphs in $\ggGrc$ manifested e.g.\ by 
the fact that, while the category in 
\cite[\S 2.15]{getzler-kapranov:CompM98} has only one local terminal object for each
arity $n$ and genus $g$, the local terminal objects in $\ggGrc$ are
indexed by \hbox {$n$, $g$} and by a permutation $\sigma \in \Sigma_n$,
cf.~Example~\ref{Snad_se_ta_Achylovka_trochu_lepsi.}.
The category of
operads over the operadic category $\ggGrc$ is however equivalent to
the category of hyperoperads in
the sense of~\cite[\S 4.1]{getzler-kapranov:CompM98}. Moreover, there
is a canonical isomorphism between the category of algebras for a
$\ggGrc$-operad and the category of algebras for the corresponding hyperoperad.

\def\det{{\rm det}}
This relation enables one to compare the operad  $\oddGr$ of
Theorem~\ref{vcera_jsem_podlehl} to
a similar object considered in~\cite{getzler-kapranov:CompM98}.
Recall that the {\em determinant\/} $\det(S)$ of a finite set~$S$ 
is the top-dimensional piece of the  exterior
(Grassmann) algebra generated by the elements of $S$ placed in 
degree~$+1$. In particular,  $\det(S)$ is a one-dimensional vector space
concentrated in degree $k$,  with
$k$ the cardinality of $S$. Mimicking the arguments in the second
half of the proof of Theorem~\ref{v_nedeli_si_snad_zaletam_s_NOE} one
can establish that  $\oddGr(\Gamma) \cong \det(\edg(\Gamma))$, the
determinant of the set of internal edges of $\Gamma$. This relates
$\oddGr$ directly to the {\em dualizing cocycle\/} of
\cite[\S 4.8]{getzler-kapranov:CompM98}, 
cf.~also Example ~II.5.52 of \cite{markl-shnider-stasheff:book}. 
\end{remark}

\section{Other operad-like structures}
\label{Asi_pojedu_vecer.}

In this section we analyze other operad-like structures whose pasting
schemes are obtained from the basic
operadic category  $\Gr$ of graphs by means of the iterated Grothendieck
construction.
For all these categories the properties 
$\UFB$ and $\SGrad$ can be easily checked
``manually.''  By the
reasoning from the beginning of Section~\ref{zitra_budu_pit_na_zal} 
they are rigid and fulfill \hbox{\AA.}

\def\CGr{{\Tr}}\def\termCGr{{\sf 1}_{\CGr}}
\def\twxxi{\bullet}
\def\bl{\big(} \def\br{\big)} 
 \def\({{\hbox{$(\!($}}}\def\){{\hbox{$)\!)$}}}
\newcommand{\twooo}[2]{\sideset{_{#1}}{_{#2}}{\mathop{\bullet}}}
\newcommand{\twoooprime}[2]{\sideset{_{#1}}{_{\hskip -.2em#2}}{\mathop{{\bullet'}}}}
\newcommand{\twoxiprime}[1]{{\bullet'}_{\hskip -.2em#1}} 
\def\longiso{{\redukce{$\, \stackrel\cong\lra\, $}}}
\newcommand{\oxitl}[1]{\underline{\hbox{\scriptsize
      $\blacklozenge$}}_{\, #1}}
\newcommand{\twoootl}[2]{\sideset{_{#1}}{_{\, #2}}{\mathop{\overset{L}{\rule{0em}{.4em}\redukce{${\underline\bullet}$}}}}}
\newcommand{\twoootr}[2]{\sideset{_{#1}}{_{\, #2}}{\mathop{\overset{R}{\rule{0em}{.4em}\redukce{${\underline\bullet}$}}}}}
\newcommand{\ooo}[2]{\sideset{_{#1}}{_{#2}}{\mathop{\circ}}}
\def\xxi{\circ}\def\ush{{\rm uSh}}
\def\Chain{{\tt Vect}}
\def\oddCGr{{\mathfrak {K}}_\CGr}

\subsection{Cyclic operads}
Cyclic operads introduced in~\cite{getzler-kapranov:CPLNGT95} 
are, roughly speaking, modular operads without the
genus grading and contractions~(\ref{Galway}). Explicitly, 
a~cyclic operad is a~functor $\oC : \Fin  \to
\Vect$ along with operations
\begin{equation}
\label{v_Galway_jsem_byl}
\ooo ab:\oC\big(S_1 \sqcup \stt a \big)
\otimes \oC\big(S_2\sqcup \stt b\big)  
\longrightarrow \oC ( S_1\sqcup S_2)
\end{equation}
indexed by disjoint finite 
sets $S_1$, $S_2$ and symbols $a,b$.  These operations 
satisfy axioms~(\ref{piji_caj_z_Jarcina_hrnku}), (\ref{eq:24})
and~(\ref{Dnes_s_Jaruskou_k_Pakouskum}) of modular operads (without
the genus grading).
Let $\CGr$ be the full subcategory of $\Gr$ consisting of graphs of
genus zero
whose geometric realizations are contractible, i.e.~which are~trees.  
The local terminal objects of $\Tr$ are corollas $c(\sigma)$, $\sigma \in
\Sigma_n$, as in~\eqref{Dasa_Vokata} but without the genus
labeling the vertex. 
The chosen local terminal objects are corollas $c_n := c(\id_n)$,~$n \in \bbN$.

\begin{theorem}
\label{Zitra_letim_do_Bari.}
The terminal $\CGr$-operad $\termCGr$ is binary quadratic. Its algebras are
cyclic operads. Its Koszul dual $\oddCGr :=\termCGr^!$ is the operad 
whose algebras are anticyclic operads.
\end{theorem}

Anticyclic operads introduced in 
\cite[\S 2.11]{getzler-kapranov:CPLNGT95} are ``odd'' versions of cyclic
operads, see also \cite[Definition~II.5.20]{markl-shnider-stasheff:book}.
Due to the absence of the operadic units in our setup, 
  the category of anticyclic operads is however isomorphic to the category of
ordinary cyclic operads, via the isomorphism given by the suspension of
the underlying collection.

\begin{proof}[Proof of Theorem~\ref{Zitra_letim_do_Bari.}]
The proof is a simplified version of calculations in
Section~\ref{zitra_budu_pit_na_zal}. The soul of the only graph in $\CGr$ with
one internal edge
is the one on the left of~\eqref{hluboka_deprese_z_voleb} (without the
genera, of course), the
corresponding operation is~(\ref{v_Galway_jsem_byl}). The souls of the 
only graphs in $\CGr$ with two internal edges are portrayed
in~\eqref{Mozna_bude_Ovara_volit_i_Jana.}. 
Let $E$ be the restriction of the collection~(\ref{vcera}) to the
virtual isomorphism classes of trees in $\CGr$. If $R$
denotes the subspace of $\Free^2(E)$ spanned by $r_2$
in~(\ref{Krammer}), 
then $\termCGr
\cong \Free(E)/(R)$. The arguments are the same as in the proof of
Theorem~\ref{v_nedeli_si_snad_zaletam_s_NOE}. With the material of
Section~\ref{zitra_budu_pit_na_zal} at hand, the identification of
$\termCGr$-algebras with cyclic operads is~immediate.

Algebras over $\oddCGr$ can be analyzed in the same
way as 
$\oddGr$-algebras in the proof of Theorem~\ref{vcera_jsem_podlehl}. 
$\oddCGr$-algebras posses degree $+1$ operations 
\begin{equation}
\label{v_Galway_jsem_byl_jednou}
\twooo ab:\oC\big(S_1 \sqcup \stt a \big)
\otimes \oC\big(S_2\sqcup \stt b\big)  
\longrightarrow \oC ( S_1\sqcup S_2)
\end{equation}
satisfying non-genus graded variants 
of~(\ref{neni}),~(\ref{kdy_zacnu_byt_rozumny})
and~(\ref{neni1}). The level-wise suspension $\susp \oC$ with
operations 
\[
\ooo ab:\susp \oC\big(S_1 \sqcup \stt a \big)
\!\otimes\! \susp \oC\big(S_2\sqcup \stt b\big) \longrightarrow  
\susp \oC ( S_1\sqcup S_2)
\]
defined as the composite
\[
\susp \oC\big(S_1 \sqcup \stt a \big)
\!\otimes\! \susp \oC\big(S_2\sqcup \stt b\big)  \stackrel{\susp \ot
  \susp}\longrightarrow  \oC\big(S_1 \sqcup \stt a \big)
\!\otimes\! \oC\big(S_2\sqcup \stt b\big) \stackrel{\twooo ab}\longrightarrow
 \oC ( S_1\sqcup S_2) \stackrel{\susp}\to \susp \oC ( S_1\sqcup S_2)
\]
can easily be shown to be an anticyclic 
operad~\cite[Definition~II.5.20]{markl-shnider-stasheff:book}.
\end{proof}

As in Remark~\ref{zitra_vycvik_vlekare}, one may observe  that
$\oddCGr(T)$ equals the determinant of the set of internal edges
of the tree $T$. Our description of anticyclic operads as
$\oddCGr$-algebras is therefore parallel to the definition as
${\mathbb T}_-$-algebras given
in~\cite[page~178]{getzler-kapranov:CPLNGT95}.  

\def\Tre{{\CGr}}
\def\RTre{\RTr}
\def\ssRTre{{\RTr}}

\subsection{Ordinary operads}
\label{Cinani_jsou_hovada.}
Let us consider a variant $\RTre$  of the operadic category $\Tre$
consisting of  trees that are rooted in the sense explained in
\cite[Example~4.8]{part1}.
By definition, the output half-edge of each vertex  is the minimal
element in the local order; we will denote this minimal element 
by~$0$ in the context of rooted trees. 
We use the same convention  also for the smallest leg in
the global order, i.e.\ for the root. Since $\RTre$ was obtained from
the basic operadic category $\Gr$ by a Grothendieck construction, it is
again an operadic category sharing all the nice properties of~$\Gr$.

\def\termRTre{{\sf 1}_\ssRTre}\def\oddRTre{{\mathfrak K}_\ssRTre}
\begin{theorem}
\label{markl1}
The terminal\/ $\RTre$-operad  $\termRTre$ is binary quadratic. 
Its algebras are nonunital Markl operads recalled in
Definition~\ref{b3} of Appendix~\ref{a1}.  
The category 
of algebras over its Koszul dual
$\oddRTre :=  \termRTre^!$
is isomorphic to the category of  Markl operads,  via the isomorphism
given by the suspension of the underlying~collection.
\end{theorem}

\begin{proof}
The soul of graphs in $\RTre$ with one internal edge is
the oriented interval consisting of two oriented half-edges as 
in
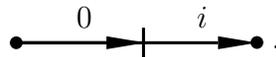
\[
\psscalebox{1.0 1.0} 
{
\begin{pspicture}(-5,-0.32000703)(8.167188,0.32000703)
\psdots[linecolor=black, dotsize=0.16963564](0.083594486,-0.121504106)
\psdots[linecolor=black, dotsize=0.16963564](3.283593,-0.11848185)
\psline[linecolor=black, linewidth=0.04](1.7738773,0.08000693)(1.7741597,-0.3199929)
\psline[linecolor=black, linewidth=0.04, arrowsize=0.05291666666666667cm 3.0,arrowlength=3.0,arrowinset=0.0]{->}(0.08359377,-0.11999298)(3.2835937,-0.11999298)
\psline[linecolor=black, linewidth=0.04, arrowsize=0.05291666666666667cm 3.0,arrowlength=3.0,arrowinset=0.0]{->}(0.08359377,-0.11999298)(1.8,-0.11999298)
\rput[bl](2.4835937,0.08000702){$i$}
\rput[bl](3.4835937,-.2){.}
\rput[bl](0.8835938,0.08000702){$0$}
\end{pspicture}
}
\]
Since the label
of the out-going half-edge is always the minimal one in the local
order, we omit it from pictures and draw the internal edges  as arrows
acquiring the label of the in-going half-edge, see
\begin{equation}
\label{je_patek_a_melu_z_posledniho}
\psscalebox{1.0 1.0} 
{
\begin{pspicture}(4.75,-0.32000703)(8.167188,0.32000703)
\psdots[linecolor=black, dotsize=0.16963564](5.2835946,-0.121504106)
\psdots[linecolor=black, dotsize=0.16963564](8.083593,-0.11848185)
\psline[linecolor=black, linewidth=0.04, arrowsize=0.05291666666666668cm 3.0,arrowlength=3,arrowinset=0.0]{->}(5.2835937,-0.11999298)(8.083593,-0.11999298)
\rput[bl](6.6835938,0.08000702){$i$}
\rput[bl](8.3,-.2){.}
\end{pspicture}
}
\end{equation}
Let $E$ be an obvious modification of the constant
collection~(\ref{vcera}) to the category
$\RTre$. The display
\begin{equation}
\label{Noha porad boli.}
\raisebox{-7em}{
\psscalebox{1.0 1.0} 
{
\begin{pspicture}(0,-1.7993219)(12.822262,1.7993219)
\psline[linecolor=black, linewidth=0.04, arrowsize=0.05291666666666667cm 3.0,arrowlength=3.0,arrowinset=0.0]{->}(0.36226174,-0.6006781)(1.5622617,0.5993219)
\psdots[linecolor=black, dotsize=0.16963564](1.5622624,0.59781075)
\psline[linecolor=black, linewidth=0.04, arrowsize=0.05291666666666667cm 3.0,arrowlength=3.0,arrowinset=0.0]{->}(2.5622618,-0.6006781)(1.5622617,0.5993219)
\psline[linecolor=black, linewidth=0.04, arrowsize=0.05291666666666667cm 3.0,arrowlength=3.0,arrowinset=0.0]{->}(3.7622616,-0.6006781)(4.9622617,0.5993219)
\psdots[linecolor=black, dotsize=0.16963564](4.9622626,0.59781075)
\psline[linecolor=black, linewidth=0.04, arrowsize=0.05291666666666667cm 3.0,arrowlength=3.0,arrowinset=0.0]{->}(5.9622617,-0.6006781)(4.9622617,0.5993219)
\rput{43.622902}(0.26521635,-0.66406024){\psellipse[linecolor=black, linewidth=0.04, linestyle=dashed, dash=0.17638889cm 0.10583334cm, dimen=outer](0.96226174,0.0)(1.2,0.6)}
\rput{-230.65244}(8.761535,-4.147821){\psellipse[linecolor=black, linewidth=0.04, linestyle=dashed, dash=0.17638889cm 0.10583334cm, dimen=outer](5.362262,0.0)(1.2,0.6)}
\psdots[linecolor=black, dotsize=0.16963564](0.36226246,-0.60218924)
\psdots[linecolor=black, dotsize=0.16963564](5.9622626,-0.60218924)
\psdots[linecolor=black, dotsize=0.16963564](3.7622623,-0.60218924)
\psdots[linecolor=black, dotsize=0.16963564](2.5622625,-0.60218924)
\rput[bl](2.9622617,0.1993219){$\oplus$}
\psline[linecolor=black, linewidth=0.04, arrowsize=0.05291666666666667cm 3.0,arrowlength=3.0,arrowinset=0.0]{->}(9.762261,-1.4006782)(9.762261,0.0)
\psdots[linecolor=black, dotsize=0.16963564](9.762262,1.3978108)
\psdots[linecolor=black, dotsize=0.16963564](9.762262,-0.0021892267)
\psdots[linecolor=black, dotsize=0.17041689](9.762262,-1.4021893)
\psline[linecolor=black, linewidth=0.04, arrowsize=0.05291666666666667cm 3.0,arrowlength=3.0,arrowinset=0.0]{->}(9.762261,0.0)(9.762261,1.3993219)
\psline[linecolor=black, linewidth=0.04, arrowsize=0.05291666666666667cm 3.0,arrowlength=3.0,arrowinset=0.0]{->}(11.962262,-1.4006782)(11.962262,0.0)
\psdots[linecolor=black, dotsize=0.16963564](11.962262,1.3978108)
\psdots[linecolor=black, dotsize=0.16963564](11.962262,-0.0021892267)
\psdots[linecolor=black, dotsize=0.17041689](11.962262,-1.4021893)
\psline[linecolor=black, linewidth=0.04, arrowsize=0.05291666666666667cm 3.0,arrowlength=3.0,arrowinset=0.0]{->}(11.962262,0.0)(11.962262,1.3993219)
\psellipse[linecolor=black, linewidth=0.04, linestyle=dashed, dash=0.17638889cm 0.10583334cm, dimen=outer](9.762261,-0.7006781)(0.6,1.1)
\psellipse[linecolor=black, linewidth=0.04, linestyle=dashed, dash=0.17638889cm 0.10583334cm, dimen=outer](11.962262,0.6993219)(0.6,1.1)
\rput[bl](10.762261,0.0){$\oplus$}
\rput[bl](0.56226176,0.0){$a$}
\rput[bl](2.1622617,0.0){$b$}
\rput[bl](3.9622617,0.0){$c$}
\rput[bl](5.5622616,0.0){$d$}
\rput[bl](9.362262,0.5993219){$e$}
\rput[bl](9.362262,-0.8006781){$f$}
\rput[bl](11.562262,0.5993219){$g$}
\rput[bl](11.562262,-0.8006781){$h$}
\rput[bl](0.96226174,-0.4006781){$E$}
\rput[bl](5.2262,-0.6){$E$}
\rput[bl](5.762262,0.7993219){$E$}
\rput[bl](0.36226174,0.7993219){$E$}
\rput[bl](12.2262,-.9){$E$}
\rput[bl](10.3,-1.6006781){$E$}
\rput[bl](12.1,.3219){$E$}
\rput[bl](9.85,-1.0006781){$E$}
\end{pspicture}
}
}
\end{equation}
features souls of rooted trees with two internal edges. It shows
that  $\Free^2(E)$ has two families of bases, $(b^1_1,b^1_2)$
corresponding to the direct sum in the left part and  $(b^2_1,b^2_2)$
corresponding to the direct sum on the right of the display. Let $R$ be the subspace of
$\Free^2(E)$ spanned by the relations
\begin{equation}
\label{Rozbila se mi mycka nadobi.}
r_1 := b^1_1 - b^1_2 \ \hbox { and } \ r_2 := b^2_1 - b^2_2.
\end{equation}
The isomorphism $\termRTre \cong \Free(E)/(R)$ can be established 
as in the proof of Theorem~\ref{v_nedeli_si_snad_zaletam_s_NOE}.

\def\calS{{\EuScript S}}
To identify $\termRTre$-algebras with Markl operads  
we proceed as in the proof 
of Theorem~\ref{pozitri_turnaj_v_Patku}. 
We start by realizing that the local terminal
objects are rooted corollas  $c^\uparrow(\sigma)$, $\sigma \in
\Sigma_n$, shown~in 
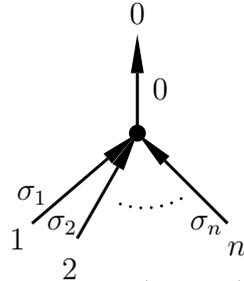
\[
\psscalebox{1.0 1.0} 
{
\begin{pspicture}(0,-1.34791)(3.84,1.6534791)
\psdots[linecolor=black, dotsize=0.225](2.0,0.14652084)
\psline[linecolor=black, linewidth=0.04, arrowsize=0.05291666666666667cm 3.0,arrowlength=3,arrowinset=0.0]{->}(0.6,-1.0534792)(2.0,0.14652084)
\psline[linecolor=black, linewidth=0.04, arrowsize=0.05291666666666667cm 3.0,arrowlength=3,arrowinset=0.0]{->}(1.2,-1.2534791)(2.0,0.14652084)
\psline[linecolor=black, linewidth=0.04, arrowsize=0.05291666666666667cm 3.0,arrowlength=3.0,arrowinset=0.0]{->}(3.2,-1.0534792)(2.0,0.14652084)
\psline[linecolor=black, linewidth=0.04, arrowsize=0.05291666666666667cm 3.0,arrowlength=3,arrowinset=0.0]{->}(2.0,0.14652084)(2.0,1.465209)
\rput[bl](0.4,-.8){$\sigma_1$}
\rput[bl](.8,-1.2){$\sigma_2$}
\rput[bl](2.7,-1.2){$\sigma_n$}
\rput[bl](.3,-1.4){$1$}
\rput[bl](1.0,-1.791){$2$}
\rput[bl](3.2,-1.4792){$n$}
\rput[bl](2.2,0.6){$0$}
\rput[b](2,1.6){$0$}
\rput(-3.177993,1){\psarc[linecolor=black, linewidth=0.04, linestyle=dotted, dotsep=0.10583334cm, dimen=outer](5.275,-0.8477939){1}{-110}{-60}}
\end{pspicture}
}
\]
while the chosen local terminal objects
are  $c^\uparrow_n := c^\uparrow(\id_n)$.
The set $\pi_0(\RTre)$ of connected components is therefore 
identified with the natural numbers $\bbN$. Analyzing
the actions of local terminal objects in~\cite[display~(60)]{part1}
we conclude that
the underlying collections for $\termRTre$-algebras are sequences
$\calS(n)$, $n \in \bbN$, of $\Sigma_n$-modules.

As in the proof of Theorem~\ref{pozitri_turnaj_v_Patku} we establish
that the value of the generating collection $E$ on graphs whose
soul is the arrow in~\eqref{je_patek_a_melu_z_posledniho}
produces partial compositions~\eqref{zitra_s_Mikesem_na_Jazz}, that the
relation $r_1$ expresses the parallel associativity, i.e.\ the
first and the last cases of the relation 
in~\eqref{zitra_s_Mikesem_do_Salmovske} of the Appendix, and
$r_2$ the sequential associativity, i.e.\ the middle
case of that relation.
 
We are sure that at this stage  the reader will easily
describe the annihilator $R^\perp$ of the space $R$ of relations and 
identify algebras of the Koszul dual 
\[
\oddRTre := \termRTre^! =  \Free(\susp E^*)/(R^\perp)
\]
as structures with degree $+1$ operations
\begin{equation}
\label{Byl_jsem_se_142_v_Benesove.}
\bullet_i : \calS(m) \ot \calS(n) \to \calS(m +n -1)
\end{equation} 
satisfying~(\ref{Krtecek_na_mne_kouka.}) and the
associativities~(\ref{zitra_s_Mikesem_do_Salmovske}) with the minus
sign. It can be verified
directly that the level-wise suspension of such a
structure is an ordinary Markl operad. However, a more conceptual
approach based on coboundaries introduced 
in~\cite[Example~6.10]{part1} is available. 

As in the cases of modular and cyclic operads we notice
that, for a rooted tree $T \in \RTre$, we have 
$\oddRTre(T) \cong \det(\edg(T))$, the
determinant of the set of internal edges of $T$. On the other hand,
the correspondence that assigns to each vertex of $T$ its out-going
edge is an isomorphism
\begin{equation}
\label{Dnes_mam_vedecke_narozeniny.}
\edg(T) \cong   \{\hbox {vertices of $T$}\} \setminus \{\hbox {the root}\}
\end{equation}
which implies that $\det(\edg(T))$
is isomorphic to   $\fdl(T)$, where $\fdl$ is the 
coboundary with $\fl: \pi_0(\RTre) \to
\Vect$ the constant function with value the desuspension 
$\hbox {$\downarrow\!\bfk$}$ of
the ground field. Therefore
\[
\oddRTre =  \termRTre  \ot \fdl
\] 
and the identification of $\oddRTre$-algebras with
Markl operads via the suspension of the underlying collection  follows
from~\cite[Proposition~6.11]{part1}.
\end{proof}

Similar statements can be proved also for the operadic categories $\PTr$ and 
$\PRTr$ of planar resp.~planar rooted trees introduced 
in~\cite[Example~4.9]{part1}. The corresponding
terminal operads ${\sf 1}_\PTr$ resp.~ ${\sf 1}_\PRTr$ will again be
self-dual binary quadratic, with algebras nonsymmetric cyclic
operads~\cite[page~257]{markl-shnider-stasheff:book}
resp.~ nonsymmetric Markl
operads~\cite[Definition~II.1.14]{markl-shnider-stasheff:book}. 
We leave the details to the reader. 

\subsection{Pre-permutads.} Pre-permutads   introduced
in~\cite{loday11:_permut} form a link between nonsymmetric operads and
permutads. They are
structures satisfying all axioms of Markl operads as recalled in
Definition~\ref{b3} except the parallel associativity, i.e.\ the first
and the last case of~(\ref{zitra_s_Mikesem_do_Salmovske}). 
Pre-permutads  are algebras for a certain binary quadratic 
operad over the category $\RTre$
of rooted trees which is very far from being Koszul self-dual.

\begin{definition}
Let $\pperm : = \Free(E)/(R)$ be the  $\RTr$-operad with the same
collection $E$ of
generators as the  $\RTr$-operad  $\termRTre$ for ordinary operads,
cf.~the proof of Theorem~\ref{markl1}. The ideal of relations $ (R)$ 
is spanned by $r_2$ in~\eqref{Rozbila se mi mycka nadobi.} belonging to the
direct sum on the right of~\eqref{Noha porad boli.}.
\end{definition}

\begin{theorem}
\label{Vcera_s_Mikesem_v_Matu.}
Pre-permutads in the sense of~\cite[page~348]{loday11:_permut} are algebras over
$\pperm$. The~category of algebras
over the Koszul dual $\pperm^!$ is isomorphic to the category of 
structures satisfying all
axioms of Markl operads, except the
associativity~(\ref{zitra_s_Mikesem_do_Salmovske}) which is  replaced by
\[
(f \circ_j g)\circ_i h =
\begin{cases}
0& \mbox{for } 1\leq i< j
\\
f \circ_j(g \circ_{i-j+1} h)& \mbox{for }
j\leq i~< b+j
\\
0& \mbox{for }
j+b\leq i\leq a+b-1.
\end{cases}
\] 
\end{theorem}

\begin{proof}
The first part of the theorem is an immediate consequence of the
definition of $\pperm$. Let $d^1_1,d^1_2$ resp.~$d^2_1,d^2_2$
be the bases  of $\Free^2(\susp E^*)$ dual to $b^1_1,b^1_2$
resp.~$b^2_1,b^2_2$. Then the annihilator $R^\perp$ is clearly spanned
by
\[
o := d^2_1 + d^2_2,\ d^1_1 \hbox { and } d^1_2. 
\]
As before, we identify algebras over $\pperm^! =
\Free(\susp^*E)/{R^\perp}$ with structures equipped with degree~$+1$
operations~(\ref{Byl_jsem_se_142_v_Benesove.}) satisfying
\[
(f \bullet_j g)\bullet_i h =
\begin{cases}
0& \mbox{for } 1\leq i< j
\\
- f \bullet_j(g \bullet_{i-j+1} h)& \mbox{for }
j\leq i < b+j
\\
0& \mbox{for }
j+b\leq i\leq a+b-1,
\end{cases}
\] 
whose first case corresponds to $d^1_1$, the middle to $o$, and the
last one to  $d^1_2$. The level-wise suspension of this object is
the structure described in Theorem~\ref{Vcera_s_Mikesem_v_Matu.}.
\end{proof}

\section{PROP-like structures and permutads}
\label{Ta_moje_lenost_je_strasna.}

In this section we treat some important 
variants of PROPs whose associated
operadic categories are sundry modifications
of the category $\Whe$ of connected ordered oriented  graphs
introduced in~\cite[Example~4.20]{part1}. 
The orientation divides 
the set of half-edges adjacent
to each vertex of the graphs involved
into two subsets -- inputs and outputs of that vertex.
The local terminal objects in these categories will thus be the ordered corollas
$c{\sigma \choose \lambda}$, $\sigma \in \Sigma_k$, $\lambda \in
\Sigma_l$, 
as in
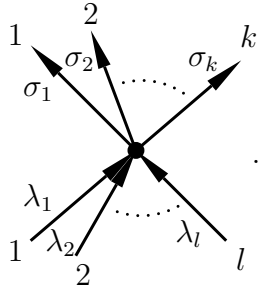
\[
\psscalebox{1.0 1.0} 
{
\begin{pspicture}(0,-1.88)(3.64,1.88)
\psdots[linecolor=black, dotsize=0.225](1.8,-0.08)
\psline[linecolor=black, linewidth=0.04, arrowsize=0.05291666666666667cm 3.0,arrowlength=3,arrowinset=0.0]{->}(0.4,-1.28)(1.8,-0.08)
\psline[linecolor=black, linewidth=0.04, arrowsize=0.05291666666666667cm 3.0,arrowlength=3,arrowinset=0.0]{->}(1.0,-1.48)(1.8,-0.08)
\psline[linecolor=black, linewidth=0.04, arrowsize=0.05291666666666667cm 3.0,arrowlength=3,arrowinset=0.0]{->}(3.0,-1.28)(1.8,-0.08)
\rput(0.5,.7){$\sigma_1$}
\rput(1.05,1.1){$\sigma_2$}
\rput(2.7,1.1){$\sigma_k$}
\rput[b](0.2,1.32){$1$}
\rput[b](1.2,1.6){$2$}
\rput[b](3.3,1.3){$k$}
\rput[t](3.4,-.2){.}
\psline[linecolor=black, linewidth=0.04, arrowsize=0.05291666666666667cm 3.0,arrowlength=3,arrowinset=0.0]{->}(1.8,-0.08)(0.4,1.32)
\psline[linecolor=black, linewidth=0.04, arrowsize=0.05291666666666667cm 3.0,arrowlength=3,arrowinset=0.0]{->}(1.8,-0.08)(3.2,1.12)
\psline[linecolor=black, linewidth=0.04, arrowsize=0.05291666666666667cm 3.0,arrowlength=3,arrowinset=0.0]{->}(1.8,-0.08)(1.2,1.52)
\rput(0.5,-.7){$\lambda_1$}
\rput(.8,-1.3){$\lambda_2$}
\rput(2.5,-1.18){$\lambda_l$}
\rput[t](0.2,-1.28){$1$}
\rput[t](1.1,-1.6){$2$}
\rput[t](3.2,-1.38){$l$}
\rput(-3.4,.9){\psarc[linecolor=black, linewidth=0.04,
  linestyle=dotted, dotsep=0.10583334cm,
  dimen=outer](5.275,-0.8477939){1}{-110}{-60}}
\rput(-3.4,.7){\psarc[linecolor=black, linewidth=0.04, linestyle=dotted, dotsep=0.10583334cm, dimen=outer](5.275,-0.8477939){1}{55}{105}}
\end{pspicture}
}
\]
The chosen
local terminal objects are the ordered corollas $c^k_l
:= c{\id_k \choose \id_l}$, $k,l \in \bbN$.
The underlying collections of the
corresponding algebras will be  families
\begin{equation}
\label{Dnes_jsem_byl_s_NOE_v_Pribyslavi.}
D(m,n),\ m,n \in \bbN,
\end{equation}
of $\Sigma_m \times \Sigma_n$-modules. We will see that the
orientation of the underlying graphs implies that the corresponding
terminal operads are self-dual.

\noindent 
\subsection{Wheeled properads}
These structures were introduced in \cite{mms} as an
extension  of Vallette's properads~\cite{vallette:TAMS07} 
that allowed ``back-in-time'' edges in order to
capture traces and therefore also
master equations in mathematical physics. Surprisingly, this extended
theory is better behaved than the theory of properads
in that their composition laws are
iterated composites of elementary ones, that is, in terms of pasting
schemes, of those given by contraction of a single edge.

The operadic category relevant for wheeled properads is 
the category $\Whe$ of connected oriented  ordered graphs.
Since $\Whe$ was 
constructed in~\cite[Example~4.20]{part1} 
from the basic operadic category $\Gr$ by iterating
the Grothendieck construction, and since it clearly satisfies the conditions 
$\UFB$ and $\SGrad$, we conclude as in the previous
sections that our theory of Koszul duality applies to it.

\begin{theorem}
The terminal\/ $\Whe$-operad  $\Wheterm$ is binary quadratic. 
Its algebras are wheeled properads introduced in~\cite[Definition~2.2.1]{mms}.
The operad $\Wheterm$  
is self-dual in the sense of Definition~\ref{Uz_je_o_te_medaili_zapis.}.
\end{theorem}

\begin{proof}
The proof goes along the same lines as the proofs of similar
statements, namely 
Theorems~\ref{v_nedeli_si_snad_zaletam_s_NOE}, \ref{vcera_jsem_podlehl},
\ref{Zitra_letim_do_Bari.}, \ref{markl1}, \ref{Vcera_s_Mikesem_v_Matu.}
and Proposition~\ref{pozitri_turnaj_v_Patku}, so we will be telegraphic.
As before, for a wheeled graph $\Gamma$ we put
\[
{\mathfrak {K}}_\Whe  (\Gamma):=\Wheterm^!(\Gamma) \cong
\det(\edg(\Gamma)), 
\]
the determinant of the set of internal edges of $\Gamma$. On the other hand,
the correspondence that assigns to each vertex $v$ of $\Gamma$ the set
$\out(v)$ of its out-going
edges defines an isomorphism
\begin{equation}
\label{Dnes_mam_vedecke_narozeniny_bis.}
\textstyle
\edg(\Gamma) \cong   \bigcup_{v \in \Vert(\Gamma)} \out(v) \setminus \out(\Gamma)
\end{equation}
which implies that $\det(\edg(\Gamma))$
is isomorphic to   $\fdl(\Gamma)$, where $\fdl$ is the 
coboundary with $\fl: \pi_0(\RTre) \to
\Vect$ the function defined by
\[
\fl(c^k_l) : = \downarrow^k \bfk,
\]
the desuspension of the ground field  iterated $k$ times. Therefore
\[
{\mathfrak {K}}_\Whe \cong  \Wheterm  \ot \fdl
\] 
which, by~\cite[Proposition~6.11]{part1},
implies the self-duality of $ \Wheterm$.

It follows from the description of the local terminal objects in
$\Whe$ that the underlying structure of a $\Wheterm$-algebra is 
a collection of bimodules as
in~(\ref{Dnes_jsem_byl_s_NOE_v_Pribyslavi.}). The composition laws
are given by wheeled graphs with one internal edge, whose souls are
depicted in
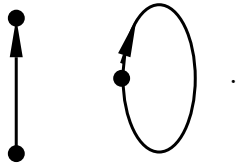
\[
\psscalebox{1.0 1.0} 
{
\begin{pspicture}(0,-1.0554386)(2.5108774,1.0554386)
\psdots[linecolor=black, dotsize=0.225](0.11087738,-0.9445613)
\psline[linecolor=black, linewidth=0.04, arrowsize=0.05291666666666667cm 3.0,arrowlength=3.0,arrowinset=0.0]{->}(0.11087738,-0.9445613)(0.11087738,0.8554387)
\psdots[linecolor=black, dotsize=0.225](0.11087738,0.8554387)
\psellipse[linecolor=black, linewidth=0.04, dimen=outer](2.0108774,0.05543869)(0.5,1.0)
\psdots[linecolor=black, dotsize=0.225](1.5108774,0.05543869)
\psline[linecolor=black, linewidth=0.04,
arrowsize=0.05291666666666667cm
3.0,arrowlength=3.0,arrowinset=0.0]{->}(1.5108774,0.2554387)(1.7108774,0.8554387)
\rput[b](3,0){.}
\end{pspicture}
}
\]
We recognize them as the operations
\begin{subequations}
\begin{align}
\label{musim_to_dokoukat}
\circ^i_j :\ &  D(m,n) \ot D(k,l) 
\longrightarrow D(m\! +\! k\! -\!1,n\!+ \!l \!-\!1),\
1 \leq i \leq n, \ 1 \leq j \leq k, \hbox { and}   
\\
\label{Budu_mit_silu_to_dokoukat?}
\xi^i_j :\ & D(m,n) \longrightarrow D(m\!-\!1,n\!-\!1),\
1\leq i \leq m, \ 1 \leq j \leq n
\end{align}
\end{subequations}
in formulas~(16) and~(17) of~\cite{mms}.

As in the previous cases, the axioms that these operations satisfy are
determined by graphs with two internal edges whose souls are depicted
in the following display.
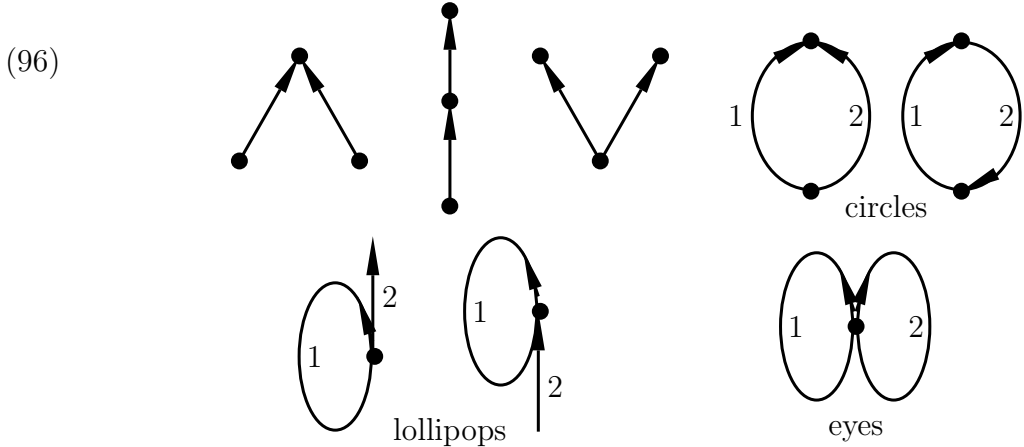
\begin{equation}
\label{Ten_film_mne_oddelal.}
\raisebox{-12em}{
\psscalebox{1.0 1.0} 
{
\begin{pspicture}(0,-2.9354386)(10.510878,2.9354386)
\psdots[linecolor=black, dotsize=0.225](1.7108774,0.8245613)
\psdots[linecolor=black, dotsize=0.225](0.9108774,2.2245612)
\psellipse[linecolor=black, linewidth=0.04, dimen=outer](7.7108774,1.4245613)(0.8,1.0)
\psline[linecolor=black, linewidth=0.04, arrowsize=0.05291666666666668cm 3.0,arrowlength=3.0,arrowinset=0.0]{->}(0.11087738,0.8245613)(0.9108774,2.2245612)
\psline[linecolor=black, linewidth=0.04, arrowsize=0.05291666666666668cm 3.0,arrowlength=3.0,arrowinset=0.0]{->}(1.7108774,0.8245613)(0.9108774,2.2245612)
\psline[linecolor=black, linewidth=0.04, arrowsize=0.05291666666666668cm 3.0,arrowlength=3.0,arrowinset=0.0]{->}(2.9108775,0.2245613)(2.9108775,1.6245613)
\psline[linecolor=black, linewidth=0.04, arrowsize=0.05291666666666668cm 3.0,arrowlength=3.0,arrowinset=0.0]{->}(2.9108775,1.6245613)(2.9108775,2.8245614)
\psdots[linecolor=black, dotsize=0.225](2.9108775,2.8245614)
\psdots[linecolor=black, dotsize=0.225](2.9108775,1.6245613)
\psdots[linecolor=black, dotsize=0.225](0.11087738,0.8245613)
\psdots[linecolor=black, dotsize=0.225](4.1108775,2.2245612)
\psdots[linecolor=black, dotsize=0.225](4.910877,0.8245613)
\psline[linecolor=black, linewidth=0.04, arrowsize=0.05291666666666668cm 3.0,arrowlength=3.0,arrowinset=0.0]{->}(4.910877,0.8245613)(5.7108774,2.2245612)
\psline[linecolor=black, linewidth=0.04, arrowsize=0.05291666666666668cm 3.0,arrowlength=3.0,arrowinset=0.0]{->}(4.910877,0.8245613)(4.1108775,2.2245612)
\psdots[linecolor=black, dotsize=0.225](5.7108774,2.2245612)
\psdots[linecolor=black, dotsize=0.225](2.9108775,0.2245613)
\psellipse[linecolor=black, linewidth=0.04, dimen=outer](1.3858774,-1.7754387)(0.5,1.0)
\psdots[linecolor=black, dotsize=0.225](1.9108773,-1.7754387)
\psline[linecolor=black, linewidth=0.04, arrowsize=0.05291666666666668cm 3.0,arrowlength=3.0,arrowinset=0.0]{->}(1.8858774,-1.5754387)(1.6858773,-0.9754387)
\psline[linecolor=black, linewidth=0.04, arrowsize=0.05291666666666668cm 3.0,arrowlength=3.0,arrowinset=0.0]{->}(1.8858774,-1.7754387)(1.8858774,-0.17543869)
\psellipse[linecolor=black, linewidth=0.04, dimen=outer](8.810878,-1.3754387)(0.5,1.0)
\psdots[linecolor=black, dotsize=0.225](8.310878,-1.3754387)
\psline[linecolor=black, linewidth=0.04, arrowsize=0.05291666666666668cm 3.0,arrowlength=3.0,arrowinset=0.0]{->}(8.310878,-1.1754386)(8.510878,-0.5754387)
\psellipse[linecolor=black, linewidth=0.04, dimen=outer](7.785877,-1.3754387)(0.5,1.0)
\psline[linecolor=black, linewidth=0.04, arrowsize=0.05291666666666668cm 3.0,arrowlength=3.0,arrowinset=0.0]{->}(8.285877,-1.1754386)(8.085877,-0.5754387)
\psdots[linecolor=black, dotsize=0.225](7.7108774,0.42456132)
\psdots[linecolor=black, dotsize=0.225](7.7108774,2.4245613)
\psline[linecolor=black, linewidth=0.04, arrowsize=0.05291666666666667cm 3.0,arrowlength=3.0,arrowinset=0.0]{<-}(7.7108774,2.4245613)(7.3108773,2.2245612)
\psline[linecolor=black, linewidth=0.04, arrowsize=0.05291666666666667cm 3.0,arrowlength=3.0,arrowinset=0.0]{<-}(7.7108774,2.4245613)(8.110877,2.2245612)
\psellipse[linecolor=black, linewidth=0.04, dimen=outer](9.710877,1.4245613)(0.8,1.0)
\psdots[linecolor=black, dotsize=0.225](9.710877,2.4245613)
\psline[linecolor=black, linewidth=0.04, arrowsize=0.05291666666666667cm 3.0,arrowlength=3.0,arrowinset=0.0]{<-}(9.710877,2.4245613)(9.310878,2.2245612)
\psdots[linecolor=black, dotsize=0.225](9.710877,0.42456132)
\psline[linecolor=black, linewidth=0.04, arrowsize=0.05291666666666667cm 3.0,arrowlength=3.0,arrowinset=0.0]{<-}(9.710877,0.42456132)(10.110877,0.6245613)
\psellipse[linecolor=black, linewidth=0.04, dimen=outer](3.5858774,-1.1754386)(0.5,1.0)
\psdots[linecolor=black, dotsize=0.225](4.1108775,-1.1754386)
\psline[linecolor=black, linewidth=0.04, arrowsize=0.05291666666666668cm 3.0,arrowlength=3.0,arrowinset=0.0]{->}(4.0858774,-0.9754387)(3.8858774,-0.3754387)
\psline[linecolor=black, linewidth=0.04, arrowsize=0.05291666666666668cm 3.0,arrowlength=3.0,arrowinset=0.0]{->}(4.0858774,-2.7754388)(4.0858774,-1.1754386)
\rput(6.7108774,1.4245613){$1$}
\rput(8.310878,1.4245613){$2$}
\rput(3.3108773,-1.1754386){$1$}
\rput(1.1108774,-1.7754387){$1$}
\rput(9.110877,1.4245613){$1$}
\rput(7.5108776,-1.3754387){$1$}
\rput(2.1108773,-0.9754387){$2$}
\rput(4.3108773,-2.1754386){$2$}
\rput(9.110877,-1.3754387){$2$}
\rput(10.310878,1.4245613){$2$}
\rput(2.9108775,-2.7754388){lollipops}
\rput(8.310878,-2.7754388){eyes}
\rput(8.710877,0.2245613){circles}
\end{pspicture}
}
}
\end{equation}
The graphs with three vertices
induce the parallel and sequential associativity of the
$\xi$-operations, similar to that for Markl
operads~(\ref{zitra_s_Mikesem_do_Salmovske}). 
They were explicitly given in the dioperadic context as axioms
(a) and (b) in \cite[page 111]{gan}. 

The circles in~\eqref{Ten_film_mne_oddelal.} represent
the rules of the type
$
\circ_1 \xi_2 = \circ_2 \xi_1,
$  
where $\xi_1$ resp.~$\xi_2$ is the operation corresponding to the
shrinking of the edge labeled $1$ resp.~$2$, and similarly for
$\circ_1$ and $\circ_2$. The lollipops  in~\eqref{Ten_film_mne_oddelal.}
force the interchange rule $\xi_2 \circ_1 = \circ_1 \xi_2$, and the eyes 
the rule $\circ_1 \circ_2 =  \circ_2 \circ_1$. To
expand these remaining axioms into explicit forms similar to that on
\cite[page 111]{gan} would
not be very helpful; we thus leave it as an exercise for a
determined reader.
\end{proof}

\subsection{Dioperads} They
were introduced in \cite{gan} as PROP-like structures whose algebras
are objects such as Lie or infinitesimal
bialgebras (called mock bialgebras in \cite{markl:dl}). 
A short definition is that a dioperad is a wheeled
properad without the $\xi^i_j$-operations~(\ref{Budu_mit_silu_to_dokoukat?}). 
The underlying operadic category is the category 
$\Dio$  of ordered simply connected oriented graphs introduced in
\cite[Example~4.21]{part1}. As before, one may check that $\Dio$ meets
all requirements of our theory. One has the expected:

\begin{theorem}
Dioperads  are algebras over the
terminal $\Dio$-operad ${\sf 1}_\Dio$, which is binary quadratic and
self-dual. 
\end{theorem}

\begin{proof}
The proof is a simplified version of the wheeled case. The
self-duality of  ${\sf 1}_\Dio$ is established in precisely the same 
way as the self-duality of the terminal $\Whe$-operad 
$\Wheterm$; the existence of the relevant coboundary is given by
isomorphism~(\ref{Dnes_mam_vedecke_narozeniny_bis.}) which clearly
holds in $\Dio$ as well. 
The soul of graphs in $\Dio$ with one internal edge is the oriented
interval, with the corresponding operation as in~(\ref{musim_to_dokoukat}).
The souls of graphs in $\Dio$ with two internal edges are the three upper left
graphs in~\eqref{Ten_film_mne_oddelal.}. 
The resulting axioms are the parallel and
sequential associativities which are the same as for
$\xi^i_j$-operations of wheeled
properads, see~\cite[\S1.1]{gan}.
 \end{proof}

\def\jcirc#1{{\hskip 1mm {}_{#1}\circ \hskip .2mm}}
\noindent 
\subsection{$\frac12$PROPs} These structures 
were introduced, following a 
suggestion of Kontsevich, in~\cite{mv} as a link between dioperads
and PROPs. A  $\frac12$PROP is a collection of
bimodules~(\ref{Dnes_jsem_byl_s_NOE_v_Pribyslavi.}) which is {\em
  stable\/} in that it fulfills
\[
D(m,n) = 0\ \hbox { if } \ m+n < 3,
\]
together with partial vertical compositions
\begin{align*}
\circ_i:\ &   D(m_1,n_1) \ot D(1,l) 
\to D(m_1,n_1+l-1),\ 1 \leq i \leq n_1, \hbox { and}
\\
\jcirc j :\ & D(k,1) \ot D(m_2,n_2) 
\to D(m_2 + k -1,n_2),\ 1 \leq j \leq m_2,
\end{align*}
that satisfy the axioms of vertical compositions in PROPs. 
The corresponding operadic category 
$\hGr$ is introduced in \cite[Example~4.22]{part1}.
We have the expected statement whose proof is left to the reader.

\begin{theorem}
$\frac12$PROPs  are algebras over the
terminal $\hGr$-operad ${\sf 1}_\hGr$. This operad is binary quadratic and
self-dual. 
\end{theorem}

The operadic categories considered so far in this section were based on
graphs. Let us give one example where this is not the case.

\subsection{Permutads}
\label{Michael_do_Prahy.}
They are structures introduced by Loday and Ronco 
in~\cite{loday11:_permut} 
to handle the combinatorial structure of objects like the
permutahedra. We will describe an operadic category  $\Per$ such that
permutads are algebras over the terminal operad for this
category.
 
Let $\underline n$ denote the 
finite ordered set  $(\rada 1n)$,  $n \geq 1$. Objects
of $\Per$ are surjections $\alpha: \underline n \epi \underline k$,
$n \geq 1$, and the morphisms are diagrams
\begin{equation}
\label{nikdo_mi_nepopral_k_narozeninam}
\xymatrix@C=2em{\underline n \ar@{=}[r]
\ar@{->>}[d]_{\alpha'} & \underline n \ar@{->>}[d]^{\alpha''}
\\
\underline k' \ar[r]^\gamma & \underline k''
}
\end{equation}
in which $\gamma$ is order preserving (and necessarily a
surjection). 

The cardinality functor is defined by $|\alpha:
\underline n \epi \underline k| := k$. The $i$-th fiber of the
morphism in~(\ref{nikdo_mi_nepopral_k_narozeninam}) is the surjection
$(\gamma\alpha')^{-1}(i) \epi \gamma^{-1}(i)$, $i \in \underline
k$. The only local terminal objects are $\underline n \to \underline
1$, $n \geq 1$, which are also the chosen ones. The
category $\Per$ is graded by $e(\underline n \epi \underline k) :=
k-1$. All \qb{s}, and isomorphisms in general, are identities. The
first sentence of the following theorem is the content 
of~\cite[Proposition~26]{markl:perm}.

\def\Pterm{{\sf 1}_{\Per}}
\begin{theorem}
\label{Treti_den_je_kriticky.}
Algebras over the terminal $\Per$-operad $\Pterm$ are the
permutads of\/~\cite{loday11:_permut}. The~operad  $\Pterm$ is binary
quadratic. It is  self-dual in the sense that the category of algebras over
$\Pterm^!$ is isomorphic to the category of permutads via the functor
induced by the suspension of the underlying collection.
\end{theorem}

\begin{proof}
Let us give a quadratic presentation of the terminal operad
$\Pterm$. As noticed in Example~\ref{Budu_mit_nova_sluchatka.}, the
category $\QV(e)$ of virtual isomorphisms related to $\Per$ 
is isomorphic to the category $\Iso$ of isomorphisms in $\Per$. Since
all isomorphisms in $\Per$ are identities, we infer that in
fact $\QV(e) \cong \Per_{\rm disc}$, the discrete category with the same
objects as $\Per$. Therefore a $\QV(e)$-presheaf is just a rule that
assigns to each $\alpha \in \Per$ a vector space $E(\alpha) \in \Vect$.
Let us define a $1$-connected $\Per$-collection, in the sense of
Definition~\ref{Jarka_dnes_u_lekare}, by
\begin{equation}
\label{po_GGP}
E(\alpha) := 
\begin{cases}
\bfk & \hbox { if $|\alpha| = 2$}
\\
0 & \hbox {otherwise,}
\end{cases}
\end{equation}
and describe the free operad $\Free(E)$ generated by $E$. 

The first step
is to understand the labeled towers in $\ltw(\alpha)$. As all \im{s}
in $\Per$
are identities, the labeling is the identity map, so  
these towers are of the form
\[
{\mathbf \alpha}: =
\alpha \stackrel {\tau_1} \longrightarrow \alpha_1  \stackrel  {\tau_2} \longrightarrow
\alpha_2\stackrel  {\tau_3} \longrightarrow 
\cdots  \stackrel  {\tau_{s-1}}   \longrightarrow  \alpha_{s-1}.
\]
Since the generating collection $E$ is such that $E(\alpha)
\not=0$ only if $|\alpha| =2$, we may consider only towers
in which each $\tau_i$, $1 \leq i \leq s-1$, decreases the cardinality by one. 
For $\alpha : \underline n \epi \underline k$,
such a tower is a diagram
\begin{equation}
\label{Michael_chce_do_Prahy.}
\xymatrix{
\underline n \ar@{=}[r] \ar@{->>}[d]_\alpha
& \underline n  \ar@{=}[r] \ar@{->>}[d]_{\alpha_1}& 
\underline n  \ar@{=}[r]\ar@{->>}[d]_{\alpha_2} & \cdots  \ar@{=}[r]
&\ar@{->>}[d]_{\alpha_{k-1}}
 \underline n
\\
\underline k  \ar@{->>}[r]^{\nu_1}& \underline {k\!-\!1}  \ar@{->>}[r]^{\nu_2} 
&  \underline {k\!-\!2}
\ar@{->>}[r]^{\nu_3} 
& \cdots  \ar@{->>}[r]^{\nu_{k-2}}&   \underline 2
}
\end{equation}
with $\Rada \nu1{k-2}$ order-preserving surjections.
Notice that all vertical maps are determined by $\alpha$ and  $\Rada \nu1{k-2}$. 
It will be convenient to represent
$\underline k$ by a linear graph with $k$ vertices:
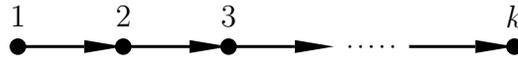
\begin{center}
\psscalebox{1.0 1.0} 
{
\begin{pspicture}(0,-0.3167067)(6.8168273,0.3167067)
\psline[linecolor=black, linewidth=0.04, linestyle=dotted, dotsep=0.10583334cm](4.5,-0.2082933)(5.108413,-0.2082933)
\rput(0.10841339,0.1917067){$1$}
\rput(1.5084134,0.1917067){$2$}
\rput(2.9084134,0.1917067){$3$}
\psline[linecolor=black, linewidth=0.04, arrowsize=0.05291666666666667cm 3.0,arrowlength=3.0,arrowinset=0.0]{<-}(1.5084134,-0.2082933)(0.10841339,-0.2082933)
\psdots[linecolor=black, dotsize=0.22](1.5084134,-0.2082933)
\psline[linecolor=black, linewidth=0.04, arrowsize=0.05291666666666667cm 3.0,arrowlength=3.0,arrowinset=0.0]{<-}(6.7084136,-0.2082933)(5.3084135,-0.2082933)
\psdots[linecolor=black, dotsize=0.22](0.10841339,-0.2082933)
\psline[linecolor=black, linewidth=0.04, arrowsize=0.05291666666666667cm 3.0,arrowlength=3.0,arrowinset=0.0]{<-}(4.3084135,-0.2082933)(2.9084134,-0.2082933)
\psdots[linecolor=black, dotsize=0.22](6.7084136,-0.2082933)
\psline[linecolor=black, linewidth=0.04, arrowsize=0.05291666666666667cm 3.0,arrowlength=3.0,arrowinset=0.0]{<-}(2.9084134,-0.2082933)(1.5084134,-0.2082933)
\psdots[linecolor=black, dotsize=0.22](2.9084134,-0.2082933)
\rput(6.7084136,0.1917067){$k$}
\end{pspicture}
}
\end{center}
and denote by $\edg(\underline k)$ or $\edg(\alpha)$
the set of $k\!-\!1$ edges of this
graph. In this graphical presentation, each $\Rada \nu1{k-2}$
contracts one of the edges of our linear graph; thus $\Rada \nu1{k-2}$
and therefore also the tower~(\ref{Michael_chce_do_Prahy.}) is
determined by the linear order on $\edg(\underline k)$ in which the
edges are contracted. We readily get the following analog of
Proposition~\ref{uz_melo_prvni_cislo_davno_vyjit}, which we formulate
as a separate claim so we can refer to it later in the proof.

\begin{claim}
\label{Uz_melo_prvni_cislo_davno_vyjit.}
The isomorphism classes of labeled towers~(\ref{Michael_chce_do_Prahy.})
are in one-to-one correspondence with the linear orders on
$\edg(\underline k)$ modulo the relation $\bowtie$ that
interchanges two edges  adjacent in this linear order
that do not share a common vertex.  
\end{claim}

Let us continue the proof of
Theorem~\ref{Treti_den_je_kriticky.}. By the above claim, $\Free(E)(\alpha)$
equals the span of the set of linear orders on $\edg(\underline k)$ modulo the
equivalence   $\bowtie$. Let us inspect in detail
its component $\Free^2(E)(\alpha)$. 
It might be nonzero only for $\alpha : \underline n \to \underline k
\in \Per$
with $k=3$, for which~(\ref{Michael_chce_do_Prahy.}) takes the form
\[
\xymatrix@C=3em{
\underline n \ar@{=}[r] \ar@{->>}[d]_\alpha
& \underline n   \ar@{->>}[d]_{\alpha_1}
\\
\underline 3  \ar@{->>}[r]^{\nu}& \underline 2  
}
\]
and the relation $\bowtie$ is vacuous.

There are two possibilities for the map $\nu$ and therefore also for
$\alpha_1$. 
The map $\nu$ may either equal
$\nu_{\{1,2\}}: \underline 3 \to \underline 2$ defined by
\[
\nu_{\{1,2\}}(1) = \nu_{\{1,2\}}(2) :=1,\ \nu_{\{1,2\}}(3) := 2
\]
which corresponds to the linear order
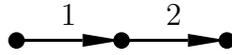
\begin{center}
\psscalebox{1.0 1.0} 
{
\begin{pspicture}(-4.5,-0.3167067)(6.8168273,0.3167067)
\rput(0.8,0.14){$1$}
\rput(2.2,0.14){$2$}
\psline[linecolor=black, linewidth=0.04, arrowsize=0.05291666666666667cm 3.0,arrowlength=3.0,arrowinset=0.0]{<-}(1.5084134,-0.2082933)(0.10841339,-0.2082933)
\psdots[linecolor=black, dotsize=0.22](1.5084134,-0.2082933)
\psdots[linecolor=black, dotsize=0.22](0.10841339,-0.2082933)
\psline[linecolor=black, linewidth=0.04, arrowsize=0.05291666666666667cm 3.0,arrowlength=3.0,arrowinset=0.0]{<-}(2.9084134,-0.2082933)(1.5084134,-0.2082933)
\psdots[linecolor=black, dotsize=0.22](2.9084134,-0.2082933)
\end{pspicture}
}
\end{center}
on $\edg(3)$, or equal
$\nu_{\{2,3\}}: \underline 3 \to \underline 2$ defined by
\[
\nu_{\{2,3\}}(1) :=1,\  \nu_{\{2,3\}}(2) = \nu_{\{2,3\}}(3) := 2,
\]
corresponding  to the order
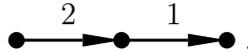
\begin{center}
\psscalebox{1.0 1.0} 
{
\begin{pspicture}(-4.5,-0.3167067)(6.8168273,0.3167067)
\rput(0.8,0.14){$2$}
\rput(2.2,0.14){$1$}
\psline[linecolor=black, linewidth=0.04, arrowsize=0.05291666666666667cm 3.0,arrowlength=3.0,arrowinset=0.0]{<-}(1.5084134,-0.2082933)(0.10841339,-0.2082933)
\psdots[linecolor=black, dotsize=0.22](1.5084134,-0.2082933)
\psdots[linecolor=black, dotsize=0.22](0.10841339,-0.2082933)
\psline[linecolor=black, linewidth=0.04, arrowsize=0.05291666666666667cm 3.0,arrowlength=3.0,arrowinset=0.0]{<-}(2.9084134,-0.2082933)(1.5084134,-0.2082933)
\psdots[linecolor=black, dotsize=0.22](2.9084134,-0.2082933)
\rput(3.2,-.3){.}
\end{pspicture}
} 
\end{center}
The fiber sequence associated to $\nu_{\{1,2\}}$ is
$\alpha|_{\alpha^{-1}\{1,2\}}, \nu_{\{1,2\}}\alpha$, 
and the one associated to $\nu_{\{2,3\}}$ is
$\alpha|_{\alpha^{-1}\{2,3\}}, \nu_{\{2,3\}}\alpha$;
therefore
\[
\Free^2(E)(\alpha) \cong \{ E(\alpha|_{\alpha^{-1}\{1,2\}}) \ot
E(\nu_{\{1,2\}}\alpha)\} \oplus  \{ E(\alpha|_{\alpha^{-1}\{2,3\}}) \ot
E(\nu_{\{2,3\}}\alpha)\}.
\]
Since $ E(\alpha|_{\alpha^{-1}\{1,2\}}) =
E(\alpha|_{\alpha^{-1}\{2,3\}}) = E(\nu_{\{1,2\}}\alpha) =
E(\nu_{\{2,3\}}\alpha) =\bfk$ by definition,
$\Free^2(E)(\alpha)$ has a basis formed~by
\[
b_1 : =[ 1\ot 1] \oplus [0 \ot 0] \ \hbox { and } \ 
b_2 : = [0\ot 0] \oplus [1 \ot 1].
\]
Let $R$ be the subspace of $\Free^2(E)$ spanned by $b_2-b_1$. 
Quotienting by the ideal $(R)$ generated by $R$ extends the relation
$\bowtie$ of Claim~\ref{Uz_melo_prvni_cislo_davno_vyjit.} by 
allowing edges that do share a~common vertex,  thus 
$\Free(E)/(R)(\alpha) \cong \bfk$ for any $\alpha$, in other words, 
\[
\Pterm \cong \Free(E)/(R).
\]

Now we describe  $\Pterm$-algebras. Since $\pi_0(\Per) = \{1,2,\ldots\}$, their
underlying collections are sequences 
of vector spaces $P(n)$,
$n\geq 1$. As we saw several times before, the structure operations
of $\Pterm$-algebras are parametrized by the generating collection $E$,
therefore, by~(\ref{po_GGP}), by surjections
$r : \underline n \epi \underline 2 \in \Per$.
If
$n_i := |r^{-1}(i)|$, $i
=1,2$, the operation corresponding to $r$ is of the form
\begin{equation}
\label{Nevim_jestli_do_ty_Ciny_opravdu_pojedu.}
\circ_r : P(n_1) \ot P(n_2) \to P(n_1+n_2)
\end{equation}
by~\cite[display~(57)]{part1}.    
It is easy to verify that the vanishing of the induced map 
$\Free(E) \to \End_P$  on the generator  $b_2-b_1$ of the ideal of
relations $(R)$  
is equivalent to the associativity  
\begin{equation}
\label{Oslava_premie_bude_za_14_dni.}
\circ_t(\circ_s \ot \id) = \circ_u(\id \ot \circ_v)
\end{equation} 
with 
$s := \alpha|_{\alpha^{-1}\{1,2\}},\ t:= \nu_{\{1,2\}}\alpha,\
u := \alpha|_{\alpha^{-1}\{2,3\}}$ and  $v:= \nu_{\{2,3\}}\alpha$. We
recognize it as the associativity 
of \cite[Lemma~2.2]{loday11:_permut} featured in the biased
definition of permutads.

It can easily be seen that ${\mathfrak{K}}_\Per(\alpha) := \Pterm^!(\alpha) \cong
\det(\edg(\alpha))$.  
As in \S\ref{Cinani_jsou_hovada.}
we identify 
${\mathfrak{K}}_\Per$-algebras as structures with degree $+1$ operations 
\[
\bullet_r : P(n_1) \ot P(n_2) \to P(n_1+n_2)
\]
with $r$ as in~(\ref{Nevim_jestli_do_ty_Ciny_opravdu_pojedu.})
satisfying an odd version
\[
\bullet_t(\bullet_s \ot \id) + \bullet_u(\id \ot \bullet_v) =0
\]
of~(\ref{Oslava_premie_bude_za_14_dni.}). It is elementary to show
that the structure induced on the component-wise suspension of the
underlying collection is that of a permutad.
\end{proof}

In~\cite{markl:perm} we prove the following theorem:

\begin{theorem}
The terminal $\ttP$-operad $\Pterm$ is Koszul.  
\end{theorem}

Its meaning is that the canonical map $\Omega(\Pterm^!) \to \Pterm$
from a suitably defined bar construction of $\Pterm^!$ to $\Pterm$ is
a component-wise homology equivalence. In other words, the dg-$\Per$ operad
$\Omega(\Pterm^!)$ is the minimal model of $\Pterm$; therefore,
according to the philosophy of~\cite[Section~4]{markl:zebrulka},
$\Omega(\Pterm^!)$-algebras are {\em strongly homotopy\/}
  permutads. An explicit description of these objects is given
in~\cite{markl:perm} as well.

\def\twr#1#2#3#4#5#6#7{
\xymatrix@R=-.3em@C=1em{
&&#2
\\
&&\triangledown
\\
#1  \ar[rr]^{#5} & &#3 \ar[ddddd]^{#6}
\\
\\
&#7&
\\
\\
\\
&&#4
}
}

\section{Derivations and the cobar construction}
\label{Ceka mne nehezky vikend.}

Derivations of traditional operads defined in terms of partial
compositions~\eqref{zitra_s_Mikesem_na_Jazz} (i.e.\
``traditional'' Markl operads) were introduced in 
\cite[Definition~1.5]{markl:zebrulka}, and the cobar construction
in that context then 
implicitly in \cite[Theorem~1.9]{markl:zebrulka}.
The aim of this section is to generalize these notions to Markl
(co)operads over operadic categories.

We require that the base operadic category $\ttO$ fulfills \AA. 
Then $e(X) = 0$ if and only if
$X\in \ttO$ is local terminal by~\cite[Lemma~3.25]{part1}. 
All Markl operads are tacitly assumed to be strictly
unital and $1$-connected.  The base monoidal category 
$\ttV$ will be the category $\Vect$ of graded vector spaces
over a field $\bfk$ of characteristic~$0$.

\subsection{Derivations, cooperads}
To simplify the notation, we will use the same symbol both for  
a (co)operad and for its underlying collection when 
the meaning is clear from the context.

\begin{definition}
\label{Uz mam 6,5 hodin na Tereji tento rok.}
A degree $s$ {\em derivation\/} of a Markl operad $\Markl$ is defined as a degree $s$
endomorphism  \hbox{$\varpi
: \Markl \to  \Markl$} of the underlying collection such that, for each  elementary morphism
\redukce{$F\fib T \stackrel\phi\to S$} and the related composition
law $\circ_{\phi}: \Markl(S)\otimes \Markl(F)\to \Markl(T)$, one has the equality
\begin{equation}
\label{zase mi nejde aktuator}
\varpi_T\,\circ_\phi =  \circ_\phi(\id_S \ot
\varpi_F) + \circ_\phi (\varpi_S \ot \id_T)
\end{equation}
of maps $\Markl(S)\otimes \Markl(F)\to \Markl(T)$. A {\em dg Markl
  operad\/} is a pair $(\Markl, \pa)$ consisting of a Markl operad
$\Markl$ and a degree $-1$ derivation $\pa$ such that $\pa\pa = 0$.

\end{definition}

In the following proposition, $\Free(E)$ is the free Markl operad
generated by a $1$-connected collection $E$, which is considered as a
subcollection of $\Free(E)$ via the inclusion $\iota :E
\hookrightarrow \Free(E)$ in~\eqref{Kdy bude hotovy Terej?}.

\begin{proposition}
\label{a zaskakuje pastorek}
Each degree $s$ map of collections $\zeta :E \to \Free(E)$ uniquely
extends to~a~degree $s$ derivation $\varpi$  of the free operad \,
$\Free(E)$ satisfying $\varpi|_E = \zeta$.
\end{proposition}

\begin{proof}
Our proof follows the scheme of the proof of an analogous statement for operad
algebras~\cite[Proposition~12.3.11]{loday-vallette}. To avoid cumbersome but conceptually insignificant sign issues, we
assume that $s=0$. The modification for a general $s$ is
indicated at the end of the proof.  
Given a Markl operad $\Markl$, we give the component-wise
direct sum $\Markl \oplus \Markl$
of the underlying collections  the structure of an operad with
composition laws $\circ^\oplus_\phi$ given by
\[
\circ^\oplus_\phi (a' \oplus b',a'' \oplus b'')
:= \circ_\phi(a',a'') \oplus\big (\circ_\phi(a',b'') +
(-1)^{|a''|\cdot |b'|}
 \circ_\phi(a'',b') \big),
\]
for $a',a'' \in \Markl(S),\ b',b'' \in \Markl(F)$,
where  $\circ_\phi$ is the composition law of $\Markl$ associated to
an elementary morphism \redukce{$F\fib T \stackrel\phi\to S$}. Then
clearly $\varpi : \Markl \to \Markl$ is a degree $0$ derivation if and only if
the map $(\id,\varpi): \Markl \to  \Markl \oplus \Markl$ is a morphism
of Markl operads. 

Let $\iota:E \to \Free(E)$ be the 
inclusion~\eqref{Kdy bude hotovy Terej?}.
The map $(\iota ,\zeta) : E \to \Free(E)
\oplus \Free(E)$ 
extends to a~unique operad morphism $\Free(E) \to  \Free(E)
\oplus  \Free(E)$ by the freeness of
$\Free(E)$. This extension is of the form $(\id ,\varpi)$, where 
$\varpi$ is the required derivation. If $s \not= 0$, we 
replace the direct sum $\Markl \oplus \Markl$ by $\Markl \oplus \susp^s \Markl$ and
introduce the canonical Koszul signs.
\end{proof}

\begin{definition}
\label{Jarca byla na chalupe sama.}
A {\em Markl $\ttO$-cooperad\/} 
is a functor $\CO: \Iso \to \ttV$ equipped, for each
elementary morphism \redukce{$F\fib T \stackrel\phi\to S$} as in
Definition~\ref{plysacci_postacci}, 
with the ``partial coproduct''
\begin{equation}
\label{ten_prelet_jsem_podelal1}
\delta_{\phi}: \CO(T) \to  \CO(S)\otimes \CO(F).
\end{equation}
These operations must fulfill the axioms obtained by reversing the
arrows in the diagrams in Definition~\ref{markl} of Markl operads.
\end{definition}

\begin{example}
\label{Vse roztaje.}
Assume that $\Markl$ is a Markl operad whose components $\Markl(T)$, $T
\in \ttO$, are either
finite-dimensional, or non-negatively, or non-positively graded
dg vector spaces of finite type.  Then its component-wise 
linear dual $\Markl^* := \{\Markl(T)^*\}_{T \in \ttO}$ is a Markl
cooperad. The partial coproducts~\eqref{ten_prelet_jsem_podelal1} 
are given by dualizing the
operations~\eqref{ten_prelet_jsem_podelal}, i.e.\
\[
\delta_\psi := \circ_\psi^* : \Markl(T)^* \to (\Markl(S) \ot \Markl(F))^* \cong
\Markl(S)^* \ot \Markl(F)^*.
\]
The finiteness assumption guarantees that the inclusion 
$(\Markl(S) \ot \Markl(F))^*   \hookleftarrow \Markl(S)^* \ot
\Markl(F)^*$ is an isomorphism. Since
$\Markl$ is a  $\ttV$-presheaf on $\Iso$, its dual is a functor $\Markl^*: \Iso
\to \ttV$ as required in our definition of a cooperad.
\end{example}

Markl cooperad $\CO$ is {\em  counital\/} if there is given, for each
trivial $U$, a ``counit'' $\epsilon_U
:\CO(U) \to \bfk$ such that the diagram
\[
\xymatrix{
\CO(T) \ar[r]^(.4){\delta_!}   &
\CO(U) \ot \CO(T) \ar[d]^{\epsilon_U \ot \id}
\\
 \CO(T) \ar@{=}[u]  \ar@{=}[r]^(.42)\cong   & \bfk \ot \CO(T) 
}
\]
commutes whenever $T$ is such that $e(T) \geq 1$ and 
$T \fib T \stackrel!\to U$  the unique map.
By reversing the arrows of~\eqref{proc_ty_lidi_musej_porad_hlucet} we
obtain a map  
$\vartheta(T,u) :  \CO(T) \to  \CO(F)$
for each $T$ with $e(T) \geq 1$ and \redukce{$F \fib T \stackrel!\to u$}, with
$u$ a local terminal object. 

\begin{definition}
\label{Jak bude?}
A counital Markl cooperad $\CO$ is {\em strictly counital\/} if all
the maps  $\vartheta(T,u)$ are identities. It is {\em
  $1$-connected\/} if  $\epsilon_U
:\CO(U) \to \bfk$ is an isomorphism for each trivial $U \in \ttO$.
\end{definition}
 
From this moment on, all Markl cooperads
will be tacitly assumed to be {\em strictly counital\/} and \hbox{\em
    $1$-connected\/}. The main source of examples will be
  component-wise linear duals of  strictly unital
  $1$-connected Markl operads that satisfy the finiteness assumption of
  Example~\ref{Vse roztaje.}.

\subsection{The cobar construction}
The underlying collection of a Markl cooperad $\CO$ is a covariant
functor $\Iso \to \ttV$. Since $\Iso$ is a groupoid, we may consider
$\CO$ also as a $\ttV$-presheaf with the
contravariant action of $\omega \in \Iso$ given by $\omega^* :=
(\omega^{-1})_*$.  With this convention in mind, the family $\uCO =
\{\CO(T)\}_{T \in \ttO}$ defined by
\[
\uCO(T) := 
\begin{cases}
\CO(T) &\hbox {if $e(T) \geq 1$}
\\
0 & \hbox {otherwise,} 
\end{cases}
\]
and also its component-wise desuspension $\dus$,
becomes a
$1$-connected $\ttO$-collection in the sense of
Definition~\ref{Jarka_dnes_u_lekare}, so it make sense to form the
free operad $\Free(\dus)$ which it generates.
We denote the restrictions of the partial coproducts  $\delta_\phi$
of~\eqref{ten_prelet_jsem_podelal1} 
to $\uCO$ by
\[
\overline\delta_{\phi}: \uCO(T) \to  \uCO(S)\otimes \uCO(F),\ \hbox {for }
F \fib T \stackrel\phi\to S.
\]
Note that $\uCO$ with the above operations is an analog of the
coaugmentation coideal of an coaugmented coalgebra featured in the
classical cobar construction.
We finally define, for every elementary map $F \fib T \stackrel\phi\to
S$, degree~$-1$ operations
\[
\dusd\phi :=  (\susp \ot \susp) \, \overline \delta_{\phi}\, \desusp: 
\dus(T) \to  \dus(S)\otimes \dus(F).
\]

A prominent r\^ole in the calculations below will be played by
labeled towers of height~$2$. Recall that such a tower
$\bftau = (\ell,\bfT) \in \lTw^2(X)$ consists of an elementary
morphism $\tau:T \to S$ and an isomorphism $\ell : X \stackrel\cong\to
T$ which can always be replaced by a \qb. Its
associated fiber sequence is the pair $(F,S)$, with $F$ the
unique nontrivial fiber of~$\tau$. We will denote such a tower by
\begin{equation}
\label{Zitra dopoledne zavolam co je s Terejem.}
\twr XF{\ T\,.}S\ell\tau\bftau
\end{equation}
In the rest of this section we also assume that the groupoid
$\lTw^2(X)$ has, for each $X \in \ttO$, only {\em finitely many\/} 
connected components. 

Let us denote by $\C(\CO)$ the free operad $\Free(\dus)$ generated by
the $1$-connected collection~$\dus$, with the natural grading $\C(\CO)
=\bigoplus_{n \geq 0} \C^n(\CO)$ inherited from $\Free(\dus)$  
by $\C^n(\CO) := \Free^n(\dus)$.
We are going to introduce a degree $-1$ 
derivation $\pa_{\C} :  \C(\CO) \to \C(\CO)$
that squares to zero, 
thus making $\C(\CO)$ a differential graded
Markl operad. Referring to Proposition~\ref{a zaskakuje pastorek}, 
$\pa_{\C}$~will be defined as the unique extension of its restriction
to the generators of $\C(\CO)$.
For $x \in \dus(X) \cong \C^1(\CO)(X)$ and $X\in \ttO$, this
restriction is the finite~sum
\begin{equation}
\label{Terej bude mit elektrickou zarazku.}
\pa_{\C}(x) = \sum_{[\bftau] \in \pi_0(\lTw^2(X))} \pa_\bftau(x),
\end{equation}
with $\bftau$ running over representatives of the connected components
of the groupoid $\lTw^2(X)$. If $\bftau = (\ell, \bfT)$ as
in~\eqref{Zitra dopoledne zavolam co je s Terejem.} is such a
representative, we put
\begin{equation}
\label{Jarka mne chce prestavet byt.}
\pa_\bftau (x)   := \ell^*\! \circ_\tau\! (\dusd\tau(\ell_* x))  \in \C^2(\CO)(X),
\end{equation}
where $\circ_\tau : \C(\CO)(S) \otimes \C(\CO)(F)\to
\C(\CO)(T)$ is the partial composition in
the operad $\C(\CO)$ associated to the elementary map $\tau$. 
In detail, $\ell_*x \in \dus(T)$, thus $\dusd\tau(\ell_* x) \in
\dus(S) \ot \dus(F)$ which is canonically a subspace of $ \C(\CO)(S)
\otimes \C(\CO)(F)$, so the application of $\circ_\tau$ makes sense.

We must show that $\pa_\bftau (x)$ does not depend on the choices 
of the representatives of the
connected components of $\lTw^2(X)$. Suppose thus that $\bftau' = (\ell',\bfT')$ 
and $\bftau'' =  (\ell'',\bfT'')$ are two
isomorphic labeled towers as in
\[
\xymatrix{&X  \ar[ld]_{\ell'}^\cong  \ar[rd]^{\ell''}_\cong     &
\\
T'  \ar[d]_{\tau'} \ar[rr]^{\sigma_1}_\cong  &&T'' \ar[d]^{\tau''}
\\
S' \ar[rr]^{\sigma_2}_\cong   &&S''
}
\]
with associated fiber sequences $(F',S')$
resp.~$(F'',S'')$. Denote by
$F = \inv{\phi}(j)$  the unique nontrivial fiber of the auxiliary morphism 
$\phi := \sigma_2 \tau' : T' \to S''$. By the strong counitality
assumption, we have $\CO(F) = \CO(F'')$, so the dual
of commutative diagram~\eqref{Ve_ctvrtek_letim_do_Prahy.} leads~to
\[
\xymatrix@C=4em@R=1em{
\dus(S'') \ot \dus(F'')
& \dus(T'') \ar[l]_(.35){\dusd{\tau''}}&
\\
&&  \ar[lu]_{\ell''_*} \ar[ld]^{\ell'_*}     \dus(X)\,.
\\
\ar[uu]^{ {\sigma_2}_* \ot \, {\sigma_{1j}}}_\cong  \dus(S') \ot
\dus(F')& \ar[l]_(.35){\dusd{\tau'}}  \dus(T') \ar[uu]_{{\sigma_1}_*}^\cong&
}
\]
We conclude from this diagram that $\dusd{\tau'}(\ell'_*x) \in \dus(\bfT')$
and~$\dusd{\tau''}(\ell''_*x)  \in \dus(\bfT'')$ are related  
by the isomorphism $\bfsigma^*$ associated to $(\sigma_2,\sigma_1)$
as in~\eqref{Dnes snad uvidim Tereje.},    
thus  $\pa_{\bftau'} (x)  = \ell'^*\! \circ_{\tau'}\! (\dusd{\tau'}(\ell'_* x))$
and~$\pa_{\bftau''} (x)  =  \ell''^*\! \circ_{\tau''}\! (\dusd{\tau''}(\ell''_*
x))$ are the same.

\begin{proposition}
The derivation $\pa_{\C} : \C(\CO) \to \C(\CO)$ introduced above squares to zero.
\end{proposition}

\begin{proof}
It is simple to verify, using the defining equation~(\ref{zase mi
  nejde aktuator}), that the square of an odd-degree derivation is a
derivation again. In particular, $\pa_{\C}^2$ is a derivation, thus it
suffices, by Proposition~\ref{a zaskakuje pastorek}, to verify that
$\pa_{\C}^2(x) = 0$ for  $x \in \dus(X)$, $X\in \ttO$, i.e.\ that 
\[
\sum_\bftau
\pa_{\C} \pa_\bftau (x) = \sum_\bftau 
\pa_{\C} \ell^* \! \circ_\tau \! (\dusd\tau(\ell_* x)) =
\sum_\bftau  \ell^* \pa_{\C}  \! \circ_\tau \! (\dusd\tau(\ell_* x))=  0.
\]
In the above display, as well as in the following ones, we will not
specify the summation range where it is clear from the context.
By the derivation
rule~(\ref{zase mi nejde aktuator}) we have
\[
\sum_\bftau \ell^*
\pa_{\C} \! \circ_\tau \! (\dusd\tau(\ell_* x)) =\sum_\bftau 
\ell^*\! \circ_\tau \! (\pa_{\C} \ot \id) (\dusd\tau(\ell_* x) ) +
\sum_\bftau \ell^*
\! \circ_\tau \! (\id \ot \pa_{\C}) (\dusd\tau(\ell_* x) ),
\]
thus $\pa_{\C}^2(x)$ decomposes as $\pa_{\C}^2(x) = \pa^2_L(x) + \pa^2_R(x)$, where
$\pa^2_L(x)$, resp.~$\pa^2_R(x)$, is the left, resp.~right, term at the
left-hand side of the above equation.

\noindent 
{\it Analyzing $\pa^2_L(x)$.}
 Let us focus on $\pa^2_L(x)$ first. Expanding
further using definitions, we get
\begin{align*}
\pa^2_L(x) =& \sum_{\bfpsi'} 
{\ell^{\, *}_{\psi'}}
\circ_{\psi'} (\pa_{\C} \ot \id) (\dusd{\psi'}({\ell^{\, *}_{\psi'}} x) )=
\sum_{\bfphi,\bfpsi'}
{\ell^{\, *}_{\psi'}}
\circ_{\psi'} (\pa_{\bfphi} \ot \id) (\dusd{\psi'}({\ell_{\psi'}}_* x) )
\\
&=
\sum_{\bfphi,\bfpsi'} 
{\ell^{\, *}_{\psi'}}
\circ_{\psi'} ({\ell^{\, *}_{\phi}}\circ_\phi \dusd\phi {\ell_{\phi}}_* \ot \id) 
(\dusd{\psi'}({\ell_{\psi'}}_* x) )
\\
&= \sum_{\bfphi,\bfpsi'} {\ell^{\, *}_{\psi'}} \circ_{\psi'}
({\ell^{\, *}_{\phi}} \ot \id)
 (\circ_\phi \ot \id)
(\dusd\phi  \ot \id)( {\ell_{\phi}}_* \ot
     \id)(\dusd{\psi'}({\ell_{\psi'}}_* x)),
\end{align*}
where $\bfpsi'$ and $\bfphi$ run over the representatives of the
isomorphism classes  of labeled towers of the form
\[
\twr X{F}{\ T'\ ,}{H'}{\ell_{\psi'}}{\psi'}{{\boldsymbol \psi'}}
\twr {H'}B{\ H\,.}S{\ell_\phi}{\phi}{{\boldsymbol \phi}}
\]
Consider now the diagram
\begin{equation*}
\xymatrix@R=-.3em@C=1em{
&F  \ar@{~>}[r]^\id  &F
\\
&\triangledown&\triangledown 
\\
X  \ar[r]^{\ell_{\psi'}}  &T' \ar[ddddd]_{\psi'}  \ar[r]^{\tilde
  \ell_\phi}   &T\ar[ddddd]^{\psi}
\\
\\ & \hskip 5em \BU&
\\
\\
\\
\rule{0pt}{2em}&H'\ar[r]^{\ell_\phi}& H  \ar[dddddd]^\phi 
\\
\\
\\
\\
\\
\\
&&S
}
\end{equation*}
whose square was obtained using the blow-up axiom with the
condition that the prescribed maps of fibers are identities, which is
symbolized by the wavy arrow decorated by
$\id$. From~\eqref{Ve_ctvrtek_letim_do_Prahy.} we conclude that 
$\circ_{\psi'} (\ell_\phi^{\, *} \ot \id) = {{\tilde \ell}_\phi}^{\, *}
\circ_\psi$ and, dually, 
$({\ell_{\phi}}_* \ot   \id)\delta_{\psi'} 
=  \delta_\psi\hbox{${{{\tilde \ell}_\phi}}$}_*$, thus also 
$({\ell_{\phi}}_* \ot   \id)\dusd{\psi'} 
=  \dusd\psi\hbox{${{{\tilde \ell}_\phi}}$}_*$, and therefore
\[
\pa^2_L(x) = \sum_{\bfphi,\bfpsi'} 
({\hbox{${{{\tilde \ell}_\phi}}$}
{\ell_{\psi'}}})^* \circ_{\psi'} (\circ_\phi \ot \id)
(\dusd\phi  \ot \id)
\dusd\psi({\hbox{${{{\tilde \ell}_\phi}}$}
{\ell_{\psi'}}})_* (x).
\]
This can clearly be rewritten as
\begin{equation}
\label{Zitra davam Zoom talk.}
\pa^2_L(x) =
\sum_{\bfphi,\bfpsi} 
\ell_{*} \!
 \circ_{\psi}\! (\circ_\phi \ot \id)
(\dusd\phi  \ot \id)\,
\dusd\psi\ell_* (x),
\end{equation}
where $\bfphi$, $\bfpsi$ runs over  the representatives of the
isomorphism classes  of labeled towers as in
\begin{equation}
\label{Snehu bude az moc.}
\twr XF{\ T\ ,}H{\ell}\psi{\boldsymbol \psi}
\twr HB{\ H\,.}S{\id}\phi{{\boldsymbol \phi}}
\end{equation} 
Let us further decompose
\[
\pa_L^2(x) = \pa_L^2(x)'+ \pa_L^2(x)'',
\]
where $\pa_L^2(x)'$, resp.~$\pa_L^2(x)''$, is the part of the sum
in~\eqref{Zitra davam Zoom talk.} running over $\phi$ and $\psi$ with
joint, resp.~disjoint, fibers in the sense of Definition~\ref{d3}. We
are going to show that $\pa_L^2(x)'' = 0$. 

To start, let $k,i \in |S|$, $k \not= i$. Recall that
$(\phi,\psi)$ is a $(k,i)$-pair if there is $j \in \inv{|\phi|}(k)$ such that
$\inv \psi(j)$ is the only nontrivial fiber of $\psi$, 
and  $\inv \phi(i)$ the only nontrivial fiber of $\phi$.
This results in the decomposition
\[
\pa_L^2(x)'' = \sum_{k \not=  i} \pa_L^2(x)_{(k,i)}''
\]
where $\pa_L^2(x)_{(k,i)}''$ is the part of the sum in~\eqref{Zitra davam
  Zoom talk.} taken over  $(k,i)$-pairs. Out next aim will be to prove that 
\begin{equation}
\label{Kimci}
\pa_L^2(x)''_{(k,i)} + \pa_L^2(x)''_{(i,k)} =0,
\end{equation}
for each  $k \not=i \in |S|$, using Lemma~\ref{Dnes naposledy jede Lednacek.}.
The commutative diagram 
\[
\xymatrix@R = 1em@C=4em{& {P'} \ar[dr]^{(\phi',\, i)} &
\\
T\ar[dr]^{(\psi'',\, l)} \ar[ur]^{(\psi',\, j)} && S
\\ 
&{P''}\ar[ur]^{(\phi'',\, k)}&
}
\]
leads, with the help of the dual of~\eqref{Napadne jeste tento rok
  snih?}, to
\[
\xymatrix@C=1em@R=1.3em{
\CO(P') \ot \CO(F) \ar[rr]^(.42){\delta_{\phi'} \ot \id}
&&\CO(S) \ot \CO(G) \ot \CO(F)   \ar[dd]^{\id \ot {\, \rm transposition}}
\\
\CO(T)\ar[d]_{\delta_{\psi''}}\ar[u]^{\delta_{\psi'}}  &&
\\
\CO(P'') \ot \CO(G) \ar[rr]^(.42){\delta_{\phi''} \ot \id} 
&&\CO(S) \ot \CO(F) \ot \CO(G) \, .
}
\]
As a simple application of the Koszul sign rule we obtain the diagram
\[
\xymatrix@C=1em@R=1.5em{
\dus(P') \ot \dus(F) \ar[rr]^(.42){\dusd{\phi'} \ot \id}
&&\dus(S) \ot \dus(G) \ot \dus(F)   \ar[dd]^{\id \ot {\, \rm transposition}}
\\
\dus(T)\ar[d]_{\dusd{\psi''}}\ar[u]^{\dusd{\psi'}}  & \hbox{\rm commutes up
  to $-1$}&
\\
\dus(P'') \ot \dus(G) \ar[rr]^(.42){\dusd{\phi''} \ot \id} 
&&\dus(S) \ot \dus(F) \ot \dus(G)
}
\]
commuting up to multiplication by $-1$. Combining it
with diagram~\eqref{Napadne jeste tento rok snih?} applied to the composition
laws of the free Markl operad $\C(\CO)$, we get
\[
\xymatrix{
\dus(P') \ot \dus(F) \ar[r]^(.42){\dusd{\phi'} \ot \id}
&\dus(S) \ot \dus(G) \ot \dus(F)\ar[r]^(.58){\circ_{\phi'} \ot \id} 
& \dus(P') \ot \dus(F) \ar[d]^{\circ_{\psi'}}
\\
\dus(T)\ar[d]_{\dusd{\psi''}}\ar[u]^{\dusd{\psi'}}  & \hbox{\rm commutes up
  to $-1$}& \dus(T)
\\
\dus(P'') \ot \dus(G) \ar[r]^(.42){\dusd{\phi''} \ot \id} 
&\dus(S) \ot \dus(F) \ot \dus(G)
\ar[r]^(.58){\circ_{\phi''} \ot \id}  & \dus(P'') \ot \dus(G) \,,
\ar[u]_{\circ_{\psi''}}
}
\]
expressing the equation
\[
\circ_{\psi'} (\circ_{\phi'} \ot \id)(\dusd{\phi'}  \ot \id)
\dusd{\psi'}  + 
\circ_{\psi''} (\circ_{\phi''} \ot \id)(\dusd{\phi''}  \ot \id)
\dusd{\psi''} =0.
\]
Invoking~\eqref{Zitra davam Zoom talk.}, 
we conclude that each summand of $\pa_L^2(x)''_{(k,i)}$ has its counterterm in
$\pa_L^2(x)''_{(i,k)}$. 
This establishes~(\ref{Kimci}), and therefore $\pa_L^2(x)'' = 0$, so we
may assume that the sum in~\eqref{Zitra davam Zoom talk.} runs over isomorphism
classes of pairs~(\ref{Snehu bude az moc.}) with $\phi$ and $\psi$
having {\em joint fibers\/}.

\noindent 
{\it Analyzing $\pa^2_R(x)$.\/}
Let us perform a similar analysis of $\pa^2_R(x)$. We have
\begin{align*}
\pa^2_R(x) = &
 \sum_{\bfalpha'} \ell^{\, *}_{\alpha'} \circ_{\alpha'}(\id \ot
            \pa_{\C})\dusd{\alpha'}{\ell_{\alpha'}}_*(x)
= \sum_{\bfbeta,\bfalpha'} \ell^{\, *}_{\alpha'} \circ_{\alpha'}(\id \ot
            \pa_{\bfbeta})\dusd{\alpha'}{\ell_{\alpha'}}_*(x)
\\
&= \sum_{\bfbeta,\bfalpha'} \ell^{\, *}_{\alpha'} \circ_{\alpha'}(\id
     \ot
\ell_\beta^{\, *} \circ_\beta \dusd\beta {\ell_\beta}_*)
\dusd{\alpha'}{\ell_{\alpha'}}_*(x)
\\
&= 
\sum_{\bfbeta,\bfalpha'} \ell^{\, *}_{\alpha'} \circ_{\alpha'}
(\id \ot \ell_\beta^{\, *})(\id \ot  \circ_\beta)(\id \ot
     \dusd\beta)(\id \ot  
	 {\ell_\beta}_*)\dusd{\alpha'}{\ell_{\alpha'}}_*(x)\,,
\end{align*}
where $\bfalpha'$  and $\bfbeta$ run over  representatives of the
isomorphism classes  of labeled towers  as in
\[
\twr X{A'}{\ T'\ ,}S{\ell_{\alpha'}}{\alpha'}{{\boldsymbol \alpha'}}
\twr {A'}F{\ A\,.}B{\ell_\beta}\beta{{\boldsymbol \beta}}
\]
Let us construct, out of $\bfalpha'$ and $\bfbeta$, the commutative diagram
\[
\xymatrix@R=-.3em@C=1em{
&&A'  \ar@{~>}[rr]^{\ell_\beta}  &&A
\\
&&\triangledown&&\triangledown 
\\
X  \ar[rr]^{\ell_{\alpha'}}&  &T' \ar[dddr]_{\alpha'}
\ar[rr]^{\underline \ell_\beta}
& &T\ar[dddl]^{\alpha}
\\ 
&&&    \WBU&
\\ 
&&&  \rule{0pt}{2em} &&
\\
&&& S& 
}
\]
whose triangle is given by the weak blow-up axiom, under the condition
that the prescribed map between the only nontrivial fibers is $\ell_\beta$ and
the remaining maps of the fibers are identities. 
Using~\eqref{Ve_ctvrtek_letim_do_Prahy.} and its dual as before, we
rewrite $\pa^2_R(x)$ as 
\[
\pa^2_R(x) = \sum_{\bfbeta,\bfalpha'}
(\underline \ell_\beta \ell_{\alpha'})^*\!
\circ_\alpha \!(\id \ot \circ_\beta)(\id \ot \dusd\beta)\dusd\alpha
(\underline \ell_\beta \ell_{\alpha'})_*(x)
\]
which is the same as
\begin{equation}
\label{2-2-3-4}
\pa^2_R(x) = \sum_{\bfbeta,\bfalpha}
\ell^{*} \!
\circ_\alpha  \!(\id \ot \circ_\beta)(\id \ot \dusd\beta)\,\dusd\alpha
\ell_*(x),
\end{equation}
where $\bfalpha$ and $\bfbeta$ run over representatives of the 
isomorphism classes of towers as in
\begin{equation}
\label{po dvou dnech}
\twr XA{\ T\, ,}S{\ell}\alpha{\boldsymbol \alpha}
\twr AF{\ A\,.}B{\id}\beta{\boldsymbol \beta} 
\end{equation}
The weak blow-up axiom  produces a unique factorization
$\alpha = \phi\psi$ in the diagram
\begin{equation}
\label{na lyzich do Osecka}
\xymatrix@R=-.3em@C=1em{
A  \ar@{~>}[rr]^{\beta}  &&B
\\
\triangledown&&\triangledown 
\\
S \ar[ddr]_{\alpha}
\ar[rr]^{\psi}
& &H\ar[ddl]^{\phi}
\\ 
&  \raisebox{.7em}{\WBU} \rule{0pt}{3em} &&
\\
& T& 
}
\end{equation}
in which $\alpha$ is elementary and the map between the only
nontrivial fibers of $\alpha$ resp.~$\phi$ is $\beta: A \to B$, as
symbolized by the wavy arrow. Notice that the pair $(\phi,\psi)$ has
joint fibers.
The dual of~\eqref{vymena} gives
\[
\xymatrix@R = 1.3em@C=1em{
&\CO(S)\otimes \CO(A)
\ar[rd]^{\id \ot \delta_\beta}
&
\\
\CO(T) \ar[ur]^(.4){\delta_\alpha}\ar[dr]_(.4){\delta_\psi}
&&\CO(S)\otimes \CO(B)\otimes \CO(F) 
\\ 
&\CO(H)\otimes \CO(F) 
\ar[ur]_{\delta_\phi \ot \id}
&&
}
\]
which, combined with the Koszul sign rule, leads to
\[
\xymatrix@R = 1.3em@C=1em{
&\dus(S)\otimes \dus(A)
\ar[rd]^{\id \ot \dusd\beta}
&
\\
\dus(T) \ar[ur]^(.4){\dusd\alpha}\ar[dr]_(.4){\dusd\psi}
&\hbox{\rm commutes up
  to $-1$}&\dus(S)\otimes \dus(B)\otimes \dus(F) 
\\ 
&\dus(H)\otimes \dus(F) 
\ar[ur]_{\dusd\phi \ot \id}
&&
}
\]
which commutes up to multiplication by $-1$. Combining it with
diagram~\eqref{vymena} for the composition laws of the free operad 
$\C(\CO)$, we obtain the diagram
\[
\xymatrix@R = 1.3em@C=.5em{
&\dus(S)\otimes \dus(A)
\ar[rd]^{\id \ot \dusd\beta}
&& {\dus(S)\otimes \dus(A)} 
\ar[dr]^(.65){\circ_\alpha} &
\\
\dus(T) \ar[ur]^(.4){\dusd\alpha}\ar[dr]_(.4){\dusd\psi}
&\hbox{\rm commutes up
  to $-1$}&\dus(S)\otimes \dus(B)\otimes \dus(F) 
\ar[dr]_{\circ_\phi\ot \id} \ar[ur]^{\id \ot \circ_{\beta}}  &&  \dus(T)
\\ 
&\dus(H)\otimes \dus(F) 
\ar[ur]_{\dusd\phi \ot \id}
&&{\dus(H)\otimes \dus(F)}\ar[ur]_(.65){\circ_\psi}&
}
\]
which translates into the equation
\begin{equation}
\label{Jarka uz je 14 dni na chalupe.}
\circ_{\psi} (\circ_\phi \ot \id)(\dusd\phi \ot \id)\dusd\psi 
+ \circ_\alpha(\id \ot \circ_\beta)(\id \ot \dusd\beta) \dusd\beta=0.
\end{equation}

Notice that the assignment $ (\alpha,\beta) \mapsto (\phi,\psi)$
described by diagram~\eqref{na lyzich do Osecka} induces a
bijection between the set of isomorphism classes of towers~(\ref{po
  dvou dnech}) and~(\ref{Snehu bude az moc.}) with $ (\phi,\psi)$
having joint fibers. Its inverse
$(\phi,\psi) \mapsto (\alpha,\beta)$ produces $\alpha$ as the
composite $\phi\psi$ while $\beta$ is the induced map between the
unique nontrivial fibers as in
\[
\xymatrix@R=-.3em@C=1em{
A  \ar@{~>}[rr]^{\beta}  &&B
\\
\triangledown&&\triangledown 
\\
S \ar[ddr]_{\phi\psi}
\ar[rr]^{\psi}
& &H\ar[ddl]^{\phi}
\\ 
&  \rule{0pt}{2em} &&
\\
& \phantom{.}T\,.& 
}
\] 
Therefore, by~(\ref{Jarka uz je 14 dni na chalupe.}), each summand
$\ell^* \!\circ_{\psi}\! (\circ_\phi \ot \id)(\dusd\phi \ot
\id)\dusd\psi \ell_*(x)$ 
of~\eqref{Zitra davam Zoom talk.} possesses a unique counterterm $\ell^* \!\circ_\alpha \!(\id \ot \circ_\beta)
(\id \ot \dusd\beta)\dusd\beta \ell_*(x)$ in the sum~(\ref{2-2-3-4}),
which proves that 
\[
\pa_{\C}^2(x) = \pa^2_L(x) + \pa^2_R(x) = 0
\] 
as claimed.
\end{proof}

\begin{definition}
\label{Za chvili zavolam Taxovi.}
We call the dg operad $\C(\CO) = (\C(\CO),\pa_{\C})$ the {\em cobar
  construction\/} of a~strictly counital $1$-connected
Markl cooperad $\CO$.
\end{definition}

\section{The dual dg operad and Koszulity}
\label{Prevezu Tereje?}

The dual dg operad of a traditional operad was introduced in Section~3
of~\cite{ginzburg-kapranov:DMJ94}, and the Koszulity for
quadratic operads
in Section~4 of the same article. In this section we generalize
these notions to Markl operads over operadic categories and prove 
that operads whose algebras are the most common structures are~Koszul.

\begin{remark}
The duality for the (colored) operad whose algebras are non-unital
nonsymmetric operads was 
studied in 2002 by van der Laan~\cite{van}. Nineteen
years later, Dehling and Vallette
proved in~\cite{deh-val} that the
colored operad whose algebras are the classical (symmetric) operad is
curved Koszul. Koszulity of the groupoid-colored
operad whose algebras are modular operads was proved by Ward~\cite{Ward}.
A general approach to graph-based operadic 
structures in the language of
Feynman categories was suggested by Kaufmann and Ward in~\cite{KaWa}.
\end{remark}

The base operadic category $\ttO$ is required to fulfill \AA.
Let $\Markl$ be a strictly unital $1$-connected Markl operad satisfying
the assumptions of Example~\ref{Vse roztaje.}, so that the component-wise 
linear dual $\Markl^* := \{\Markl(T)^*\}_{T \in \ttO}$ is a   strictly counital
$1$-connected  Markl cooperad.
 
\begin{definition}
\label{Nebrzdi mi kolo.}
 The {\em dual dg
  operad\/} of a Markl operad $\Markl$ as above is the dg operad
\[
\D(\Markl) = (\D(\Markl),\pa_\D)   := (\C(\Markl^*),\pa_{\C}),
\] 
the cobar construction of the component-wise
linear dual of $\Markl$. 
\end{definition}

Assume that $\Markl =  \Free(E)/(R)$  is a quadratic Markl operad as in
Definition~\ref{zitra_seminar} and $\Markl^!$ its
Koszul dual, cf.\ Definition~\ref{vice nez 10 000 pripadu denne}.
To introduce the Koszulity, we
start from the injection $\susp E \hookrightarrow  \Markl^!$ of
collections defined as the composite
\[
\susp E \hookrightarrow \Free(\susp E) 
\twoheadrightarrow \Free(\susp E)/(R^\perp) = \Markl^!.
\]
Its linear dual  ${\Markl^!}^* \twoheadrightarrow\ \susp
E$ desuspends to a map $\pi:\ \desusp {\Markl^!}^*
\twoheadrightarrow E$. The related twisting
morphism  $\desusp {\Markl^!}^*  \to \Markl$, defined as
the composite
\[
 \desusp  {\Markl^!}^*\stackrel\pi\twoheadrightarrow  E 
 \hookrightarrow \Free(E) 
\twoheadrightarrow \Free(E)/(R) = \Markl,
\]
extends, by the freeness of $\Free(\desusp
{\Markl^!}^*)$, to a  morphism
$\rho_\Markl : \Free(\desusp {\Markl^!}^*) \to \Markl$  of Markl $\ttO$-operads. 
One verifies by direct calculation:

\begin{proposition}
The morphism $\rho_\Markl$ induces the {\em
  canonical map\/} 
\begin{equation}
\label{Jarce_prijede_M1_a_bude_do_soboty.}
\can_\Markl:
\D(\Markl^!) = (\Free(\desusp {\Markl^!}^*), \pa_\D)    \longrightarrow (\Markl,0)
\end{equation}  
of Markl dg $\ttO$-operads.
\end{proposition}

\begin{definition}
\label{hodnoceni snad odlozeno}
A quadratic Markl $\ttO$-operad $\Markl$ is {\em Koszul\/} if the 
canonical map~\eqref{Jarce_prijede_M1_a_bude_do_soboty.} is 
a~component-wise homology isomorphism.
\end{definition}

\begin{remark}
In the ``classical'' operad theory one proves that a quadratic operad is
Koszul if and only if its Koszul dual is Koszul~\cite[Proposition~4.1.4]{ginzburg-kapranov:DMJ94}. We believe
the same is true also in our setup, but postpone the proof for future work. 
\end{remark}

In the rest of this section we establish the Koszulity of some of the binary
quadratic operads introduced in 
Sections~\ref{zitra_budu_pit_na_zal}--\ref{Ta_moje_lenost_je_strasna.}. Namely,
we prove the Koszulity of the operad $\term_\ggGrc$ whose algebras are
modular operads, of the operad $\termCGr$ whose algebras are cyclic operads, of
the operad  $\termRTre$ whose algebras are ordinary Markl operads and of the operad
$\Wheterm$ whose algebras are wheeled properads. The Koszulity of the operad
$\Pterm$ whose algebras are 
permutads was already established \hbox{in~\cite[Corollary~49]{markl:perm}.}

\begin{theorem}
\label{Privezu Tereje?}
The binary quadratic operads $\term_\ggGrc$, $\termCGr$, $\termRTre$
and $\Wheterm$ are Koszul.
\end{theorem}

\begin{proof}
In the proof we drop the subscripts of $\rho$ and $\can$ since they will
always be clear from the context.
We start with the most complex case of the terminal $\ggGrc$-operad  
$\term_\ggGrc$.
Our strategy will be to show that its minimal model $\minggGrc$
constructed in~\cite[Section~3]{BMO} is isomorphic to the dual dg operad $\D(\oddGr)$ of its
Koszul dual $\oddGr = \term_\ggGrc^!$, via an isomorphism compatible with the
resolving maps $\minggGrc \stackrel\rho\to \term_\ggGrc$
resp.~$\D(\oddGr) \stackrel\can\longrightarrow 
\term_\ggGrc$. 

It is not difficult to determine, using the presentation in the proof
of Theorem~\ref{vcera_jsem_podlehl},
the operad structure of  $\oddGr$. As we noticed in Remark~\ref{zitra_vycvik_vlekare}, given a genus-graded connected
graph $\Gamma \in \ggGrc$, the corresponding piece $\oddGr(\Gamma)$  can be
identified with the one-dimensional vector space  
$\det(\Gamma):= \det(\edg(\Gamma))$, the
determinant of the set of internal edges of $\Gamma$, placed in degree
$e(\Gamma)+1$, where by definition $e(\Gamma)$, the grade of $\Gamma$ is
the number of internal edges.
The unit map associated to a corolla $c \in \ggGrc$ is the canonical
isomorphism 
\[
\bfk = \det(\emptyset) \cong \oddGr(c).
\] 
Given two finite sets $S_1 = \{e^1_1,\ldots,e^1_a\}$ and 
$S_2 = \{e^2_1,\ldots,e^2_b\}$, we define the canonical isomorphism
\[
\omega_{S_1,S_2}:
\det(S_1 \sqcup S_2) \to \det(S_1) \ot \det(S_2)
\]
by
\[
\omega_{S_1,S_2}(e^1_1 \land \cdots \land e^1_a \land  
e^2_1 \land \cdots \land e^2_b) := (e^1_1 \land \cdots \land e^1_a) \ot
(e^2_1 \land \cdots \land e^2_b).
\] 
Notice that if $\Gamma \fib \Gamma' \stackrel\tau\to \Gamma''$ is an
elementary morphism, there is a natural isomorphism
\[
\edg(\Gamma') \cong \edg(\Gamma'') \sqcup \edg(\Gamma)
\]
of the sets of internal edges. The partial composition related to
$\tau$ then equals
\[
\xymatrix@1@C=5em{
\circ_\tau :
\oddGr(\Gamma'') \ot \oddGr(\Gamma) \cong \det(\Gamma'') \ot
\det(\Gamma)
\ar[r]^(.63){ \omega^{-1}_{\edg(\Gamma''),\edg(\Gamma)}} &\
\det(\Gamma') \cong \oddGr(\Gamma').
}
\]

Our next step will be to describe the dual dg operad
$\D(\oddGr)$ which, by definition, equals the cobar construction
$\C(\rGddo)$ of the cooperad $\rGddo$.
Paying attention to the Koszul sign rule, we determine the
cooperad structure operation of $\rGddo$ related to $\tau$ as
\[
\delta_\tau :=  (-1)^{(e(\Gamma'')+1)(e(\Gamma)+1)}\cdot
\circ^*_\tau
: \rGddo(\Gamma') \to \rGddo(\Gamma'') \ot \rGddo(\Gamma).
\]
The underlying collection of the reduced cooperad $\overline{\rGddo}$ is given by 
\[
\overline{\rGddo}(\Gamma) :=     
\begin{cases}
\rGddo(\Gamma)& 
\hbox {if $\Gamma$ has at least one internal edge}
\\
0& \hbox {otherwise.}
\end{cases}
\]
The desuspended operation
\[
\dusd\tau 
:= (\desusp \ot \desusp) \kompozice \overline\delta_\tau \kompozice \susp:
\desusp \overline\rGddo(\Gamma') \to \desusp \overline\rGddo(\Gamma'') \ot \overline\desusp \rGddo(\Gamma)
\]
is then,  with the help of the identification $\overline{\rGddo}(\Gamma) \cong
\det(\Gamma)$ for $\Gamma$ with at least one internal edge, 
described as 
\[
\dusd\tau =  (-1)^{(e(\Gamma'')+1) +(e(\Gamma'')+1)(e(\Gamma)+1)}\cdot
\omega_{\edg(\Gamma''),\edg(\Gamma)} =  (-1)^{e(\Gamma)(e(\Gamma'')+1)}\cdot
\omega_{\edg(\Gamma''),\edg(\Gamma)} . 
\]
The new sign $ (-1)^{(e(\Gamma'')+1)}$ is the contribution of the
commutation of the first tensor factor  of the image of $\overline\delta_\tau
\kompozice \susp$ over the desuspension $\desusp$\,. 

Let $\bftau = (\ell,\bfT) \in \lTw^2(X)$ be the labeled tower
\[
\twr X{\Gamma}{\ \Gamma'\,.}{\Gamma''}\ell\tau\bftau
\]
We may assume that $\ell$ is a \qb\ by
Proposition~\ref{22_hodin_cesty}. 
Since \qb{s} obviously act trivially on
$\oddGr$ and thus also on $\desusp\overline{\rGddo}$, 
formula~\eqref{Jarka mne chce prestavet byt.} for the component $\pa_{\bftau}$ of the
differential~\eqref{Terej bude mit elektrickou zarazku.} associated to
$\bftau$ reads
\begin{subequations} 
\begin{equation}
\label{k}
\pa_{\bftau}(x) = 
(-1)^{e(\Gamma)(e(\Gamma'')+1)}\cdot \circ_\tau \,    
\omega_{\edg(\Gamma),\edg(\Gamma'')}(x) \in \Free^2(\desusp\overline{\rGddo})(X),
\end{equation}
for $x \in \desusp\overline\rGddo(X) =
\desusp\overline\rGddo(\Gamma)$.
Thus $\D(\oddGr) = (\Free(\desusp\overline{\rGddo}), \pa_\D)$, with
$\pa_\D$ given by the right-hand side of~\eqref{Jarka mne chce
  prestavet byt.} with $\pa_\bftau$ as in~\eqref{k}.

To describe the map $\can : \D(\oddGr) \to \term_\ggGrc$
notice that, for $\Gamma \in \ggGrc$ with precisely one internal edge,
the graded vector space
$\desusp\overline{\rGddo}(\Gamma)$ is {\em canonically\/} isomorphic
to $\bfk$ placed in degree $0$, since  it is the desuspension of the dual
of the determinant of a~one-point set. The operad morphism $\can$ is the unique extension of
the map $\desusp\overline{\rGddo} \to \term_\ggGrc$ of collections  whose
component $\desusp\overline{\rGddo}(\Gamma) \to \term_\ggGrc(\Gamma)$
is trivial if $\Gamma$ has at least two internal edges, and which is
the canonical isomorphism $\desusp\overline{\rGddo}(\Gamma) \cong
\bfk \cong \term_\ggGrc(\Gamma)$ if $\Gamma$ has exactly one internal edge.
This finishes the description of the dual dg operad $\D(\oddGr)$ and
the associated canonical map.

Recall that the underlying non-dg operad of the 
minimal model $\minggGrc$ of $\term_\ggGrc$
in~\cite{BMO} is the free operad $\Free(D)$ generated by the
collection $D$, which in fact coincides with $\desusp \overline{\rGddo}$, so
$\D(\oddGr) = \minggGrc$ as non-dg operads. With this
identification, the resolving map \hbox{$\rho: \minggGrc \to \term_\ggGrc$}
equals the  morphism $\can: \D(\oddGr) \to  \term_\ggGrc$ described
above.

The differential $\pa$ of the minimal model  translated to the formalism
used in this article is given by the 
sum in the right-hand side of~\eqref{Terej bude mit elektrickou
  zarazku.}, but the component
$\pa_{\bftau}$ is now
\begin{equation}
\label{l}
\pa_{\bftau}(x) = 
(-1)^{e(\Gamma)}\cdot \circ_\tau \,    
\omega_{\edg(\Gamma),\edg(\Gamma'')}(x) \in \Free^2(D)(X),
\end{equation}
\end{subequations} 
cf.\ formula~(18b) of~\cite{BMO}. We see that the expressions in~\eqref{k}
and~\eqref{l} agree up to the sign factor
$(-1)^{e(\Gamma'')e(\Gamma)}$. To compensate this discrepancy,  we define a map
$\chi : \overline\rGddo \to D$ of collections by
\[
\chi(x) := (-1)^{\frac{e(\Gamma)(e(\Gamma)-1)}2}\cdot x
\hbox { for } x \in \desusp \overline\rGddo(\Gamma)
\]
and denote by the same symbol also its unique extension $\chi :
\Free(\desusp \overline\rGddo)
\to \Free(D)$ to an operad morphism. It is simple to verify
that $\chi$ commutes with the differentials and that the
diagram
\[
\xymatrix@R=.6em{\D(\oddGr) \ar[rd]^\can  \ar[dd]^\cong_\chi   &
\\
&\term_\ggGrc
\\
\minggGrc \ar[ur]^{\rho}&
}
\] 
of dg operads commutes. Since $\rho$ is a component-wise homology isomorphism
by~\cite[Theorem~31]{BMO}, so is $\can$. This establishes the
Koszulity of $\term_\ggGrc$.

The proofs of the remaining cases are similar. One describes the dual
dg operad by an obvious modification of the method above, and compares
it with the corresponding minimal models described in Theorem~32 and
Sections~3.4 and 3.5 of~\cite{BMO}.
\end{proof}

\appendix
\section {(Odd) modular operads, and classical Markl operads.}
\label{a1}

In this appendix we recall 
three structures referred to in this work. 
All definitions given here  are standard today, see
e.g.~\cite{getzler-kapranov:CompM98,markl-shnider-stasheff:book}, 
so the purpose is merely to fix the notation and terminology.
Let $\fSet$ denote the category of finite sets and $\Vect$,
the category of graded vector spaces.
For $S_1,S_2 \in \fSet$, we write $S_1 \sqcup S_2 \in  \fSet$ to denote the
disjoint union, the notation implying that $S_1$ and $S_2$ are disjoint.
When we write e.g.~``elements $a,b,c$'' we tacitly assume that $a,b$
and $c$ are mutually~distinct.

Recall that a {\em modular module\/} is a functor $\fSet \!\times\! \bbN \to
\Vect$, with $\bbN$ interpreted as a~discrete category with
objects called {\em genera\/} in this context.

\begin{definition}
\begin{subequations}
\label{modular}
A {\em modular operad\/} is a  modular module
\begin{equation*}
\oM = \big\{\oM(S;g) \in \Chain\  \vrt \  (S;g)   \in  \fSet \times \bbA \big\}
\end{equation*}
together with degree $0$ morphisms (composition laws)
\begin{equation}
\label{v_Galway1}
\ooo ab:\oM\big(S_1 \sqcup \stt a;g_1\big)
\otimes \oM\big(S_2\sqcup \stt b;g_2\big)  
\to \oM ( S_1\sqcup S_2;g_1\+g_2)
\end{equation}
given for arbitrary finite 
sets $S_1$, $S_2$, elements $a,b$, and 
genera $g_1,g_2 \in \bbA$.  There are, moreover, 
degree $0$ contractions
\begin{equation}
\label{Galway}
\xxi_{uv} = \xxi_{vu} : \oM\bl S \sqcup \stt {u} \sqcup \stt {v} ;g\br  
\to \oM(S ;g\+1) 
\end{equation}
given for any finite
set $S$, genus $g \in \bbA$, and elements $u,v$.
These data are required to satisfy the following axioms.
\begin{enumerate}
\itemindent -1em
\itemsep .3em 
\item  [(i)]
For arbitrary isomorphisms $\rho : S_1 \sqcup \stt a \to  T_1$
  and $\sigma :  S_2 \sqcup \stt b \to  T_2$ of finite  
 sets and genera $g_1$, $g_2 \in \bbA$, one has the equality
\begin{equation}
\label{piji_caj_z_Jarcina_hrnku}
\oM\bl\rho|_{S_1}\sqcup\sigma|_{S_2}\br 
\ooo ab =
\ooo{\rho(a)}{\sigma(b)} \ \big(\oM(\rho)\ot\oM(\sigma)\big)
\end{equation}
of maps 
\[
\oM\bl S_1 \sqcup \stt a ;g_1\br \otimes \oM \bl S_2 \sqcup \stt b
;g_2\br 
\to
\oM\bl T_1\sqcup T_2
\setminus  \{\rho(a),\sigma(b)\};g_1+g_2\br.
\]

\item [(ii)] 
For an isomorphism $\rho : S \sqcup \stt u  \sqcup \stt v \to T
  $ of finite sets and a genus $g \in \bbA$, one has the
  equality
\begin{equation}
\label{eq:27}
\oM\bl\rho|_S\br \ \xxi_{uv} = 
\xxi_{\rho(u)\rho(v)}\oM(\rho)
\end{equation}
of maps $\oM \bl S \sqcup \stt u \sqcup \stt v;g \br \to \oM\bl T \setminus
\{\rho(u),\rho(v)\} ;g\!+\! 1\br$.

\item [(iii)]
For $S_1$, $S_2$, $a$, $b$ and $g_1$, $g_2$ as in~(\ref{v_Galway1}),
one has the equality
\begin{equation}
\label{eq:24}
\ooo{a}{b} =  \ooo{b}{a} \tau
\end{equation}
of maps $\oM(S_1 \sqcup \stt a;g_1)\otimes \oM(S_2 \sqcup \stt b;g_2)
\to \oM\bl S_1\sqcup S_2;g_1\!+\!g_2\br$; here $\tau$ denotes
  the commutativity constraint in $\Vect$.

\item  [(iv)]
For finite sets $S_1,S_2,S_3$,  elements
$a,b,c,d$ and genera $g_1,g_2,g_3 \in \bbA$,   one has the equality
\begin{equation}
\label{Dnes_s_Jaruskou_k_Pakouskum}
\ooo ab (\id \ot \ooo cd)  = \ooo cd (\ooo ab \ot
\id)
\end{equation} 
of maps 
$$
\oM \bl S_1 \sqcup \stt a;g_1 \br \ot \oM\bl S_2  \sqcup \stt
{b}\sqcup \stt {c};g_2\br 
\ot  \oM\bl S_3 \sqcup \stt d;g_3\br \longrightarrow 
\oM\bl S_1 \sqcup S_2 \sqcup S_3;
g_1\!+\!
g_2\! +\! g_3\br .
$$

\item  [(v)]
For a finite  set $S$, elements $a,b,c,d$ and a genus $g \in \bbA$
one has the equality
\begin{equation}
\label{eq:33}
\xxi_{ab} \ \xxi_{cd} = \xxi_{cd} \ \xxi_{ab}
\end{equation}
of maps $\oM\bl S \sqcup \stt a\sqcup \stt b \sqcup \stt c
\sqcup \stt d;g\br 
\to \oM (S;g+2)$.

\item [(vi)] 
For finite sets $S_1, S_2$,   elements $a,b,c,d$ and
genera $g_1,g_2   \in \bbA$, one has the equality
\begin{equation}
\label{eq:38}
\xxi_{ab} \ \ooo{c}{d} = \xxi_{cd} \ \ooo{a}{b}
\end{equation}
of maps $\oM\bl S_1 
\sqcup \stt a \sqcup \stt c;g_1\br \ot \oM \bl S_2 \sqcup \stt b
\sqcup \stt d;g_2\br
\to \oM( S_1 \sqcup S_2;g_1+g_2+1)$.

\item [(vii)] 
For  finite sets $S_1, S_2$,  elements
$a,b,u,v$, and genera  $g_1,g_2 \in \bbA$, one has the equality
\begin{equation}
\label{eq:39}
\ooo{a}{b} \ (\xxi_{uv}\ot\id) = \xxi_{uv} \ \ooo{a}{b}
\end{equation}
of maps $\oM\bl
S_1 \sqcup \stt a \sqcup \stt  u  \sqcup \stt v;g_1\br \ot \oM \bl S_2 \sqcup 
\stt b;g_2\br \to \oM (S_1 \sqcup S_2 ;g_1+g_2+1)$.
\end{enumerate} 
\end{subequations}
\end{definition}

\def\op{operad}
\def\oM{{\EuScript O}}

\begin{subequations}
\begin{definition}
\label{odd_modular}
An {\em odd modular {\op}\/} is  a modular module
\[
\oM = \big\{\oM(S;g) \in \Vect\  \vrt \  (S;g)   \in  \fSet \times \bbA \big\}
\]
together with degree $+1$ morphisms ($\twooo ab$-operations)
\begin{equation}
\label{v_GGalway}
\twooo ab:\oM\big(S_1 \sqcup \stt a;g_1\big)
\otimes \oM\big(S_2\sqcup \stt b;g_2\big)  
\to \oM ( S_1\sqcup S_2;g_1+g_2)
\end{equation}
given for arbitrary finite 
sets $S_1$, $S_2$, elements $a,b$, and 
arbitrary $g_1,g_2 \in \bbA$.  There are, moreover, 
degree $1$ morphisms (the contractions)
\[
\twxxi_{uv} = \twxxi_{vu} : \oM\bl S \sqcup \stt u  \sqcup \stt v ;g\br  
\to \oM(S ;g+1)
\]
given for any finite
set $S$,  $g \in \bbA$, and elements $u,v$; we are using the 
notation for composition laws of odd modular operads introduced in
\cite{KWZ}.
These data are required to satisfy the following axioms.
\begin{enumerate}
\itemindent -1em
\itemsep .3em 
\item  [(i)]
For arbitrary isomorphisms $\rho : S_1 \sqcup \stt a \to  T_1$
  and $\sigma :  S_2 \sqcup \stt b \to  T_2$ of finite  
 sets and  $g_1$, $g_2 \in \bbA$, one has the equality
\begin{equation}
\label{neni}
\oM\bl\rho|_{S_1}\sqcup\sigma|_{S_2}\br 
\twooo ab =
\twooo{\rho(a)}{\sigma(b)} \ \big(\oM(\rho)\ot\oM(\sigma)\big)
\end{equation}
of maps 
\[
\oM\bl S_1 \sqcup \stt a ;g_1\br \otimes \oM \bl S_2 \sqcup \stt b
;g_2\br 
\to
\oM\bl T_1\sqcup T_2
\setminus  \{\rho(a),\sigma(b)\};g_1+g_2\br.
\]

\item [(ii)] 
For an isomorphism $\rho : S \sqcup \stt u  \sqcup \stt v \to T
  $ of finite sets and  $g \in \bbA$, one has the
  equality
\[
\oM\bl\rho|_S\br \ \twxxi_{uv} = 
\twxxi_{\rho(u)\rho(v)}\oM(\rho)
\]
of maps $\oM \bl S \sqcup \stt u \sqcup \stt v;g \br \to \oM\bl T \setminus
\{\rho(u),\rho(v)\} ;g+1\br$.

\item [(iii)]
For $S_1$, $S_2$, $a$, $b$ and $g_1$, $g_2$ as in~(\ref{v_GGalway}),
one has the equality
\begin{equation}
\label{kdy_zacnu_byt_rozumny}
\twooo{a}{b} =  \twooo{b}{a} \tau
\end{equation}
of maps $\oM(S_1 \sqcup \stt a;g_1)\otimes \oM(S_2 \sqcup \stt b;g_2)
\to \oM\bl S_1\sqcup S_2;g_1+g_2\br$.

\item  [(iv)]
For   finite sets
  $S_1,S_2,S_3$, elements  
$a, b,c, d$ and $g_1,g_2,g_3 \in \bbA$,   one has the equality
\begin{equation}
\label{neni1}
\twooo ab (\id \ot \twooo cd)  = -\twooo cd (\twooo ab \ot
\id)
\end{equation}
of maps
$$
\oM \bl S_1 \sqcup \stt a;g_1 \br \ot \oM\bl S_2  \sqcup \stt b \sqcup
\stt c;g_2\br 
\ot  \oM\bl S_3 \sqcup \stt d;g_3\br \longrightarrow
\oM\bl S_1 \sqcup S_2 \sqcup S_3;
g_1\!+\!
g_2\! +\! g_3\br .
$$

\item  [(v)]
For a finite  set $S$, elements $a,b,c,d$ and  $g \in \bbA$
one has the equality
\[
\twxxi_{ab} \ \twxxi_{cd} =- \twxxi_{cd} \ \twxxi_{ab}
\]
of maps $\oM\bl S \sqcup \stt a  \sqcup \stt b  \sqcup \stt c \sqcup
\stt d;g\br 
\to \oM (S;g+2)$.

\item [(vi)] 
For finite sets $S_1, S_2$, elements $a,b,c,d$ and  $g_1,g_2
  \in \bbA$, one has the equality
\[
\twxxi_{ab} \ \twooo{c}{d} = -\twxxi_{cd} \ \twooo{a}{b}
\]
of maps $\oM\bl S_1 
\sqcup \stt a  \sqcup \stt c;g_1\br \ot \oM \bl S_2 \sqcup \stt b
\sqcup \stt d;g_2\br
\to \oM( S_1 \sqcup S_2;g_1+g_2+1)$.

\item [(vii)] 
For  finite sets $S_1, S_2$, elements
$a,b,u,v$, and $g_1,g_2 \in \bbA$, one has the equality
\[
\twooo{a}{b} \ (\twxxi_{uv}\ot\id) = -\twxxi_{uv} \ \twooo{a}{b}
\]
of maps $\oM\bl
S_1 \sqcup \stt a  \sqcup \stt u  \sqcup \stt v;g_1\br \ot \oM \bl S_2 \sqcup 
\stt b;g_2\br \to \oM (S_1 \sqcup S_2 ;g_1+g_2+1)$.
\end{enumerate} 
\end{definition}
\end{subequations}

\begin{remark}
Odd modular operads appeared
in~\cite[Section~4]{getzler-kapranov:CompM98} as {\em modular ${\mathfrak
    {K}}$-operads\/} for the dualizing cocycle ${\mathfrak {K}}$. The
terminology we use was suggested by Ralph Kaufmann. A~discussion of
odd modular operads and similar structures can be found e.g.\ in~\cite{markl:odd}.
\end{remark}

\def\calS{{\EuScript S}}
\begin{subequations}
\begin{definition}
\label{b3}
A (classical) {\em Markl operad\/} is a collection
$\calS =\{\calS(n)\}_{n\geq 0}$ of right $\bfk[\Sigma_n]$-modules,
together with $\bfk$-linear maps ($\circ_i$-compositions)
\begin{equation} 
\label{zitra_s_Mikesem_na_Jazz}
\circ_i : \calS(m) \ot \calS(n) \to \calS(m +n -1),
\end{equation}
for $1 \leq i \leq m$ and $n \geq 0$.
These data fulfill the following axioms.
\begin{enumerate}
\item[(i)]
For each $1 \leq j \leq a$, $b,c \geq 0$, $f\in
\calS(a)$, $g \in \calS(b)$ and $h \in \calS(c)$,
\begin{equation}
\label{zitra_s_Mikesem_do_Salmovske}
(f \circ_j g)\circ_i h =
\begin{cases}
(f \circ_i h) \circ_{j+c-1}g& \mbox{for } 1\leq i< j
\\
f \circ_j(g \circ_{i-j+1} h)& \mbox{for }
j\leq i~< b+j
\\
(f \circ_{i-b+1}h) \circ_j g& \mbox{for }
j+b\leq i\leq a+b-1.
\end{cases}
\end{equation}
\item[(ii)]
For each  $1 \leq i \leq m$, $n \geq 0$, $\tau \in \Sigma_m$ and
$\sigma \in \Sigma_n$,  let $\tau \circ_i
\sigma \in \Sigma_{m+n -1}$ be given by inserting the permutation
$\sigma$ at the $i$th place in $\tau$. Let $f \in \calS(m)$ and $g \in
\calS(n)$. Then
\begin{equation}
\label{Krtecek_na_mne_kouka.}
(f\tau)\circ_i(g\sigma) = (f\circ_{\tau(i)} g)(\tau \circ_i
\sigma).
\end{equation}
\end{enumerate}
\end{definition}
\end{subequations}

\begin{theindex}
{
\item Algebra over a Markl operad,
  Definitions~6.7 and~6.9 of~\cite{part1}
\item Canonical grading,~\cite[Section~2]{part1}
\item Cocycle, Example~6.10 of~\cite{part1}
\item Cobar construction, Definition~\ref{Za chvili zavolam Taxovi.}
\item Coboundary, Example~6.10 of~\cite{part1}
\item Constant-free operadic category,
    page~\pageref{Jak bude o vikendu jeste nevim.},  Subsection~\ref{Potkali se u Kolina.}
\item $\Coll(\ttO)$, category of $1$-connected $\ttO$-collections,
  Definition~\ref{Jarka_dnes_u_lekare} 
\item Contraction, Definition~3.3 of~\cite{part1} 
\item Derivation of a Markl operad, Definition~\ref{Uz mam 6,5 hodin
    na Tereji tento rok.} 
\item Discrete operadic fibration, Definition~4.1 of~\cite{part1}
\item Derived sequence, display~\eqref{e1}
\item Dual dg Markl operad, Definition~\ref{Nebrzdi mi kolo.}
\item Elementary morphism, Definition~\ref{plysacci_postacci}
\item Extended units, page~\pageref{Mam zdravotni zpusobilost na dalsi dva
    roky.}, Subsection~\ref{Jsem nachlazen?}
\item Factorizability, \Fac,
      page~\pageref{dnes_prednaska_na_Macquarie},  Subsection~\ref{Potkali se u Kolina.}
\item Free Markl operad, display~\eqref{Konci_Neuron?}
\item $F \Fib T\to t$, map to a local terminal object and its
  fiber, page~\pageref{Jsem treti na ``Potkali se u Kolina.''},  
Section~\ref{letam_205}
\item $F \fib T\to S$, elementary map and its
  fiber, page~\pageref{Jsme treti na ``Potkali se u Kolina.''},  
Subsection~\ref{Jsem nachlazen?}
\item Grading of an operadic category, page~\pageref{Za
  14 dni prvni davka.},  Subsection~\ref{Potkali se u Kolina.}
\item Harmonic pair, Definition~\ref{har}
\item Ideal in a Markl operad, Definition~\ref{Prestane Jarka pit?}
\item Invertibility of \qb{s}, \QBI, page~\pageref{Za chvili na
    kozni.},  Subsection~\ref{Potkali se u Kolina.}
\item Koszul dual, Definition~\ref{vice nez 10 000 pripadu denne}
\item Labeled tower, labeling, Definition~\ref{zitra_vedeni}
\item Local isomorphism, Definition~3.3 of~\cite{part1}
\item Local reordering morphism, Definition~3.3 of~\cite{part1}
\item $\lTw(X)$, groupoid of labeled towers under $X$,
  page~\pageref{vcera_jsem_bezel_36_minut}, Section~\ref{zitra_letim_do_Pragy}
\item Markl operad, Definition~\ref{markl}
\item Markl cooperad, Definition~\ref{Jarca byla na chalupe sama.}
\item Order-preserving morphism, Definition~3.3 of~\cite{part1}
\item Ordered graph, Definition~3.13 of~\cite{part1}
\item $\Iso \subset \ttO$, subcategory of isomorphisms, page~\pageref{Co
  zitra zjisti?}, Subsection~\ref{Jsem nachlazen?}
\item $\LT \subset \ttO$, subcategory of local terminal objects,
  page~\pageref{Jsem znepokojen.}, Subsection~\ref{Jsem nachlazen?}
\item $\DO\subset \ttO$, subcategory of order-preserving morphisms,
  page~\pageref{Zitra v 7 na koznim.}, Subsection~\ref{Potkali se u Kolina.}
\item $\QO\subset \ttO$, subcategory of \qb{s},
  page~\pageref{Internet porad nejde.}, Subsection~\ref{Potkali se u Kolina.}
\item $\VR$, groupoid of virtual isomorphisms,
  page~\pageref{Za_chvili_pojedu_na_schuzi_klubu.}, Section~\ref{letam_205}
\item  $\VR(e)\subset \VR$, subgroupoid of objects of positive grade,
  page~\pageref{Musim to dorvat.}, Section~\ref{letam_205}
\item Pairs with disjoint fibers, Definition~\ref{d3}
\item Preodered graph, Definition~3.1 of~\cite{part1}
\item Pure contraction, Definition~3.3 of~\cite{part1}
\item Quadratic Markl operad, Definition~\ref{zitra_seminar}
\item Quadratic data, Definition~\ref{zitra_seminar}
\item $\QV$, quotient of $\VI$ by virtual isomorphisms,
  page~\pageref{Je streda.}, Section~\ref{letam_205}
\item $\QV(e)$, quotient of $\VIe$ by virtual isomorphisms,
  page~\pageref{Musim to dorvat.}, Section~\ref{letam_205}
\item Rigidity, \Rig, 
      page~\pageref{Koleno se pomalu lepsi.},  Subsection~\ref{Potkali se u Kolina.}
\item Self-dual Markl operad, Definition~\ref{Uz_je_o_te_medaili_zapis.}
\item Strict grading, \SGrad, page~\pageref{Zadne
    viditelne kozni projevy nemam.},  Subsection~\ref{Potkali se u Kolina.}
\item Strictly unital Markl operad, Definition~\ref{svedeni}
\item Strictly counital Markl cooperad, Definition~\ref{Jak bude?}
\item Blow-up axiom, \BU, page~\pageref{bu}, Subsection~\ref{Potkali se u Kolina.}
\item Strongly factorizable operadic category, \SFac, 
      Definition~2.9 of~\cite{part1}
\item ${\tt SUMOp}^\ttV_1(\ttO)$, category of $1$-connected strictly
  unital Markl $\ttO$-operads, page~\pageref{Pomuze_vitamin_C?},
  Section~\ref{zitra_letim_do_Pragy} 
\item Tower, display~\eqref{t1-1}
\item Unital Markl operad, page~\pageref{Habaneros},
  Subsection~\ref{Jsem nachlazen?} 
\item Unique fiber axiom, \UFB, page~\pageref{Kveta_asi_spi.},
  Subsection~\ref{Potkali se u Kolina.}
\item Virtual (iso)morphism, page~\pageref{Za chvili sraz s Denisem.}, 
   Section~\ref{letam_205}
\item Weak blow-up axiom, \WBU, page~\pageref{wbu},
  Subsection~\ref{Potkali se u Kolina.} 
\item $1$-connected Markl operad, Definition~\ref{svedeni}
}
\end{theindex}



\providecommand{\doi}[1]{}
\renewcommand{\doi}[1]{\href{https://doi.org/\detokenize{#1}}{DOI: \detokenize{#1}}}%
\newcommand{\arxiv}[1]{\href{http://arxiv.org/abs/#1}{arXiv:#1}}

\providecommand{\fulldoi}[1]{}
\renewcommand{\fulldoi}[1]{\href{\detokenize{#1}}{\detokenize{#1}}}%
\renewcommand{\fulldoi}[1]{\url{\detokenize{#1}}}%

\label{lastpage}

\end{document}